\documentclass{amsart}

\usepackage{amsfonts,amsmath,amsthm,amssymb,amscd,latexsym,enumerate,enumitem, toolbox,mathrsfs,graphicx, soul, hyperref}
\usepackage{marginnote}

\usepackage[all,cmtip]{xy}

\usepackage{xcolor}
\usepackage{cleveref}
\allowdisplaybreaks

\newcommand{\E}{\mathcal{E}}
\newcommand{\F}{\mathcal{F}}
\newcommand{\Z}{{\mathbb Z}}

\renewcommand{\phi}{\varphi}
\newcommand{\Ell}{\mathrm{Ell}}

\newcommand{\into}{{\hookrightarrow}}

\newcommand{\cpalgOfE}{{\mathcal{O}(\E)}} 
\newcommand{\cpalgOfEY}{{\mathcal{O}(\E_Y)}} 
\newcommand{\setYn}[1]{\alpha^{#1}(Y)} 
\newcommand{\setVn}[1]{\alpha^{#1}(V)}
\newcommand{\setWn}[1]{\alpha^{#1}(W)}

\newcommand{\restrict}[2]{{\left.{#1}\right|_{#2}}}
\newcommand{\cstar}{$\mathrm{C}^*$}

\theoremstyle{plain}
    \newtheorem{theorem}{Theorem}[section]
    \newtheorem{lemma}[theorem]{Lemma}
    \newtheorem{corollary}[theorem]{Corollary}
    \newtheorem{proposition}[theorem]{Proposition}

\theoremstyle{definition}
    \newtheorem{definition}[theorem]{Definition}
    \newtheorem{example}[theorem]{Example}
      \newtheorem{examples}[theorem]{Examples}

    \newtheorem{remark}[theorem]{Remark}
    
\theoremstyle{remark}

\DeclareMathOperator{\id}{id}
\DeclareMathOperator{\Aut}{Aut}
\DeclareMathOperator{\supp}{supp}

\DeclareMathOperator{\Span}{span}
\DeclareMathOperator{\orb}{orb}

\author[Adamo, Archey, Forough, Georgescu, Jeong, Strung, Viola]{Maria Stella Adamo \and Dawn E. Archey \and  Marzieh Forough \and Magdalena C. Georgescu \and Ja A Jeong \and Karen R. Strung \and Maria Grazia Viola}

\address{Department of Mathematical Sciences \\
The University of Tokyo \\
3-8-1 Komaba, Tokyo, 153-8914, Japan} 
\email{adamoms@ms.u-tokyo.ac.jp}

\address{Department of Mathematics\\
University of Detroit Mercy\\
4001 W. McNichols Road\\
Detroit, MI, USA 48221-3038 }
\email{archeyde@udmercy.edu}

\address{Department of Applied Mathematics, Faculty of Information Technology\\ Czech Technical University in Prague\\ Th\'akurova 9 \\160 00, Prague 6, Czech Republic, and Department of Abstract Analysis\\ Institute of Mathematics, Czech Academy of Sciences, \v{Z}itn\'a 25, 115 67 Prague 1, Czech Republic}
\email{foroumar@fit.cvut.cz}

\address{Toronto, Ontario, Canada} 
\email{magda@uvic.ca}

\address{Department of Mathematical Sciences and Research Institute of Mathematics\\
Seoul National University\\
Seoul, Korea 08826}
\email{jajeong@snu.ac.kr}

\address{Department of Abstract Analysis\\ Institute of Mathematics, Czech Academy of Sciences, \v{Z}itn\'a 25, 115 67 Prague 1, Czech Republic}
\email{strung@math.cas.cz}

\address{Lakehead University\\
Orillia ON L3V 0B9\\
Canada\\ and Fields Institute \\ 222 College Street \\ Toronto, ON M5T 3J1, Canada}
\email{mviola@lakeheadu.ca}

\date{\today}
\thanks{KRS is currently funded by GA\v{C}R project 20-17488Y and \mbox{RVO: 67985840} and part of this work was carried out while funded by a Radboud Excellence Initiative Postdoctoral Fellowship. }
\thanks{JAJ was partially supported by NRF 2018R1D1A1B07041172. }
\thanks{MF was supported by GA\v{C}R project 19-05271Y, RVO:67985840.}
\thanks{MGV was supported by an NSERC Discovery Grant.}
\thanks{MSA was supported by ERC Advanced Grant no. 669240 QUEST “Quantum Algebraic Structures and Models”. MSA acknowledges the Gruppo Nazionale per l'Analisi Matematica, la Probabilit\`{a} e le loro applicationi (GNAMPA) of INdAM. Part of this work was carried out while MSA was funded by an Oberwolfach Leibniz Fellowship in 2020 and in 2021, and supported by the University of Rome ``Tor Vergata'' funding scheme ``Beyond Borders'' CUP E84I19002200005. MSA is currently a JSPS International Research Fellow supported by the Grant-in-Aid Kakenhi n. 22F21312.}

\begin{document}

\title[$\mathrm{C}^*$-algebras of homeomorphisms twisted by vector bundles]{$\mathrm{C}^*$-algebras associated to homeomorphisms twisted by vector bundles over finite dimensional spaces}

\begin{abstract} In this paper we study Cuntz--Pimsner algebras associated to \mbox{$\mathrm{C}^*$-correspondences} over commutative $\mathrm{C}^*$-algebras from the point of view of the $\mathrm{C}^*$-algebra classification programme. We show that when the correspondence comes from an aperiodic homeomorphism of a finite dimensional infinite compact metric space $X$ twisted by a vector bundle, the resulting Cuntz--Pimsner algebras have finite nuclear dimension. When the homeomorphism is minimal, this entails  classification of these $\mathrm{C}^*$-algebras by the Elliott invariant. This establishes a dichotomy: when the vector bundle has rank one, the Cuntz--Pimsner algebra has stable rank one. Otherwise, it is purely infinite. 

For a Cuntz--Pimsner algebra of a minimal homeomorphism of an infinite compact metric space $X$ twisted by a line bundle over $X$, we introduce orbit-breaking subalgebras. With no assumptions on the dimension of $X$, we show that they are centrally large subalgebras and hence simple and stably finite. When the dimension of $X$ is finite, they are furthermore $\mathcal{Z}$-stable and hence classified by the Elliott invariant.
\end{abstract}

\maketitle
\tableofcontents

\section{Introduction}

One of the most pleasing aspects of $\mathrm{C}^*$-algebra theory is the ability to study various mathematical objects using a $\mathrm{C}^*$-algebraic toolkit. Well-known examples include studying dynamical systems via the crossed product construction, directed graphs and shifts of finite type via Cuntz--Krieger algebras and more general graph $\mathrm{C}^*$-algebras, topological groups via their group $\mathrm{C}^*$-algebras, and many more. In \cite{Pimsner1997}, Pimsner generalised both crossed products by the integers and Cuntz--Krieger algebras via a single construction which we now call \emph{Cuntz--Pimsner algebras}. 

In the Cuntz--Pimsner construction, the underlying mathematical object is a \mbox{$\mathrm{C}^*$-correspondence}. If $A$ and $B$ are two $\mathrm{C}^*$-algebras then an $A$-$B$ $\mathrm{C}^*$-correspondence is an $A$-$B$-module with some extra structure (see \Cref{defn.Corr}). It can be thought of as a generalised homomorphism between $A$ and $B$, particularly if we are only interested in the Morita equivalence (or stable isomorphism) classes of $\mathrm{C}^*$-algebras, rather than isomorphism classes. Given an $A$-$A$ $\mathrm{C}^*$-correspondence---which for brevity we call a $\mathrm{C}^*$-correspondence over $A$---and viewing it as a generalised endomorphism, we can think of a Cuntz--Pimsner algebra as a type of ``crossed product'' by $\mathbb{N}$ or $\mathbb{Z}$. 

The Cuntz--Pimsner construction was later made more general in the work of Katsura (see for example, \cite{Katsura2003}), where fewer restrictions were required on the \mbox{$\mathrm{C}^*$-correspondences} under consideration. Around the same time, the concept of a $\mathbb{Z}$-crossed product of a $\mathrm{C}^*$-algebra $A$ by a Hilbert $A$-bimodule (which can be seen as a type of $\mathrm{C}^*$-correspondence) was introduced by Abadie, Eilers and Exel (see Definition 2.4 in \cite{AEE:Cross}). It turns out that if we view a Hilbert $A$-bimodule $\E$ as a $\mathrm{C}^*$-correspondence, Katsura's construction gives a $\mathrm{C}^*$-algebra isomorphic to the crossed product $A \rtimes_\E \mathbb{Z}$ of \cite{AEE:Cross}, see \cite[Proposition 3.7]{Katsura2003}.

In this paper, we study Cuntz--Pimsner algebras coming from $\mathrm{C}^*$-correspondences over commutative $\mathrm{C}^*$-algebras.  Viewing a $\mathrm{C}^*$-correspondence as a generalised homomorphism, it is not surprising that these Cuntz--Pimsner algebras have a strong dynamical flavour and, especially in the case of Hilbert $C(X)$-bimodules, bear many similarities to crossed products by homeomorphisms. Topological dynamical systems of the form $(X, \alpha)$, where $X$ is a compact metric space and $\alpha$ is a homeomorphism, have long been a source of interesting examples in the theory of $\mathrm{C}^*$-algebras. In particular, when $X$ is an infinite compact metric space and $\alpha$ is minimal, which is to say, there are no non-empty proper closed $\alpha$-invariant subsets of $X$, the crossed products $C(X) \rtimes_\alpha \mathbb{Z}$ have received special attention from Elliott's classification programme for $\mathrm{C}^*$-algebras.  Here we investigate the Cuntz--Pimsner algebras of $C(X)$-correspondences through the lens of the classification programme, further contributing to this long line of research. 

Elliott's classification programme was initiated by George Elliott in the 1990s.  In 1976, he showed that approximately finite dimensional (AF) algebras are classified, up to $^*$-isomorphism, by their scaled ordered $K_0$-groups \cite{Ell:AF}. After further classification successes in the case of simple and nuclear $\mathrm{C}^*$-algebras,  Elliott conjectured that all simple, separable, unital and nuclear $\mathrm{C}^*$-algebras  could be classified, up to $^*$-isomorphism, by an invariant consisting of $K$-theory and tracial information. For a unital $\mathrm{C}^*$-algebra $A$, the invariant $\Ell(A)$, now called the \emph{Elliott invariant}, is given by 
 \[ \Ell(A) = (K_0(A), K_0(A)_+, [1_A], K_1(A), T(A), \rho),\]
 where $(K_0(A), K_0(A)_+, [1_A], K_1(A))$ is the (pointed, ordered) $K$-theory of $A$, $T(A)$ the tracial state simplex, and $\rho : K_0(A) \times T(A) \to \mathbb{R}$ a pairing map defined by $\rho([p]-[q], \tau) = \tau(p) - \tau(q)$, for $\tau$ the (non-normalised) inflation of a tracial state to a suitable matrix algebra over $A$. 
 
When pathological counterexamples to Elliott's original conjecture were found (see for example \cite{Villadsen1998, Villadsen1999, Rordam2003, Toms2005}), it became clear that the class of simple, separable, unital, nuclear $\mathrm{C}^*$-algebras was too broad and would have to be further restricted by imposing certain regularity conditions. In particular, the construction of the Jiang--Su algebra $\mathcal{Z}$ meant that any hope of classification would necessitate restricting to those simple, separable, unital, nuclear $\mathrm{C}^*$-algebras which are $\mathcal{Z}$-stable. The Jiang--Su algebra is a simple, separable, unital, nuclear, infinite dimensional $\mathrm{C}^*$-algebra with the same Elliott invariant as $\mathbb{C}$ \cite{JS1999}. Since the Elliott invariant behaves well under tensor products, this means that (under a mild condition on the order structure of the $K_0$-group), if $A$ is a simple, separable, unital, nuclear $\mathrm{C}^*$-algebra then $\Ell(A) \cong \Ell(A\otimes \mathcal{Z})$. We say that $A$ is $\mathcal{Z}$-stable if $A \cong A \otimes \mathcal{Z}$.

Upon restricting to those simple, separable, unital, nuclear, $\mathcal{Z}$-stable $\mathrm{C}^*$-algebras which satisfy the Universal Coefficient Theorem (UCT), the classification programme has been a resounding success: the class of all such $\mathrm{C}^*$-algebras can be classified, up to $^*$-isomorphism, by Elliott invariants. See \Cref{ClassThm} for a precise statement. The final pieces of the classification theorem have only been put into place within the last decade, stemming from the work of many mathematicians. We highlight in particular Elliott, Kirchberg, Phillips, Winter, Tikuisis, White, Toms, Lin, Niu, and Gong. An overview of the most recent results can be found, for example, in \cite{Win:ICM, TWW, Str:book}. The UCT condition is a rather mild assumption, as every known nuclear $\mathrm{C}^*$-algebra satisfies the UCT.  In particular, the $\mathrm{C}^*$-algebras in this paper all satisfy the UCT. However, it remains an open problem to determine whether nuclearity implies the UCT.

For simple, separable, unital, nuclear, infinite dimensional $\mathrm{C}^*$-algebras, the seemingly mysterious condition of $\mathcal{Z}$-stability of a given $\mathrm{C}^*$-algebra turns out to be equivalent to its \emph{nuclear dimension} being finite \cite{Win:Z-stabNucDim, CETWW}. The nuclear dimension, introduced by Winter and Zacharias \cite{WinterZac:dimnuc} based on earlier work by Winter \cite{Win:cpr, Win:cpr2} and Kirchberg and Winter \cite{KirWinter:dr}, is a noncommutative generalisation of covering dimension in the sense that $\dim_{\mathrm{nuc}} C(X) = \dim X$ for a locally compact Hausdorff space $X$ \cite{WinterZac:dimnuc}. Rephrasing $\mathcal{Z}$-stability in terms of finite nuclear dimension is quite satisfying, since the fact that non-finite dimensional objects might be pathological is not so surprising. Furthermore, even for a noncommutative $\mathrm{C}^*$-algebra, the nuclear dimension can often be estimated by a suitable notion of dimension for the underlying mathematical structure. Of importance for this paper is the notion of Rokhlin dimension for a dynamical system \cite{HirWinZac:RokDim}, and more generally, Rokhlin dimension of a $\mathrm{C}^*$-correspondence \cite{MR3845113}. These allow us to establish, when $\dim X< \infty$,  the classification of the Cuntz--Pimsner algebras of $C(X)$-correspondences which are finitely generated projective as right Hilbert modules,  see~\Cref{class}.

We also obtain a classification for Cuntz--Pimsner algebras of certain $C(X)$-correspondences which are not necessarily full: starting with a full Hilbert $C(X)$-bimodule, we restrict the left action to a proper ideal of $C(X)$.  The resulting Cuntz--Pimsner algebra is a $\mathrm{C}^*$-subalgebra of the Cuntz--Pimsner algebra of the original full bimodule. We call such Cuntz--Pimsner algebras \emph{orbit-breaking} subalgebras. The orbit-breaking construction in the case of a minimal dynamical system $(X, \alpha)$ is originally due to Putnam, who constructed AF algebras as orbit-breaking subalgebras of crossed products of Cantor minimal systems \cite{Putnam:MinHomCantor}. Every simple unital AF algebra can be obtained by such an orbit-breaking construction, while $K$-theoretic obstructions prevent a crossed product by a minimal homeomorphism from being AF. Orbit-breaking constructions have furthermore been used to construct the Jiang--Su algebra (also not obtainable as a crossed product by a minimal homeomorphism) \cite{DPS:JiangSu}, as well as simple, separable, unital, nuclear, $\mathcal{Z}$-stable $\mathrm{C}^*$-algebras with a wide range of Elliott invariants \cite{DPS:OrbitBreaking}.  An interesting question is whether such dynamical constructions can exhaust the range of the Elliott invariant. In related work, Li has constructed twisted principal \'etale groupoid $\mathrm{C}^*$-algebras with any prescribed Elliott invariant \cite{Li:ClassifiableCartan}. There the groupoids are constructed by mimicking known inductive limit structures rather than groupoids coming from topological dynamical systems. To reach all possible invariants, Li also requires twists over such groupoids. The Cuntz--Pimsner algebras considered in Sections~\ref{sec.cp} and \ref{sec.large} turn out to be isomorphic to twisted groupoid $\mathrm{C}^*$-algebras, where the twist is over either a transformation groupoid $X \times \mathbb{Z}$ or an orbit-breaking groupoid in the sense of \cite{Put:K-theoryGroupoids, DPS:OrbitBreaking}. In the case that $\dim X < \infty$, these Cuntz--Pimsner algebras are also classified by the Elliott invariant, see \Cref{class2}. The proof relies on the notion of a large subalgebra due to Phillips \cite{PhLarge}. In particular, the orbit-breaking algebras of Section~\ref{sec.large} make the possibility of exhausting the Elliott invariant through dynamical constructions look increasingly plausible. The groupoid picture will be further developed in subsequent work by Forough and Strung \cite{AdFoSt:Cartan}. Orbit-breaking subalgebras will also play a key role in \cite{FoJeSt:rsh}, where it is shown they have an approximately subhomogeneous structure allowing us to push our classification results beyond the case that $\dim X < \infty$.

\subsection*{Summary of the paper} The paper is structured as follows. In Section~\ref{sec.prelim} we establish the necessary preliminaries by recalling the definition for Hilbert modules, $\mathrm{C}^*$-correspondences, and Hilbert bimodules. We then define Cuntz--Pimsner algebras via covariant representations. 

In Section~\ref{sec.Corr} we specialise to $\mathrm{C}^*$-correspondences over $\mathrm{C}^*$-algebras which are unital and commutative, and set up notation for the correspondences of interest in this paper, those of the form $\Gamma(\mathscr{V}, \alpha)$. These are $C(X)$-correspondences which, as right modules have the structure of a module of sections of a vector bundle $\mathscr{V}$ over a compact metric space $X$, and whose left action is given by composing a function \mbox{$f \in C(X)$} with a homeomorphism $\alpha : X \to X$. We determine when the corresponding Cuntz--Pimsner algebras are simple, see \textbf{\Cref{{19-10-18-3}}}. 
Section~\ref{sec.cp} takes a more detailed look at the $\mathrm{C}^*$-algebras associated to the particular $C(X)$-correspondences of the previous section. We establish that they are stably finite if and only if $\mathscr{V}$ is a line bundle, and determine their tracial state spaces, \textbf{Propositions \ref{prop:traces} and \ref{prop.T(A)}.}

Our first classification theorem appears in Section~\ref{sec.RD}. Using the Rokhlin dimension for $\mathrm{C}^*$-correspondences defined by Brown, Tikuisis and Zelenberg \cite{MR3845113}, we show that the $C(X)$-correspondence $\Gamma(\mathscr{V}, \alpha)$, when $\dim X < \infty$ and $\alpha$ is aperiodic, has finite Rokhlin dimension, \textbf{\Cref{thm.FinRokDim}}, and hence has finite nuclear dimension, \textbf{\Cref{fin_nuc_dim}}. We conclude that $\mathrm{C}^*$-algebras constructed from such correspondences are classified by the Elliott invariant, \textbf{\Cref{class}}. Moreover, this establishes a dichotomy: when $\mathscr{V}$ is a line bundle, $\mathcal{O}(\Gamma(\mathscr{V}, \alpha))$ has stable rank one; otherwise it is purely infinite, \textbf{\Cref{dichotomy}}. 

Section~\ref{sec.large} is  the most technical part of the paper. We restrict to the case of Hilbert bimodules. More explicitly, this means that for $\Gamma(\mathscr{V}, \alpha)$ we assume that $\mathscr{V}$ is a line bundle over $X$. We introduce orbit-breaking subalgebras of $\mathcal{O}(\Gamma(\mathscr{V},\alpha))$, which are also Cuntz--Pimsner algebras by $C(X)$-correspondences. However, the Hilbert bimodules are not finitely generated projective so we cannot apply the classification results of the previous section. For this reason, the main goal of Section~\ref{sec.large} is to show that the orbit-breaking subalgebras are \emph{centrally large subalgebras} (see \Cref{defn.largesubalgebra}), in the sense of Archey and Phillips~\cite{ArchPhil:SR1}. Centrally large subalgebras share many properties with their containing $\mathrm{C}^*$-algebras, most notably, $\mathcal{Z}$-stability.
This allows us to extend the nuclear dimension result of Section~\ref{sec.RD}, see \textbf{\Cref{cor.OBNucdim}},
as well as the classification results to include the orbit-breaking subalgebras, \textbf{\Cref{class2}}.

Finally, in Section~\ref{sec.examples} we discuss certain examples as a way of showing the richness of this class of $\mathrm{C}^*$-algebras. Special attention is given to quantum Heisenberg manifolds, which can be realised as Cuntz--Pimsner algebras of Hilbert $C(\mathbb{T}^2)$-bimodules. 

\subsubsection*{Acknowledgments} The ideas for this paper stemmed from the Women in Operator Algebras workshop, which took place at the Banff International Research Station in December 2018. The authors express their thanks to BIRS for hosting them on that occasion. The collaboration further benefited from visits to Radboud University and the University of Münster in spring 2019, as well as Seoul National University in autumn 2019. MSA is grateful to the Institute of Mathematics of the Czech Academy of Sciences for her visit to MF and KRS in summer 2020, and to the ITP of the University of Leipzig, where part of this work was carried out in Spring/Summer 2021.  DA thanks the Fields Institute for hosting her in August 2019 to visit MCG and MGV. MCG thanks the University of Detroit Mercy for hosting her in March of 2019 for a visit to DA.  KRS thanks the Fields Institute for hosting her in January 2020 to visit MCG and MGV, as well as the University of Rome, Tor Vergata, for hosting her in February 2020 to visit MSA. The authors would also like to thank the referee for a very careful reading and suggestions which improved the paper, as well as Aaron Kettner who pointed out a missing assumption in the discussion following Lemma~\ref{lem.SectionsOfEy}.

\section{Preliminaries} \label{sec.prelim}

In this section we define Hilbert modules, $\mathrm{C}^*$-correspondences, and Hilbert bimodules. The reader should be aware that the terminology for $\mathrm{C}^*$-correspondences and Hilbert bimodules is not consistent in the literature and they often appear under different names.

\subsection{Hilbert modules}
Here we recall some elementary facts about Hilbert modules. Further details can be found in \cite{Lan:modules}, as well as \cite[Chapter 15]{Weg:k-theory} and \cite[Chapters 2 and 3]{RaeWil:morita}.

We will be interested in Hilbert modules over unital $\mathrm{C}^*$-algebras which are finitely generated projective, or equivalently, those which are algebraically finitely generated \cite[Corollary 15.4.8]{Weg:k-theory}. Let $A$ be a unital $\mathrm{C}^*$-algebra. A Hilbert $A$-module $\E$ is algebraically finitely generated if there exist $k \in \mathbb{N}$ and $\xi_1, \dots, \xi_k \in \E$ such that every $\xi \in \E$ can be written as $\xi = \sum_{i=1}^k \xi_i a_i$ for some $a_i \in A$, $1 \leq i \leq k$. It is finitely generated projective if there exists $n \in \mathbb{Z}_{\geq 0}$ such that it is a direct summand in the  Hilbert $A$-module $A^{\oplus n}$. The dual $\E^{\dagger}$ of a right Hilbert $A$-module $\E$ is the left $A$-module of bounded $A$-linear maps $\varphi : \E \to A$, with left action $(a \phi)(\xi) = a \phi(\xi)$. Define 
\[\hat{\xi}(\eta) = \langle \xi, \eta \rangle_\E, \quad \eta \in \E.\]
Then $\hat{\xi} \in \E^{\dagger}$. We say that $\E$ is \emph{self-dual} if $\E^{\dagger} = \widehat{\E} := \{ \hat{\xi} \mid \xi \in \E\}$, or, equivalently, if we consider the induced right module structure $(\varphi a)(\xi) = a^*(\varphi(\xi))$,  then $\E^{\dagger} \cong \E$ as right Hilbert $A$-modules. Every finitely generated projective Hilbert $A$-module is self-dual, see the discussion following Theorem 1.3 in \cite{Miscenko1979}. 

We denote by $\mathcal{L}(\E)$ the $\mathrm{C}^*$-algebra of adjointable operators on $\E$, and by $\mathcal{K}(\E)$ the ideal of compact operators. For any $\xi, \eta \in \mathcal{E}$, we denote the associated rank one operator by $\theta_{\xi, \eta}$, that is, the adjointable operator given by
\begin{equation} \theta_{\xi, \eta}(\zeta) = \xi \langle \eta, \zeta \rangle_\E, \qquad \zeta \in \mathcal{E}. \label{eqn.rankoneop}
\end{equation}
If $A$ is unital and $\mathcal{E}$ is a finitely generated and projective Hilbert $A$-module then $\mathcal{K}(\mathcal{E}) = \mathcal{L}(\mathcal{E})$, since $\mathcal{K}(\E)$ is unital \cite[Remarks 15.4.3]{Weg:k-theory}. 

\subsection{\texorpdfstring{$\mathrm{C}^*$}{C*}-correspondences and Hilbert bimodules} \label{subsec.CorBimod}
\begin{definition} \label{defn.Corr}
Let $A$ and $B$ be $\mathrm{C}^*$-algebras. An $A$-$B$ \emph{$\mathrm{C}^*$-correspondence} is a right Hilbert $B$-module $\mathcal{E}$ equipped with a $^*$-homomorphism, 
\[ \varphi_\E : A \to \mathcal{L}(\mathcal{E}),\]
called the \emph{structure map}. If $A = B$ then we call $\E$ a \emph{$\mathrm{C}^*$-correspondence over $A$}. 
\end{definition}

The map $\varphi_\E$ gives $\mathcal{E}$ a left $A$-module structure, but not necessarily a left Hilbert $A$-module structure, as there need not be a left $A$-valued inner product on $\E$. 

We also have the related notion of a \emph{Hilbert $A$-$B$-bimodule}, which is a right Hilbert $B$-module that also has the structure of a left Hilbert $A$-module, and for which the inner products satisfy a certain compatibility condition.

\begin{definition}\label{defn.bimod} Let $\mathcal{E}$ be a right Hilbert $B$-module  and left Hilbert $A$-module. We say that $\mathcal{E}$ is an $A$-$B$ \emph{Hilbert bimodule} if 
\[ \xi \langle \eta, \zeta \rangle_\E = _\E\!\!\langle \xi, \eta \rangle \zeta, \qquad \xi, \eta, \zeta \in \mathcal{E},\]
where $\langle \cdot, \cdot \rangle_\E$ denotes the right inner product taking values in $B$ and  $_\E \langle \cdot, \cdot \rangle$ the left inner product taking values in $A$. If $A = B$ then we call $\E$ a \emph{Hilbert $A$-bimodule}.
\end{definition}

Given a $\mathrm{C}^*$-correspondence $\mathcal{E}$ over $A$, denote by 
\begin{equation*}
\langle\E,\E\rangle_\E=\overline{\textup{span}}\{ \langle \xi, \eta \rangle_\E \mid \xi, \eta \in \mathcal{E} \},
\end{equation*} 
 where $\overline{\Span}(S)$ of a set $S$ denotes the closure of the linear span of $S$. 

A $\mathrm{C}^*$-correspondence $\mathcal{E}$ over $A$ is \emph{full} if $\langle\E,\E\rangle_\E= A$. When $\E$ is a Hilbert $A$-bimodule, it is right full if $\langle \E, \E \rangle_{\E} = A$ and left full if $_\E\langle \E, \E \rangle =A$. 

Let us now consider the relationship between $\mathrm{C}^*$-correspondences and Hilbert bimodules. 

Suppose that $\E$ is a Hilbert $A$-bimodule.  As is shown in \cite{BrownMingoShen}, the definition of a Hilbert $A$-bimodule guarantees that the left action of $A$ is adjointable, which is to say that $\langle a \xi, \eta \rangle_\E = \langle \xi, a^* \eta \rangle_\E$ for every $a \in A$ and every $\xi, \eta \in \E$. Thus we have a well-defined $^*$-homomorphism $\varphi_\E: A \to \mathcal{L}(\E)$, so any Hilbert $A$-bimodule can be given the structure of a $\mathrm{C}^*$-correspondence over $A$. Note that $_\E \langle \E, \E \rangle$ is an ideal in $A$ and by the compatibility condition of the inner products, we have $\varphi_\E( _\E \langle \E, \E \rangle) = \mathcal{K}(\E)$. If $\E$ is full as a left Hilbert $A$-module (for example, if $A$ is simple), then this implies that $A \cong \mathcal{K}(\E)$.

Now suppose that $(\mathcal{E}, \varphi_\E)$ is a $\mathrm{C}^*$-correspondence over a $\mathrm{C}^*$-algebra $A$. 
Denote 
\[ J_\E := \varphi_\E^{-1}(\mathcal{K}(\E))\cap (\ker \varphi_\E)^{\perp},\]
where
\[ (\ker \varphi_\E)^{\perp} := \{ a \in A \mid ab = 0 \text{ for all } b \in \ker \varphi_\E\}.\]
Evidently $J_\E$ is an ideal in $A$. If $\E$ is a $\mathrm{C}^*$-correspondence over $A$ with $\varphi_\E(J_\E) = \mathcal{K}(\E)$, then $\E$ can be given the structure of a Hilbert $A$-bimodule with left inner product
\[ _\E\langle \xi, \eta \rangle := \varphi_\E^{-1}(\theta_{\xi, \eta}).\]

\subsection{Tensor products} \label{subsec.TP}
Let $A$ be a $\mathrm{C}^*$-algebra and let $(\E, \varphi_\E)$ and $(\F, \varphi_\F)$ be $\mathrm{C}^*$-correspondences over $A$. We denote their internal tensor product by $\E \otimes_A \F$, which can always be given the structure of a $\mathrm{C}^*$-correspondence over $A$, see for example, \cite[Section 4.6]{BrownOzawaBook} and \cite[Chapter 4]{Lan:modules}.

If $\E$ is a $\mathrm{C}^*$-correspondence over $A$, to ease notation we omit the subscript $A$ and denote by $\E^{\otimes n}$ the $n$-fold internal tensor product of $\E$ over $A$.

When $\E$ and $\F$ are both Hilbert $A$-bimodules, so is $\E \otimes_A \F$. The left inner product on the tensor product $\E \otimes_A \F$ is given by
\[ _{\E \otimes_A \F}\langle \xi_1 \otimes \xi_2, \eta_1 \otimes \eta_2 \rangle = \,_\E\langle \xi_1, \eta_1 \,_\F\langle \eta_2, \xi_2 \rangle \rangle,  \qquad \xi_1, \eta_1 \in \E, \xi_2, \eta_2 \in \F, \]
or, equivalently, 
\[ _{\E \otimes_A \F}\langle \xi_1 \otimes \xi_2, \eta_1 \otimes \eta_2 \rangle = \,_\E\langle \xi_1 \,_\F\langle \xi_2, \eta_ 2\rangle, \eta_1  \rangle,  \qquad \xi_1, \eta_1 \in \E, \xi_2, \eta_2 \in \F. \]

When $\E$ is a Hilbert $A$-bimodule which is self-dual as a right Hilbert $A$-module, we can equip the dual module $\widehat{\E}$ with the structure of a Hilbert $A$-bimodule by defining
\[ a \, \hat{\xi} = \widehat{\xi a^*}, \quad \hat{\xi} a = \widehat{a^* \xi }, \quad a \in A, \xi \in \E, \]
and
\[ _{\widehat{\E}}\langle \hat{\xi}, \hat{\eta} \rangle = \langle \xi, \eta \rangle_\E, \quad \langle \hat{\xi}, \hat{\eta} \rangle_{\widehat{\E}} = \,_\E \langle \xi, \eta \rangle, \quad \xi, \eta \in \E.\]
Mapping $\hat{\xi} \otimes \eta \to \hat{\xi}(\eta) = \langle \xi, \eta \rangle_\E$ extends to an injective Hilbert $A$-bimodule map $\widehat{\E} \otimes_A \E \to A$ which is surjective when $\E$ is right full. Similarly, $\eta \otimes \hat{\xi} \to \,_\E\langle \eta,\xi\rangle$ extends to an injective Hilbert $A$-bimodule map $\E \otimes_A \widehat{\E} \to A$ which is surjective if $\E$ is left full. When $\widehat{\E} \otimes_A \E \cong A \cong \E \otimes_A \widehat{\E}$, we say that $\E$ is invertible.

\subsection{Cuntz--Pimsner algebras} 

Cuntz--Pimsner algebras were originally introduced by Pimsner \cite{Pimsner1997} and subsequently generalised by Katsura \cite{Katsura2003}. They generalise both Cuntz--Krieger algebras and crossed products by $\mathbb{Z}$.

 \begin{definition}[{\cite[Definition 2.1]{Katsura2004}}]  \label{def:cov rep} Let $A$ and $B$ be $\mathrm{C}^*$-algebras and let $\mathcal{E}$ be a \mbox{$\mathrm{C}^*$-correspondence} over $A$ with structure map $\varphi_\E : A \to \mathcal{L}(\mathcal{E})$. A \emph{representation} $(\pi, \tau)$ of $\E$ on $B$ consists of a $^*$-homomorphism $\pi: A \to B$ and a linear map $\tau: \mathcal{E} \to B$ satisfying
 \begin{enumerate}
 \item $\pi(\langle \xi, \eta \rangle_\E) = \tau(\xi)^*\tau(\eta)$, for every $\xi, \eta  \in \E$;
 \item $\pi(a)\tau(\xi) = \tau(\varphi_\E(a)\xi)$, for every $\xi \in \E$, $a \in A$.
 \end{enumerate}

 Note that (1) and the $\mathrm{C}^*$-identity imply $\tau(\xi) \pi(a) = \tau(\xi a)$ for every $\xi \in \E$ and $a \in A$.
 
Let $\psi_{\tau}  : \mathcal{K}(\E) \to B$ be the $^*$-homomorphism defined on rank one operators by 
\[ \psi_{\tau}(\theta_{\xi, \eta}) = \tau(\xi)\tau(\eta)^*, \qquad \xi, \eta \in \E.\]
 We say that the representation $(\pi , \tau)$ is \emph{covariant} \cite[Definition 3.4]{Katsura2004} if in addition
\begin{enumerate}[resume]
\item $\pi(a) =  \psi_{\tau}(\varphi_\E(a))$, for every $a \in J_{\E}$. 
\end{enumerate}
\end{definition}

\begin{definition}[{\cite[Definition 3.5]{Katsura2004}}] Let $A$ be a $\mathrm{C}^*$-algebra and let $\mathcal{E}$ be a \mbox{$\mathrm{C}^*$-correspondence} over $A$. The \emph{Cuntz--Pimsner algebra of $\mathcal{E}$ over $A$}, denoted $\mathcal{O}_A(\mathcal{E})$ (or simply $\mathcal{O}(\E)$ if the $\mathrm{C}^*$-algebra $A$ is understood) is the $\mathrm{C}^*$-algebra generated by the universal covariant representation of $\E$. 
\end{definition}

By universality, we mean that the universal covariant representation $(\pi_u, \tau_u)$ satisfies the following: for any covariant representation $(\pi, \tau)$ of $\E$ there exists a surjective $^*$-homomorphism $\psi \colon \mathrm{C}^*(\pi_u, \tau_u) \rightarrow \mathrm{C}^*(\pi, \tau)$ such that $\pi=\psi \circ \pi_u$ and $\tau=\psi \circ \tau_u$. See \cite[Section 4]{Katsura2004} for the construction of a concrete universal covariant representation. 

In \cite{AEE:Cross}, covariant representations were defined for Hilbert $A$-bimodules as follows. Let $\E$ be a Hilbert $A$-bimodule with left and right inner products given by $_\E\langle \cdot, \cdot \rangle$ and $\langle \cdot, \cdot \rangle_\E$, respectively. A covariant representation $(\pi, \tau)$ on a $\mathrm{C}^*$-algebra $B$ consists of a $^*$-homomorphism $\pi : A \to B$ and linear map $\tau : \E \to B$
satisfying
\begin{enumerate}
\item $\tau(\xi)\pi(a) = \pi(\xi a)$, for every $a \in A$ and every $\xi \in \E$;
\item $\pi(a)\tau(\xi) = \pi(a\xi)$, for every $a \in A$ and every $\xi \in \E$;
\item  $\pi(\langle \xi, \eta \rangle_\E) = \tau(\xi)^*\tau(\eta)$, for every $\xi, \eta  \in \E$;
\item $\pi(_\E\langle \xi, \eta \rangle)  = \tau(\xi)\tau( \eta)^*$, for every $\xi, \eta \in \E$.
\end{enumerate}

In \cite{AEE:Cross}, the $\mathrm{C}^*$-algebra generated by the universal covariant representation was denoted $A \rtimes_{\E} \mathbb{Z}$ and called \emph{the crossed product of $A$ by the Hilbert bimodule $\E$}. It is easy to see that if $\E$ is a  $\mathrm{C}^*$-correspondence such that the structure map satisfies $\varphi(J_\E) = \mathcal{K}(\E)$ and $\mathcal{F}$ is the associated Hilbert $A$-bimodule, then $A \rtimes_{\mathcal{F}} \mathbb{Z} \cong \mathcal{O}(\E)$. 

Just as in the case of a $\mathrm{C}^*$-correspondence, one can construct a covariant representation of a Hilbert $A$-bimodule $\E$, see \cite{AEE:Cross} for details. It follows from the definition of a covariant representation that when $A$ is unital, $A \rtimes_\E \mathbb{Z}$ contains a $^*$-isomorphic copy of $A$ as a subalgebra. Furthermore,  there is a closed linear subset $E \subset A \rtimes_\E \mathbb{Z}$ satisfying $A E \subset E$ and $E A \subset E$ and $E^*E \subset A$, $EE^* \subset A$ which gives both $E$ and $E^*$ the structure of a Hilbert $A$-bimodule such that $E \cong \E$ and $E^* \cong \widehat{\E}$.


\section{\texorpdfstring{$\mathrm{C}^*$}{C*}-correspondences over commutative \texorpdfstring{$\mathrm{C}^*$}{C*}-algebras}  \label{sec.Corr}

Let $X$ be a compact metric space. Here we are interested in the Cuntz--Pimsner algebras associated to $\mathrm{C}^*$-correspondences over $C(X)$. Of particular interest are those which are finitely generated projective as right Hilbert $C(X)$-modules.

\subsection{Sections of Hermitian vector bundles} First, we fix some notation. For further details on bundles, we direct the reader to \cite{Hus:fibre}. We denote a complex vector bundle over a locally compact metric space $X$ by $\mathscr{V}= [V, p, X]$, where $p : V \to X$  is a continuous surjective map and for every $x\in X$, the fibre $p^{-1}(x) \cong \mathbb{C}^{n_x}$ for some $n_x \in \mathbb{Z}_{\geq 0}$. By definition, a vector bundle is locally trivial, which is to say, for every $x \in X$ there exists an open neighbourhood $U$ of $x$ such that, for some $n \in \mathbb{Z}_{\geq 0}$, $\mathscr{V}|_U := [p^{-1}(U) \cong U \times \mathbb{C}^n, p|_{p^{-1}(U)}, U]$ is trivial. If there exists $n \in \mathbb{Z}_{\geq 0}$ such that $p^{-1}(x) \cong \mathbb{C}^n$ for every $x \in X$, then we say that $\mathscr{V}$ has constant rank, and say that $n$ is the rank of $\mathscr{V}$. Note the fibres of any connected component of $X$ must have the same dimension. In particular, if $X$ is connected, then $\mathscr{V}$ has constant rank.

A \emph{chart} for a complex vector bundle $\mathscr{V}= [V, p, X]$ is an open subset of $X$ together with an isometric isomorphism $h : U \times \mathbb{C}^{n_U} \to \mathscr{V}|_U$ where $\mathbb{C}^{n_U} \cong p^{-1}(x)$ for any $x \in U$. An \emph{atlas} for $\mathscr{V}$ is a family of charts $\{ h_i : U_i \times \mathbb{C}^{n_i} \to \mathscr{V}|_{U_i} \}_{i \in I}$ such that the $U_i$ cover $X$. Since $\mathscr{V}$ is a complex vector bundle over a locally compact Hausdorff space, we can equip it with a Hermitian structure, from which it follows that any atlas $\{ h_i : U_i \times \mathbb{C}^{n_i} \to \mathscr{V}|_{U_i} \}_{i \in I}$ comes with a \emph{system of transition functions} $\{g_{i,j} : U_i \cap U_j \to U(n_i) = U(n_j)\}$, where $U(n_i)$ denotes the $n_i \times n_i$ unitary group. The transition functions satisfy $h_j(x, v) = h_i(x, g_{i,j}(x)v)$ for $(x, v) \in (U_i \cap U_j) \times \mathbb{C}^{n_i}$ and for every $x \in U_i \cap U_j \cap U_k$, the relation $g_{i,k}(x) = g_{i,j}(x) g_{j,k}(x)$ holds. 

We recall here the Serre--Swan theorem, see~\cite[Theorem 2]{Swa:bundles}, which tells us what an algebraically finitely generated projective right $C(X)$-module looks like. (By saying $\E$ is algebraically finitely generated projective, we mean as a $C(X)$-module, that is, we do not assume there is any inner product on $\E$.) Given a vector bundle $\mathscr{V} = [V, p, X]$, we denote by $\Gamma(\mathscr{V})$ the $C(X)$-module of continuous sections of $\mathscr{V}$.

\begin{theorem}[Serre--Swan] \label{SerreSwan}
Let $X$ be a compact metric space and $\mathcal{E}$ be an algebraically finitely generated projective right $C(X)$-module. Then there exists a vector bundle $\mathscr{V} = [V, p, X]$ such that $\E \cong \Gamma(\mathscr{V})$ as right $C(X)$-modules.
\end{theorem}

 If $\mathscr{V}$ is a trivial rank $n$ vector bundle, then any continuous section $\xi \in \Gamma(\mathscr{V})$ is of the form $(x, f(x))$ for $x \in X$ and $f : X \to \mathbb{C}^{n}$ a continuous function.  If $\mathscr{V}$ is not trivial, then since $X$ is compact, there exist $N \in \mathbb{Z}_{>0}$ and an atlas consisting of $N$ charts $\{h_j :  U_j \times \mathbb{C}^{n_j}\to \mathscr{V}|_{U_j}\}_{j=1, \dots, N}$. In this case any continuous function $f \in C_0(U_j, \mathbb{C}^{n_j})$ allows us to define a section $\xi \in \Gamma(\mathscr{V})$ by setting 
 \[
 \xi(x) := \left \{ \begin{array}{cl} h_j(x, f(x)) & x \in U_j, \\
 0 & x \in X \setminus U_j.
 \end{array} \right.
 \] 
By abuse of notation, when we refer to such sections $\xi(x)$ we will often write $h_j(x, f(x))$.  If $\xi \in \Gamma(\mathscr{V})$ is an arbitrary section, then we can always find continuous functions $f_j \in C_0(U_j, \mathbb{C}^{n_j})$, $1 \leq j \leq N$, such that $\xi(x) = \sum_{j=1}^N h_j(x, f_j(x))$. Note that the $f_j$ are in general not unique.

By \cite[Lemma 2]{Swa:bundles}, the right $C(X)$-module $\Gamma(\mathscr{V})$ admits a right $C(X)$-valued inner product defined as follows. Let $\gamma_1, \dots, \gamma_N$ be a partition of unity subordinate to $U_1, \dots, U_N$. Define
 \begin{equation}\label{innerproduct} \langle \xi, \eta \rangle_{\Gamma(\mathscr{V})} (x) := \sum_{j=1}^N \gamma_j(x) \langle h_j^{-1}(\xi(x)), h_j^{-1}(\eta(x))\rangle_{\mathbb{C}^{n_j}},
 \end{equation} 
where, by abuse of notation, $h_j^{-1}(\xi(x))$ is the vector in $\mathbb{C}^{n_j}$ obtained after applying the isomorphism $\{x \} \times \mathbb{C}^{n_j} \cong \mathbb{C}^{n_j}$ to $h_j^{-1}$, and $\langle \cdot, \cdot \rangle_{\mathbb{C}^{n_j}}$ denotes the usual inner product in the vector space $\mathbb{C}^{n_j}$. This makes $\Gamma(\mathscr{V})$ into a right Hilbert $C(X)$-module. 

Since the $U_j \times \mathbb{C}^{n_j}$ are trivial vector bundles, there are pairwise orthonormal generators $s_{j, 1}, \dots, s_{j,n_j}$ of $\mathbb{C}^{n_j}$ such that
\[ \xi_{j,k} (x) = \left \{ \begin{array}{ll} h_j(x, s_{j,k} \gamma_j(x)^{1/2}) & x \in U_j, \\ 0 & x \in X \setminus U_j, \end{array} \right. \]
is a continuous section. The $\xi_{j,k}$ form a \emph{Parseval frame}, also called a \emph{normalised tight frame}, which is to say, they satisfy the reconstruction formula
\[ \xi = \sum_{j=1}^N\sum_{k=1}^{n_j} \xi_{j,k} \langle \xi_{j,k}, \xi \rangle_{\Gamma(\mathscr{V})}, \qquad \xi \in \Gamma(\mathscr{V}) ,\]
see \cite{FraLar:Frames} for further details. Rewriting this in terms of rank one operators (defined in \eqref{eqn.rankoneop}), we have 
\[ \sum_{j = 1}^N \sum_{k=1}^{n_j} \theta_{\xi_{j,k}, \xi_{j,k}} (\xi) = \xi, \qquad \xi \in \Gamma(\mathscr{V}),\]
 so that $\id_{\Gamma(\mathscr{V})} \in \mathcal{K}(\Gamma(\mathscr{V}))$. It follows that $\Gamma(\mathscr{V})$ is projective as a right Hilbert $C(X)$-module \cite[Remark 15.4.3]{Weg:k-theory}, and so there are $n \in \mathbb{Z}_{>0}$ and a projection $P \in \mathcal{L}(A^{\oplus n})$ such that $\Gamma(\mathscr{V})  \cong P A^{\oplus n}$, where the isomorphism is via a unitary $C(X)$-module map \cite[Theorem 15.4.2]{Weg:k-theory}. In particular, the inner product given in \eqref{innerproduct} does not depend on the choice of atlas. In fact, we have the following:

\begin{proposition} \label{prop:Mod from VB}
Let $X$ be an infinite compact metric space. Suppose that $\E$ is an algebraically finitely generated right Hilbert  $C(X)$-module. Then there exists a vector bundle $\mathscr{V} = [V, p , X]$ and a unitary isomorphism $U : \E \to \Gamma(\mathscr{V})$, where $\Gamma(\mathscr{V})$ is equipped with an inner product as defined in \eqref{innerproduct} with respect to any choice of atlas for $\mathscr{V}$.
\end{proposition}

When we refer to the right Hilbert $C(X)$-module $\Gamma(\mathscr{V})$, we will always mean the module  $\Gamma(\mathscr{V})$ equipped with the inner product coming from the Hermitian structure of $\mathscr{V}$ as in \eqref{innerproduct}. An analogous version of \Cref{prop:Mod from VB} also holds for left Hilbert $C(X)$-modules. 

These observations lead us to the following examples of $\mathrm{C}^*$-correspondences over $C(X)$ and their associated Cuntz--Pimsner algebras. 

\begin{examples} \label{CPexamples} Let $X$ be a compact metric space.
\begin{enumerate}
\item Let $\mathscr{V} = [V, p, X]$ be a line bundle, that is, the fibre at $x$ is $\mathbb{C}$ for every $x \in X$. Define $\varphi : C(X) \to \mathcal{K}(\Gamma(\mathscr{V}))$ by $\varphi(f)(\xi) = \xi f$. Note that in this case $C(X) \cong \mathcal{K}(\Gamma(\mathscr{V}))$ so $\Gamma(\mathscr{V})$ has the structure of a Hilbert $C(X)$-bimodule and hence $A := \mathcal{O}(\Gamma(\mathscr{V})) \cong C(X) \rtimes_{\Gamma(\mathscr{V})} \mathbb{Z}$. Since $A$ is generated by $C(X)$ and $\Gamma(\mathscr{V})$, it is a commutative $\mathrm{C}^*$-algebra. The spectrum of $\mathcal{O}(\Gamma(\mathscr{V}))$ is the circle bundle associated to $\mathscr{V}$ \cite[Proposition 4.10]{Vasselli2003}.
\item Let $\mathscr{V} = [X \times \mathbb{C}^n, p, X]$ be a trivial bundle with constant rank $n >1$. Again, let $\varphi : C(X) \to \mathcal{K}(\Gamma(\mathscr{V}))$ be given by $\varphi(f)(\xi) = \xi f$.  Then $\Gamma(\mathscr{V})$ has $n$ globally defined continuous sections $\xi_1, \dots, \xi_n$ which satisfy $\langle \xi_i, \xi_j \rangle_{\Gamma(\mathscr{V})}(x) = \delta_{i,j}$ and generate $\Gamma(\mathscr{V})$ as a right Hilbert $C(X)$-module. Thus if $\xi \in \Gamma(\mathscr{V})$ then 
\[ \xi = \sum_{i=1}^n\xi_i \langle \xi_i, \xi \rangle_{\Gamma(\mathscr{V})}.\]
Thus $\sum_{i=1}^n \xi_i \xi_i^* = 1$, so $\mathcal{O}(\Gamma(\mathscr{V})) \cong C(X, \mathcal{O}_n)$.
\item Let $\E$ be the right Hilbert $C(X)$-module of sections associated to a trivial line bundle $\mathscr{V} = [X \times \mathbb{C}, p, X]$ and let $\alpha : X\to X$ be a homeomorphism. Define $\varphi : C(X) \to \mathcal{K}(\Gamma(\mathscr{V}))$ by $\varphi(f)(\xi) = \xi f \circ \alpha$.  Since $\mathscr{V}$ is trivial, $\E$ is generated by a single element $\xi$ satisfying $1 = \langle \xi, \xi \rangle_\E = \,_\E\langle \xi, \xi  \rangle$. This gives us a unitary $u \in \mathcal{O}(\mathcal{E}) = C(X) \rtimes_{\mathcal{E}} \mathbb{Z}$. If $f \in C(X)$, then $ufu^* = f\circ\alpha^{-1}$. It follows that $\mathcal{O}(\E) \cong C(X) \rtimes_{\alpha} \mathbb{Z}$.
\end{enumerate}
\end{examples}

Generalising the above examples, we have the following, the main focus of this paper. 

\begin{example} Let $X$ be a compact metric space, $\mathscr{V}$ a vector bundle over $X$ and $\alpha : X \to X$ a homeomorphism.  Denote by $\Gamma(\mathscr{V}, \alpha)$ the $\mathrm{C}^*$-correspondence which has right Hilbert $C(X)$-module structure given by $\Gamma(\mathscr{V})$ and structure map $\varphi : C(X) \to \mathcal{K}(\Gamma(\mathscr{V}, \alpha))= \mathcal{K}(\Gamma(\mathscr{V}))$ given by $\varphi(f) (\xi) = \xi f \circ \alpha.$ Let us show that this gives a well-defined $^*$-homomorphism into $\mathcal{K}(\Gamma(\mathscr{V}, \alpha))$. That the map is an algebra homomorphism is clear. To see that it is a $^*$-homomorphism, we calculate
\begin{align*}
\langle \varphi(f) (\xi), \eta \rangle_{\Gamma(\mathscr{V}, \alpha)} (x) &= \langle \xi f \circ \alpha, \eta\rangle_{\Gamma(\mathscr{V}, \alpha)}(x) = \overline{\langle \eta, \xi f \circ \alpha \rangle}_{\Gamma(\mathscr{V}, \alpha)}(x) \\
&= \overline{(\langle \eta, \xi \rangle_{\Gamma(\mathscr{V}, \alpha)} f \circ \alpha)}(x) =  \langle \xi, \eta \rangle_{\Gamma(\mathscr{V}, \alpha)} \overline{f \circ \alpha} (x)\\
&= \langle \xi, \eta \overline{f \circ \alpha} \rangle_{\Gamma(\mathscr{V}, \alpha)} (x)= \langle \xi, \varphi(\overline{f}) (\eta) \rangle_{\Gamma(\mathscr{V}, \alpha)}(x), 
\end{align*}
for every $f \in C(X)$, every $x \in X$ and every $\xi, \eta \in \Gamma(\mathscr{V}, \alpha)$. Thus $\varphi(\overline{f}) = \varphi(f)^*$. Since the map $\varphi : C(X) \to \mathcal K(\Gamma(\mathscr{V},\alpha))$ is determined by the homeomorphism, to ease notation we will usually write $f \xi$ to mean $\varphi(f) (\xi)$, for $f \in C(X)$ and $\xi \in \Gamma(\mathscr{V}, \alpha)$, unless we require $\varphi$ to avoid confusion.
\end{example}

\subsection{Properties of \texorpdfstring{$\mathrm{C}^*$}{C*}-correspondences and Hilbert bimodules associated to modules of sections} 

Let $\E = \Gamma(\mathscr{V}, \alpha)$ where $\mathscr{V} = [V, p, X]$ is a vector bundle and $\alpha : X \to X$ is a homeomorphism. When $\mathscr{V}$ is a line bundle, as in Example \ref{CPexamples}~(1) where $\alpha = \id$, $\E$ can be given the structure of a Hilbert $C(X)$-bimodule. Conversely, if $\mathscr{V}$ has a fibre of rank greater than one, $\mathcal{K}(\E) \ncong C(X)$, so $\E$ only admits a $\mathcal{K}(\E)$-$C(X)$-bimodule structure, see \Cref{prop.CharOfFull} below and, for example, \cite[Proposition 5.18]{Katsura2004}.

\begin{proposition}\label{prop.LeftIP} 
Let $X$ be a compact metric space, $\alpha: X \to X$ a homeomorphism, and $\mathscr{V} = [V,p,X]$ a line bundle. Then the $\mathrm{C}^*$-correspondence $\E = \Gamma(\mathscr{V}, \alpha)$ is a Hilbert $C(X)$-bimodule satisfying
\[ _\E\langle \xi, \eta \rangle =  \varphi^{-1}(\theta_{\xi, \eta}) =  \langle \eta, \xi \rangle_\E \circ \alpha^{-1}, \qquad \xi, \eta \in \E,\]
where $\varphi : C(X) \to \mathcal{K}(\E)$ is the structure map of $\E$. 
\end{proposition}

\begin{proof}
Let $\{h_j :U_j \times \mathbb{C} \to \mathscr{V}|_{U_j}\}_{j=1,\dots, n}$ be an atlas for $\mathscr{V}$, $\{ g_{i,j} : U_i \cap U_j \to U(1)\}_{i,j=1, \dots, n}$ the corresponding transition functions, and let $\{\gamma_j\}_{j=1,\dots,n}$ be a partition of unity subordinate to the open cover $\{U_j\}_{j=1,\dots n}$. 

Let $\xi (x) := h_i(x, r(x))$, $\eta(x) := h_j(x, s(x))$, and $\zeta(x):=h_k(x, t(x))$, where $r \in C_0(U_i)$, $s \in C_0(U_j)$ and $t \in C_0(U_k)$, and suppose that $U_i \cap U_j \cap U_k \neq \emptyset$. Then for $x \in U_i \cap U_j \cap U_k$ we have
\begin{align*}
 \langle \eta,\zeta  \rangle_\E (x) &= \sum_{l=1}^n \gamma_l (x) \langle h_l^{-1}(\eta(x)), h_l^{-1}(\zeta(x)) \rangle_{\mathbb{C}} \\
 &=  \sum_{l=1}^n \gamma_l(x) \langle g_{l,j}(x) s(x), g_{l,k}(x) t(x) \rangle_\mathbb{C} \\
 &= \sum_{l=1}^n \gamma_l(x) \langle s(x), g_{j,l}(x) g_{l,k}(x) t(x) \rangle_\mathbb{C} \\
 &= \sum_{l=1}^n \gamma_l(x) \langle s(x), g_{j,k}(x) t(x) \rangle_\mathbb{C} \\
 &= \overline{s(x)} g_{j,k}(x) t(x).
\end{align*}
Similarly, $ \langle \xi, \eta  \rangle_\E (x) = \overline{r(x)} g_{i,j}(x) s(x).$ 
Thus
\begin{align*}
\theta_{\xi, \eta}(\zeta)(x) &= \xi \langle \eta,\zeta  \rangle_\E (x) \\
&= h_i(x, r(x)) \overline{s(x)} g_{j,k}(x) t(x) = h_i(x, g_{j,k} r \overline{s} t(x)),
\end{align*}
while 
\begin{align*}
\varphi(\langle \eta, \xi \rangle_\E \circ \alpha^{-1}) \zeta (x) &= \zeta (x)\langle \eta, \xi \rangle_\E(x) \\
    &= h_i(x, g_{i,k}(x) t(x))\overline{s(x)} g_{j,i}(x) r(x) = h_i(x, \overline{s} g_{j,i} r g_{i,k} t(x)) \\
    &= h_i(x, \overline{s} g_{j,k} r t(x)).
\end{align*}
Hence $\theta_{\xi, \eta}(\zeta) = \varphi(\langle \eta, \xi \rangle_\E \circ \alpha^{-1}) \zeta$. 

If $\xi, \eta, \zeta \in \E$ are arbitrary, there are $n \in \mathbb{Z}_{>0}$ and $r_j, s_j, t_j \in C_0(U_j)$, $1 \leq j \leq n$ such that $\xi(x) = \sum_{j=1}^n h_j(x, r_j(x))$, $\eta(x) = \sum_{j=1}^n h_j(x, s_j(x))$ and $\zeta(x) = \sum_{j=1}^n h_j(x, t_j(x))$. That $\theta_{\xi, \eta} =\varphi(\langle \eta, \xi \rangle_\E \circ \alpha^{-1})$ now follows from the sesquilinearity of the inner products. Since $\E$ is full as a right Hilbert $C(X)$-module and $\alpha$ is a homeomorphism, it follows that with $J_\E = \varphi^{-1}(\mathcal{K}(\E)) \cap \ker(\varphi)^\perp$, we have $\varphi(J_\E) = \varphi(C(X)) = \mathcal{K}(\E)$. That $\E$ admits the structure of a Hilbert $C(X)$-bimodule with 
\[ _\E\langle \xi, \eta \rangle = \varphi^{-1}(\theta_{\xi,\eta}) = \langle \eta, \xi \rangle_\E \circ \alpha^{-1}, \]
now follows from the discussion in Section~\ref{subsec.CorBimod}.
\end{proof}

 For a Hilbert $C(X)$-bimodule $\E$ we will write $\E_{C(X)}$ if we are considering $\E$ as a right Hilbert $C(X)$-module, and $_{C(X)}\E$ when considering $\E$ as a left Hilbert $C(X)$-module.
 
 If $\mathscr{W} = [W, q, X]$ is a vector bundle, then we can endow $\Gamma(\mathscr{W})$ with a left $C(X)$-module structure via pointwise multiplication and if $\{ k_j : U_j \times \mathbb{C}^{n_j} \to \mathscr{W}|_{U_j}\}$ is an atlas for $\mathscr{W}$ with $\gamma_j$, $1 \leq j \leq n$ a partition of unity subordinate to $U_j$, $1\leq j \leq n$, then we make it into a left Hilbert $C(X)$-module using the inner product
 \[ _{\Gamma(\mathscr{W})}\langle \xi, \eta\rangle (x) := \sum_{j=1}^n \gamma_j(x) \,\,_{\mathbb{C}^{n_j}}\!\langle h_j^{-1}(\xi(x)), h_j^{-1}(\eta(x)) \rangle,\]
 for $\xi, \eta \in {}_{C(X)}\Gamma(\mathscr{W})$. As before, this does not depend on the choice of atlas. 
 
 Let $\mathscr{V}=[V, p, X]$ be a vector bundle and $f \colon Y \to X$ be a continuous map. The induced vector bundle of $\mathscr{V}$ under $f$, denoted by $f^*(\mathscr{V})$, has as base space $Y$, as total space $f^*(V)$ which is the subspace of all $ (v, y) \in V \times Y$  with $f(y)=p(v)$, and as a projection the map $f^*(V) \to Y: (v,y)\mapsto y$.
 
\begin{lemma} \label{lem.lefty}
 Let $\E = \Gamma(\mathscr{V}, \alpha)$ for $\mathscr{V} = [V, p,X]$ a line bundle over $X$ and $\alpha : X \to X$ a homeomorphism. Then $\Gamma((\alpha^{-1})^*\mathscr{V}) \cong \,_{C(X)} \E$ as left Hilbert $C(X)$-modules.
\end{lemma}

\begin{proof}
Let $\{h_j: U_j\times \mathbb{C}\to \mathscr{V}|_{U_j}\}$ be an atlas for $\mathscr{V}$, and $\gamma_1, \dots, \gamma_n \in C(X)$ a partition of unity subordinate to $U_1, \dots, U_n$ so that $\xi_j$, $1\leq j\leq n$ defined by $\xi_j(x) =h_j(x, \gamma_j^{1/2}(x))$ generate $\E$. An atlas for  $(\alpha^{-1})^*\mathscr{V}$ is then given by $\{ k_j : \alpha(U_j) \times \mathbb{C} \to (\alpha^{-1})^*\mathscr{V}|_{\alpha(U_j)}\}$, where $k_j(y, \lambda) = (h_j(\alpha^{-1}(y), \lambda), y)$.  Let $\eta_j$, $1 \leq j \leq n$, be defined by $\eta_j(y) := k_j(y, \gamma_j^{1/2}\circ \alpha^{-1}(y)) =  (h_j(\alpha^{-1}(y), \gamma_j^{1/2} \circ \alpha^{-1}(y)), y) = (\xi_j\circ\alpha^{-1}(y), y)$ and note that they generate the left Hilbert $C(X)$-module $\Gamma((\alpha^{-1})^*\mathscr{V})$. 

Let $T(\eta_j) = \xi_j$ for $1 \leq j \leq n$, and $y\in X$. Note that 
\begin{align*}
    f \eta_j(y) &=  (f(y) h_j(\alpha^{-1}(y),  \gamma_j^{1/2} \circ \alpha^{-1}(y), y) \\
    &=  (h_j(\alpha^{-1}(y), f\circ \alpha(\alpha^{-1}(y)) \gamma_j^{1/2} \circ \alpha^{-1}(y)), y)  \\
    &= ((\xi_j f \circ \alpha) (\alpha^{-1}(y)), y) \\
    &= (f \xi_j(\alpha^{-1}(y)), y).
\end{align*} 
It follows that
\begin{align*}
   T(f \eta_j)  =  f \xi_j = f T(\eta_j), 
\end{align*} so this extends to a well-defined left $C(X)$-linear map $T : \Gamma((\alpha^{-1})^*(\mathscr{V})) \to \,_{C(X)}\E$. Furthermore,
\begin{align*}
   \,_\E\langle T \eta_j, T \eta_k \rangle(x) &=\langle T \eta_k, T \eta_j \rangle_\E \circ \alpha^{-1}(x)\\
 &= \sum_{l=1}^n \gamma_l\circ\alpha^{-1}(x)  \overline{h_l^{-1}(\xi_k(\alpha^{-1}(x)))} h_l^{-1}(\xi_j(\alpha^{-1}(x)))  \\
 &= \sum_{l=1}^n \gamma_l\circ\alpha^{-1}(x) k_l^{-1}(\eta_j(x)) \overline{k_l^{-1}(\eta_k(x))}\\
 &= \,_{\Gamma((\alpha^{-1})^{*}\mathscr{V})}\langle \eta_j, \eta_k \rangle(x),
\end{align*}
for every $x \in X$ and every $1 \leq j,k \leq n$. It follows that $T$ is a unitary transformation. Thus $T$ is an isormorphism of left Hilbert C(X)-modules, $\Gamma((\alpha^{-1})^*\mathscr{V}) \cong \,_{C(X)} \E$.
\end{proof}

If $A = C(X)$ and $\E = \Gamma(\mathscr{V}, \alpha)$ for a homeomorphism $\alpha : X \to X$ and vector bundle $\mathscr{V} = [V,p,X]$, then it is easy to see that $\E$ is a full $\mathrm{C}^*$-correspondence. Furthermore, we have the following characterisation of Hilbert $C(X)$-bimodules which are finitely generated projective as right Hilbert $C(X)$-modules. The result below for left and right full Hilbert $C(X)$-bimodules can also be found in \cite{AbadieExel1997}, see  the discussion in Section~\ref{sec.examples}. 

\begin{proposition} \label{prop.CharOfFull}
Let $\E$ be a non-zero Hilbert $C(X)$-bimodule which is finitely generated projective as a right Hilbert $C(X)$-module. Then there exist a compact metric space $Y \subset X$, line bundle $\mathscr{V}= [V,p,X]$ and  homeomorphisms $\alpha : X \to Y$, $\beta : Y \to X$ such that $\E_{C(X)} \cong \Gamma(\mathscr{V})$, $_{C(X)}\E \cong \Gamma(\beta^*\mathscr{V})$ and $f \xi = \xi f \circ \alpha$ for every $f \in C(X)$ and every $\xi \in \E$. If $\E$ is left full, then we may take $X = Y$ and $\E \cong \Gamma(\mathscr{V}, \alpha)$.
\end{proposition}

\begin{proof}
Since $\E$ is finitely generated projective as a right Hilbert $C(X)$-module, there exists a vector bundle $\mathscr{V} = [V, p , X]$ such that $\E_{C(X)} \cong \Gamma(\mathscr{V})$ (Proposition \ref{prop:Mod from VB}). 
Let $\{h_j: U_j\times \mathbb{C}^{n_j}\to \mathscr{V}|_{U_j}\}$ be an atlas for $\mathscr{V}$, $\gamma_1, \dots, \gamma_n \in C(X)$ a partition of unity subordinate to $U_1, \dots, U_n$. 

We claim that $\mathscr{V}$ is a line bundle. Since $\E$ is a Hilbert $C(X)$-bimodule, we have $\theta_{\xi,\eta}(\zeta)=\xi\langle\eta,\zeta\rangle_\E=\,_\E\langle\xi,\eta\rangle\zeta$ for every $\zeta$ and thus $\mathcal K(\E)=\overline{\Span}\{\theta_{\xi,\eta}\mid \xi,\eta\in \E\}= \,_\E\langle\E,\E\rangle$, an ideal of $C(X)$, is commutative. But if $\mathscr{V}$ is not a line bundle, then we can find nonzero $\xi,\eta\in \E$ for which $\theta_{\xi,\eta}$ does not commute with $\theta_{\eta,\xi}$. In fact, if $n_j\geq 2$ for some $j$, then $\xi,\eta\in \E$ defined by $\xi(x):=\gamma_j(x)h_j(x,e_1)$ and $\eta(x):=\gamma_j(x)h_j(x,e_2)$ are mutually orthogonal, and with $\zeta=\xi+\eta$, we have  
$\theta_{\xi,\eta}\circ\theta_{\eta,\xi}(\zeta)=\xi\langle\eta,\eta\rangle_\E\langle\xi,\zeta\rangle_\E=\xi\langle\eta,\eta\rangle_\E\langle\xi,\xi\rangle_\E\neq\theta_{\eta,\xi}\circ\theta_{\xi,\eta}(\zeta)$. It follows that $\mathscr{V}$ is a line bundle, proving the claim.

Let $\xi_j(x) = h_j(x, \gamma_j^{1/2}(x))$, $1 \leq j \leq n$. Then  $\xi_1, \dots, \xi_n$ generate $\Gamma(\mathscr{V})$ as a right Hilbert $C(X)$-module and satisfy $\sum_{j=1}^n \langle \xi_j, \xi_j \rangle_\E = 1$. Define
\[ \lambda, \rho : C(X) \to C(X)\]
by 
\[ \lambda(f) = \sum_{j=1}^n \,_\E\langle \xi_j f, \xi_j \rangle, \]
and 
\[ \rho(f) = \sum_{j=1}^n \langle \xi_j, f \xi_j \rangle_\E.\]
Then one checks that $\lambda$ and $\rho$ are $^*$-homomorphisms. We have
\begin{align*}
    \rho(\lambda(f)) &= \sum_{j=1}^n \langle \xi_j, \lambda(f) \xi_j \rangle_\E = \sum_{j=1}^n \langle \xi_j, (\sum_{i=1}^n \,_\E \langle \xi_i f, \xi_i \rangle) \xi_j \rangle_\E \\
    &= \sum_{j=1}^n \langle \xi_j, \sum_{i=1}^n \xi_i f\langle \xi_i, \xi_j \rangle_\E \rangle_\E = \sum_{j=1}^n \langle \xi_j, \sum_{i=1}^n \xi_i \langle \xi_i, \xi_j \rangle_\E f \rangle_\E \\
    &= \sum_{j=1}^n \langle \xi_j, \sum_{i=1}^n \xi_i \langle \xi_i, \xi_j \rangle_\E  \rangle_\E f = \sum_{j=1}^n \langle \xi_j,  \xi_j  \rangle_\E f = f,
\end{align*}
and 
\begin{align*}
    \lambda(\rho(f)) &= \sum_{j=1}^n \,_\E \langle \xi_j \rho(f), \xi_j \rangle =  \sum_{j=1}^n \,_\E \langle \xi_j \sum_{i=1}^n \langle \xi_i, f \xi_i \rangle_\E, \xi_j \rangle \\
    &=\sum_{j=1}^n \,_\E \langle \sum_{i=1}^n \,_\E \langle \xi_j, \xi_i \rangle f \xi_i, \xi_j \rangle = \sum_{j=1}^n f \,_\E \langle \sum_{i=1}^n  \,_\E \langle \xi_j, \xi_i \rangle \xi_i, \xi_j \rangle \\
    &= f \sum_{j=1}^n \,_\E \langle \sum_{i=1}^n \xi_j \langle \xi_i, \xi_i \rangle_\E, \xi_j \rangle = f \sum_{j=1}^n \,_\E \langle \xi_j, \xi_j \rangle.
\end{align*}
Note that $\rho(\sum_{j=1}^n \,_\E \langle \xi_j, \xi_j \rangle) = \rho(\lambda(1_{C(X)})) = 1_{C(X)}$. Also, since $\lambda(1_{C(X)}) = \sum_{j=1}^n \,_\E \langle \xi_j, \xi_j \rangle$ is a projection in $C(X)$, we have that $\sum_{j=1}^n \,_\E \langle \xi_j, \xi_j \rangle = \chi_Y$ for some clopen subset $Y \subset X$. Thus there are continuous maps $\alpha:X \to Y$ and $\beta: Y\to X $ such that $\lambda(f)(x)=f\circ\beta(x)$ for every $x\in Y$ and $\rho(f)(x) = f \circ \alpha(x)$ for every $x \in X$. Moreover, we have $\beta \circ \alpha = \id_X$ and $\alpha \circ \beta = \id_Y$. Thus, by Gelfand duality, $X \cong Y \subset X$. Observe that $_{C(Y)}\E$,  where $_{C(Y)}\E$ denotes $_{C(X)}\E$ restricted to $Y$, is full and hence finitely generated projective as a left Hilbert $C(Y)$-module. 

Now, for every $\xi \in \E$ and every $f \in C(X)$, we have 
\begin{align*}
    \xi \rho(f) &= \xi \sum_{j=1}^n \langle \xi_j, f \xi_j \rangle_\E = \sum_{j=1}^n  \,_\E\langle \xi, \xi_j \rangle f \xi_j=  f \sum_{j=1}^n \,_\E\langle \xi, \xi_j \rangle \xi_j = f \sum_{j=1}^n \xi \langle \xi_j, \xi_j \rangle_\E = f \xi.
\end{align*}
Thus the left action is given by $f \xi = \xi f \circ \alpha$.

Define a left $C(X)$-valued sesquilinear form on $_{C(X)}\E$ by
\[ \,_\E \langle\langle \xi, \eta \rangle\rangle (x) :=  \left \{ \begin{array}{ll} \langle \eta, \xi \rangle_\E \circ \beta(x) & x \in Y,\\
0 & x \in X \setminus Y.\end{array} \right.\]

Since $\xi \in \,_{C(X)}\E$ satisfies $\xi = \sum_{j=1}^n \,_{\E} \langle \xi_j, \xi_j \rangle \xi = \chi_Y \xi$, we have that  $ _{\E}\langle\langle \cdot , \cdot  \rangle\rangle$ is positive definite. It is easy to check that $\,_{C(X)}\E$  is complete with respect to the induced norm, and hence this inner product makes $\,_{C(X)}\E$ into a left Hilbert $C(X)$-module. Moreover, $(_{C(Y)}\E, \,_\E\langle \langle \cdot, \cdot \rangle \rangle)$, where $(_{C(Y)}\E, \,_\E\langle \langle \cdot, \cdot \rangle \rangle)$ denotes $(_{C(X)}\E, \,_\E\langle \langle \cdot, \cdot \rangle \rangle)$ restricted to $Y$, is finitely generated projective as a left Hilbert $C(Y)$-module. As in the proof of \Cref{prop.LeftIP} we have
\[  \,_{\E}\langle \langle \xi, \eta \rangle \rangle \zeta = \xi \langle \eta, \zeta \rangle_\E, \]
for every $\xi, \eta, \zeta \in \E$. Since the norm of a finitely generated projective Hilbert $C(X)$-module is unique up to unitary isomorphism, we conclude that $(_{C(X)}\E, \,_\E\langle \cdot, \cdot \rangle) \cong (_{C(X)}\E,  \,_{\E}\langle \langle \cdot, \cdot \rangle \rangle)$. A proof analogous to that of \Cref{lem.lefty} tell us that there is an isomorphism of left Hilbert $C(Y)$-modules $(_{C(Y)}\E, \,_\E\langle \langle \cdot, \cdot \rangle \rangle) \cong \Gamma(\beta^*\mathscr{V})$. This extends to an isomorphism of (not necessarily full) left Hilbert $C(X)$-modules by putting
\[ _{C(X)}\langle \xi, \eta \rangle (x) := \left\{ \begin{array}{ll} _{C(Y)}\langle \xi, \eta \rangle(x) &  x \in Y,  \\
0 & x \in X \setminus Y,\end{array}\right.
\] for every $\xi, \eta \in \Gamma(\beta^*\mathscr{V})$.

Finally, if $_{C(X)} \E$ is left full, then $X = Y$ and so by  \Cref{lem.lefty} we have $\E \cong \Gamma(\mathscr{V}, \alpha)$ as Hilbert $C(X)$-bimodules. 
\end{proof}

Let $X$ be an infinite compact Hausdorff space. A homeomorphism $\alpha : X \to X$ is \emph{minimal} if, whenever $E \subset X$ is a closed subset such that $\alpha(E) \subset E$, then $E \in \{ \emptyset, X\}$, or equivalently, for every $x \in X$, the orbit of $x$, $\orb(x) := \{ \alpha^n(x) \mid n \in \mathbb{Z}\}$ is dense in $X$. A \emph{periodic point} for a homeomorphism $\alpha : X\to X$ is a point $x$ such that $\alpha^n(x) = x$ for some $n \in \mathbb{Z} \setminus \{0\}$. We say $\alpha$ is \emph{aperiodic} if there are no periodic points in $X$. Clearly a minimal homeomorphism on an infinite compact Hausdorff space is aperiodic, but aperiodicity does not imply minimality.

The notions of minimality and aperiodicity have been generalised to the setting of $\mathrm{C}^*$-correspondences.

\begin{definition} Let $A$ be a $\mathrm{C}^*$-algebra and  $\E$ a $\mathrm{C}^*$-correspondence over $A$ with structure map $\varphi_\E$.
\begin{enumerate}
    \item We say that $\E$ is \emph{periodic} if there exists $n \in \mathbb{Z}_{> 0}$ and a unitary transformation $U :  \mathcal{E}^{\otimes n} \to A$. Otherwise it is \emph{nonperiodic}.
    \item We say that $\mathcal{E}$ is \emph{minimal} if, whenever  $J \subset A $ is a (closed, two-sided) ideal satisfying $\langle \mathcal{E}, \varphi_\E(J) \mathcal{E} \rangle_\E \subset J$, then $J \in \{ 0, A\}$.
\end{enumerate}
\end{definition}

\begin{proposition}\label{minimalnonperiodic}
Let $X$ be an infinite compact metric space and $\alpha\colon X\to X$ a homeomorphism. Let $\mathscr{V} = [V, p, X]$ be a vector bundle. Then $\Gamma (\mathscr{V}, \alpha)$ is minimal if and only if $\alpha$ is minimal.
\end{proposition}

\begin{proof} Let $E$ be a closed subset of $X $ and  set \[J_E=\{f\in C(X)\, |\, f|_{E}=0\}\triangleleft C(X).\] 
Then $E = \{ x \in X \mid f(x) = 0 \text{ for every } f \in J_E\}$. 

Suppose that $\Gamma(\mathscr{V}, \alpha)$ is minimal. Let $E \subset X$ be a non-empty closed subset such that $\alpha(E) \subset E$. Then for any $f \in J_E$ we have $f\circ \alpha \in J_E$. Since
\[ \langle \xi, f \, \eta \rangle_{\Gamma(\mathscr{V}, \alpha)} = \langle \xi, \eta \, (f \circ \alpha) \rangle_{\Gamma(\mathscr{V}, \alpha)} =  \langle \xi, \eta \rangle_{\Gamma(\mathscr{V}, \alpha)} \, f \circ \alpha,\]
for any $f \in C(X)$ and any $\xi, \eta \in \Gamma(\mathscr{V}, \alpha)$, if $f \in J_E$ we have 
\[  \langle \xi, f \, \eta \rangle_{\Gamma(\mathscr{V}, \alpha)}(x) =  \langle \xi, \eta \rangle_{\Gamma(\mathscr{V}, \alpha)} f \circ \alpha (x)= 0,\]
 for every $x \in E$. Thus $\langle \xi, \varphi_\E( J_E) \eta \rangle_{\Gamma(\mathscr{V}, \alpha)} \subset J_E$. Since $E$ is non-empty, by the minimality of $\Gamma(\mathscr{V}, \alpha)$ we must have $J_E = \{0\}$ and hence $E = X$, showing that $\alpha$ is minimal. 

Conversely, suppose that $\alpha$ is minimal and that $J \subset C(X)$ is a non-zero ideal such that $\langle \xi, \varphi_\E(J) \eta \rangle_{\Gamma(\mathscr{V}, \alpha)} \subset J$ for every $\xi, \eta \in \Gamma(\mathscr{V}, \alpha)$. Since $J$ is non-zero, there exists a proper closed subset $E \subset X$ such that $J = J_E$. For any $x \in E$, there exist $\xi, \eta \in \Gamma(\mathscr{V}, \alpha)$ such that $\langle \xi, \eta \rangle_{\Gamma(\mathscr{V}, \alpha)}(x) =1$. Then 
\[ 0 =  \langle \xi, f \cdot \eta \rangle_{\Gamma(\mathscr{V}, \alpha)} (x) =  \langle \xi , \eta   \rangle_{\Gamma(\mathscr{V}, \alpha)}(x) f(\alpha(x) )= f (\alpha(x)),\]
for every $f \in J_E$.  Hence $\alpha(x) \in E$ for every $x \in E$ and so $E \in \{\emptyset, X\}$. By assumption, $J_E$ is non-zero, so we must have $E = \emptyset$ and hence $J_E = C(X)$, showing that $\Gamma(\mathscr{V}, \alpha)$ is minimal.
\end{proof}

\begin{proposition}\label{thm.aperiodic}
Let $X$ be an infinite compact metric space and $\alpha\colon X\to X$ a homeomorphism. Let $\mathscr{V} = [V, p, X]$ be a vector bundle. If $\alpha$ is aperiodic then $\Gamma(\mathscr{V}, \alpha)$ is nonperiodic.
\end{proposition}

\begin{proof}
Suppose that $\alpha$ is aperiodic and that there exists an $n \in \mathbb{Z}_{> 0}$ and a unitary $C(X)$-module map $U : \Gamma(\mathscr{V}, \alpha)^{\otimes n} \to C(X)$. 
Note that for every $f\in C(X)$ and every simple tensor $\xi_1\otimes\cdots\otimes\xi_n$ in $\Gamma(\mathscr{V},\alpha)^{\otimes n}$, we have \begin{align*}
fU(\xi_1\otimes\cdots\otimes\xi_n)
& = U(f\cdot\xi_1\otimes\cdots\otimes\xi_n)\\ 
& =U(\xi_1 (f\circ \alpha)\otimes\cdots\otimes \xi_n)
= \cdots = U(\xi_1\otimes\cdots\otimes \xi_n)(f\circ\alpha^n).
\end{align*} 
This implies that for  $\eta\in \Gamma(\mathscr{V},\alpha)^{\otimes n}$ with $U(\eta)=1_{C(X)}$, 
\begin{equation*}
f=f 1_{C(X)}=fU(\eta)=U(\eta)(f\circ\alpha^n)=(f\circ\alpha^n)
\end{equation*}
for all $f\in C(X)$. Let $x \in X$. Since $\alpha$ has no periodic points and $n \neq 0$, we have $\alpha^n(x) \neq x$ and there exist open subsets $U, V \subset X$ with $x \in U$, $\alpha^n(x) \in V$ satisfying $U \cap V = \emptyset$. Choose $f \in C(X)$ satisfying $f|_U =1$ and $f|_V = 0$. Then $1 = f(x) = f (\alpha^n(x)) = 0$, a contradiction.  Thus $\Gamma(\mathscr{V}, \alpha)$ is nonperiodic.
\end{proof}

From \Cref{minimalnonperiodic} and \Cref{thm.aperiodic} we obtain immediately the following result.
\begin{corollary}\label{19-10-18-3}
$\mathcal{O}(\Gamma(\mathscr{V}, \alpha))$ is simple if and only if $\alpha$ is minimal.
\end{corollary}

\begin{proof}
Since $\Gamma(\mathscr{V}, \alpha)$ is full, minimality and nonperiodicity of $\Gamma(\mathscr{V}, \alpha)$ is equivalent to simplicity of $\mathcal{O}(\Gamma(\mathscr{V}, \alpha))$ \cite[Theorem 3.9]{Schweizer2001}.
\end{proof}


\section{Cuntz--Pimsner algebras for \texorpdfstring{$\mathrm{C}^*$}{C*}-correspondences over \texorpdfstring{$C(X)$}{C(X)}}\label{sec.cp}

Let $\E$ be a $\mathrm{C}^*$-correspondence over a unital $\mathrm{C}^*$-algebra $A$. 
Let $a = f \eta_1 \cdots \eta_n \in \mathcal{O}(\E)$ where $f \in A$ and $\eta_j \in \{ \xi \mid \xi \in \E\} \cup \{\xi^* \mid \xi \in \E\}$, $1 \leq j \leq n$.  We say the \emph{length} of $a$ is $n$.  Since $\xi_j^* \xi_k = \langle \xi_j, \xi_k \rangle_\E \in A$, there are unique $k, l \in \mathbb{Z}_{\geq 0}$ such that $k+ l \leq n$ and $a = f \xi_1 \cdots \xi_k   \xi_{k+1}^* \dots \xi_{k+l}^*$ . It follows that $k+l$ is the unique \emph{minimal length} of $a$. We refer to   $\deg(a) = k-l$ as the \emph{degree} of $a$.  In particular we get,
\[ \mathcal{O}(\E) = \overline{\bigoplus_{n \in \mathbb{Z}} \overline{E_n}}, \]
where $E_n$ is the $A$-linear span of  $\{ a \in \mathcal{O}(\E) \mid \deg(a) = n \}$.

The Cuntz--Pimsner algebra $\mathcal{O}(\E)$ admits a gauge action of the circle $\sigma: \mathbb{T} \to \Aut(\mathcal{O}(\E))$, see \cite[Section 3]{Pimsner1997} and \cite[Section 5]{Katsura2004}. Note that this recovers the grading given above as, for every $n \in \mathbb{Z}$, the spectral subspace $\{a \in \mathcal{O}(\E) \mid \sigma_z(a) = z^n a \}$ is $E_n$. 

\begin{remark} \label{rem.DAofLB}Observe that since $\E$ is algebraically finitely generated, so too is $\E^{\otimes n}$,  for any $n \in \mathbb{Z}_{> 0}$. If $\E$ is a Hilbert bimodule, then $\xi_1 \xi_2^* = \,_\E \langle \xi_1, \xi_2\rangle \in A$,
so that if $a \in \mathcal{O}(\E)$ has degree $k >0$,  there exist  $\xi_{1}, \dots, \xi_{k} \in \E$, such that $a = \xi_{1} \cdots \xi_{k}$, while if $\deg(a)= k< 0$, there exist  $\xi_{1}, \dots, \xi_{k} \in \E$ such that $a = \xi_{1}^* \cdots \xi_{k}^*$.  It follows that for every $n \in \mathbb{Z}_{> 0}$, the subspaces $E_n \cong \E^{\otimes n}$ while  $E_{-n} \cong \widehat{\E}^{\otimes n}$ and the fixed point algebra, $\mathcal{O}(\E)^{\sigma} = E_0$ is isomorphic to $A$. Hence $E_n = \overline{E_n}$ for every $n \in \mathbb{Z}$. In particular this is the case for $A=C(X)$ and $\E = \Gamma(\mathscr{V}, \alpha)$ when $\mathscr{V}$ is a line bundle, since in that case $\E$ is a Hilbert $C(X)$-bimodule by \Cref{prop.LeftIP}. We will tacitly identify $\E^{\otimes n}$ with $E_n$ throughout the paper. 

If $\E$ is a $\mathrm{C}^*$-correspondence with structure map $\varphi_\E$ such that $\varphi_\E(A) \not \cong A$, we still have that $A \subset E_0$, but $E_0$ will be a noncommutative $\mathrm{C}^*$-algebra containing isometric copies of $\mathcal{K}(\E^{\otimes n})$ for every $n \geq 1$. In this case $E_0 \subsetneq \overline{E_0} = \mathcal{O}(\E)^{\sigma}$. Similarly, while a copy of $\E^{\otimes n}$ is contained in $E_n$, we do not have $\E^{\otimes n} \cong \overline{E_n}$. This is the case for $A = C(X)$ and $\Gamma(\mathscr{V}, \alpha)$ when  $\mathscr{V}$ has a fibre of rank greater than one.
\end{remark}

\begin{lemma} \label{lem.SwitchSides}
Let $X$ be an infinite compact metric space and $\mathscr{V} = [V,p,X]$ a line bundle. Let $\E := \Gamma(\mathscr{V}, \alpha)$ for a minimal homeomorphism $\alpha\colon X\to X$. Let $n \in \mathbb{Z}_{>0}$. Then for any $\xi, \eta \in \E^{\otimes n}$ we have
\[ \,_{\E^{\otimes n}}\langle \xi, \eta \rangle = \langle \eta, \xi \rangle_{\E^{\otimes n}} \circ \alpha^{-n}.\]
\end{lemma}

\begin{proof} The statement is true for $n=1$ by \Cref{prop.LeftIP}. Assume that it holds for $n \geq 1$. Let $\xi = \xi_1 \otimes \cdots \otimes \xi_{n+1}$ and $\eta = \eta_1 \otimes \cdots \otimes \eta_{n+1}$. Then
\begin{align*}
  \lefteqn{_{\E^{\otimes n+1}}\langle \xi_1 \otimes \cdots \otimes \xi_{n+1}, \eta_1 \otimes \cdots \otimes \eta_{n+1} \rangle}\\
  &= \,_\E \langle \xi_1 \,_{\E^{\otimes n}} \langle \xi_2 \otimes \cdots \otimes \xi_{n+1}, \eta_2 \otimes \cdots \otimes \eta_{n+1} \rangle, \eta_1 \rangle \\
   &= \,_{\E^{\otimes n}} \langle \xi_2 \otimes \cdots \otimes \xi_{n+1}, \eta_2 \otimes \cdots \otimes \eta_{n+1} \rangle \circ \alpha^{-1} \,_\E \langle \xi_1, \eta_1 \rangle \\
   &=  \langle \eta_2 \otimes \cdots \otimes \eta_{n+1},\xi_2 \otimes \cdots \otimes \xi_{n+1} \rangle_{\E^{\otimes n}} \circ \alpha^{-n-1} \langle \eta_1, \xi_1 \rangle_\E \circ \alpha^{-1}\\
   &= \langle \eta_2 \otimes \cdots \otimes \eta_{n+1},\xi_2 \otimes \cdots \otimes \xi_{n+1} \langle \eta_1, \xi_1 \rangle_\E \circ \alpha^{n} \rangle_{\E^{\otimes n}} \circ \alpha^{-n-1} \\
      &= \langle \eta_2 \otimes \cdots \otimes \eta_{n+1},\langle \eta_1, \xi_1 \rangle_\E  \xi_2 \otimes \cdots \otimes \xi_{n+1}  \rangle_{\E^{\otimes n}} \circ \alpha^{-n-1} \\
      &= \langle \eta_1 \otimes \cdots \otimes \eta_{n+1}, \xi_1 \otimes \cdots \otimes \xi_{n+1} \rangle_{\E^{\otimes n+1}} \circ \alpha^{-n-1}.
\end{align*}
That $_{\E^{\otimes n+1}} \langle \xi, \eta \rangle = \langle \eta, \xi \rangle_{\E^{\otimes n+1}} \circ \alpha^{-n-1}$ for every $\xi, \eta \in \E^{\otimes n+1}$ now follows by linearity, and the statement follows by induction.
\end{proof}

\begin{remark}\label{conditionalexpectation}
The gauge action gives us an associated conditional expectation onto the fixed point $\mathrm{C}^*$-subalgebra given by
\begin{equation} \label{CondExp}
\Phi : \mathcal{O}(\E) \to  \mathcal{O}(\E)^{\sigma}, \qquad a \mapsto \int_{\mathbb{T}} \sigma_z(a) dz.
\end{equation}

It is straightforward to check that if $a \in\bigoplus_{n \in \mathbb{Z}} E_n$ then we have  $\int_{\mathbb{T}} \sigma_z(a^*a) dz = 0$ if and only if $a = 0$.  It follows that $\Phi$ is faithful.
\end{remark}

\begin{proposition} \label{prop:traces}
Let $\E = \Gamma(\mathscr{V}, \alpha)$ where $\mathscr{V}=[V,p,X]$ is a vector bundle and $\alpha : X \to X$ is a  homeomorphism. If $\mathscr{V}$ is not a line bundle, then $\mathcal{O}(\E)$ has no faithful tracial state. On the other hand,  $T(\mathcal{O}(\E)) \neq \emptyset$  if $\mathscr{V}$ is a line bundle. 
\end{proposition}

\begin{proof}
Suppose first that $\mathscr{V}$ is not a line bundle.  Let $\{ h_i : U_i \times \mathbb{C}^{n_i} \to \mathscr{V}|_{U_i}\}_{i \in I}$ be a finite atlas of charts for $\mathscr{V}$, and let $\xi_{i, j}(x) :=  h_i(x, \gamma_i^{1/2}(x) e^{(i)}_j)$ where $e^{(i)}_1, \dots, e^{(i)}_{n_i}$ is the standard orthonormal basis for $\mathbb{C}^{n_i}$ and $\{\gamma_i\}_{i \in I}$ is a partition of unity subordinate to $\{U_i\}_i$. Then $1  = \sum_{i\in I} \langle \xi_{i,j}, \xi_{i, j} \rangle$ for every $j=1,\dots, n_i$. Since the $\xi_{i,j}$ generate the module and form a Parseval frame, we also have 
\[
\eta  = \sum_{i\in I} \sum_{j=1}^{n_i} \xi_{i,j} \langle \xi_{i,j}, \eta \rangle= \left (\sum_{i\in I} \sum_{j=1}^{n_i} \theta_{\xi_{i,j}, \xi_{i,j}} \right) \eta
\]
for every $\eta \in \E$. Thus $\sum_{i\in I} \sum_{j=1}^{n_i} \theta_{\xi_{i,j}, \xi_{i,j}}$ is the identity operator. Upon identifying the $\xi_{i,j}$  and $\theta_{\xi_{i,j}, \xi_{i,j}}$ with their images in $\mathcal{O}(\E)$ we see that, for any state $\tau$,
\[ 1 = \tau\left (\sum_{i\in I} \sum_{j=1}^{n_i} \xi_{i,j} \xi_{i,j}^*\right) = 
\sum_{i\in I} \sum_{j=1}^{n_i} \tau(\xi_{i,j} \xi_{i,j}^*).\]
If $\tau$ is faithful, then $\tau(\xi_{i,j}^*\xi_{i,j})$ is non-zero for every $1 \leq j \leq n_i$ and every  $i \in I$. If additionally $\tau$ is a tracial state, this implies
\[ 1 =  \sum_{i\in I} \sum_{j=1}^{n_i} \tau(\xi_{i,j} \xi_{i,j}^*) = \sum_{i\in I} \sum_{j=1}^{n_i} \tau(\xi_{i,j}^* \xi_{i,j}) > \sum_{i \in I}  \tau(\xi_{i,1}^* \xi_{i,1}) 
\geq 1.\]
Thus $\tau$ cannot be a faithful tracial state. 

Now let $\mathscr{V}$ be a line bundle. Let $\mu$ be an $\alpha$-invariant probability measure on $X$, and let
\[ \tau_\mu(f) = \int_X f d\mu,\]
which is evidently a state. We claim that $\tau_\mu \circ\Phi$ is tracial. It is enough to show that the tracial condition holds on the dense subalgebra $\bigoplus_{n \in \mathbb{Z}} E_n$. It follows from the observations in \Cref{rem.DAofLB}, that if $a \in \bigoplus_{n \in \mathbb{Z}} E_n$ there is $N \in \mathbb{Z}_{\geq 0}$ and $\xi_n \in E_n$, $-N \leq n \leq N$ such that $a = \sum_{n=-N}^N \xi_n$. (Note that $\xi_0 \in C(X)$.) Thus
\begin{align*}
    \Phi(a^*a) &= \Phi\left( \sum_{j,k = -N}^N \xi_j^* \xi_k \right)\\
    &= \sum_{j=-N}^N \xi_j^*\xi_j \\
    &= \sum_{j=-N}^N \xi_j \xi_j^* \circ \alpha^{-j}, \\
\end{align*} 
where the third equality follows from Lemma~\ref{lem.SwitchSides}.
Since $\mu$ is $\alpha$-invariant, this shows that $\tau_{\mu} \circ \Phi (a^*a) = \tau_{\mu} \circ \Phi (aa^*)$, which proves the claim. Since $X$ always has an $\alpha$-invariant Borel probability measure, it follows that $T(\mathcal{O}(\E)) \neq \emptyset$.  
\end{proof}

\begin{proposition} \label{prop.T(A)}
Let $X$ be an infinite compact metric space, $\mathscr{V} =[V,p, X]$ a line bundle, and $\alpha : X\to X$ an aperiodic homeomorphism. Let $\E := \Gamma(\mathscr{V}, \alpha)$. 

Then there are affine homeomorphisms
\[ T(\mathcal{O}(\E)) \cong T(C(X) \rtimes_\alpha \mathbb{Z}) \cong M^1(X, \alpha),\]
where $M^1(X, \alpha)$ denotes the space of $\alpha$-invariant Borel probability measures. 
\end{proposition}

\begin{proof}
By the proof of the previous proposition, every $\alpha$-invariant Borel probability measure gives rise to a tracial state on $\mathcal{O}(\E)$. Let us show that every tracial state arises in this way.

Let $\tau \in T(\mathcal{O}(\E))$. It is enough to show that $\tau = \tau|_{C(X)} \circ \Phi$, as the result then follows from the Riesz representation theorem. Let $a \in \bigoplus_{n \in \mathbb{Z}} E_n$. Then $a=  \sum_{j=-N}^N  \xi_{j}$ for some $\xi_{j} \in E_j$, $-N \leq j \leq N$, and $\xi_{0} \in C(X)$. Since $\alpha$ is aperiodic, there exist $n \in \Z_{n>0}$ and $s_1, s_2, \dots, s_n \in C(X)$ such that $|s_k(x)|=1$ every $x \in X$, $1\leq k \leq n$, and which satisfy
\[ \frac{1}{n} \sum_{k=1}^n s_k(x) \overline{s_k} \circ \alpha^{-m}(x) = 0,\]
for every $x \in X$,  and $m \in \{-N, -N+1, \dots, N\} \setminus \{0\}$  (see for example, \cite[Lemma 11.1.18]{GioKerPhi:CRM} or \cite[Lemma VIII.7.1]{Dav:C*-ex} for the case when $\alpha$ is minimal).
Thus 
\begin{align*}
  \frac{1}{n} \sum_{k=1}^n s_k a \overline{s_k} &= \frac{1}{n} \sum_{k=1}^n \left( s_k\left (\sum_{j=-N}^N  \xi_{j} \right) \overline{s_k}\right) \\
  &=\frac{1}{n} \sum_{j=-N}^N \sum_{k=1}^n s_k  \xi_{j} \overline{s_k} \\
  &=\frac{1}{n}  \sum_{j=-N}^N  \sum_{k=1}^n s_k\, \overline{s_k}  \circ \alpha^{-j}  \, \xi_j \, \\
  &= \frac{1}{n}   \sum_{k=1}^n \overline{s_k}s_k \xi_{0} = \xi_0 = \Phi(a).
\end{align*}
Thus for every tracial state $\tau \in T(\mathcal{O}(\E))$, we have 
\[ \tau(a) = \tau|_{C(X)}\circ \Phi(a).\]
Since $\bigoplus_{n \in \mathbb{Z}} E_n$ is dense, $\tau = \tau|_{C(X)} \circ \Phi(a)$.

The tracial state space of $C(X) \rtimes_\alpha \mathbb{Z}$ is homeomorphic to the space of $\alpha$-invariant Borel probability measures on $X$ (see for example \cite[Lemma 11.1.22]{GioKerPhi:CRM}). As all the maps considered are natural, the statement of the proposition now holds.
\end{proof}

\begin{corollary} \label{cor:stabfin}
Let $X$ be an infinite compact metric space and $\mathscr{V} =[V,p, X]$ a line bundle. Let $\E = \Gamma(\mathscr{V}, \alpha)$ where  $\alpha : X \to X$ is a minimal homeomorphism. Then $\mathcal{O}(\E)$ is stably finite.
\end{corollary}

\begin{proof}
By ~\Cref{prop:traces}, $\mathcal{O}(\E)$ has a tracial state. Since minimality of $\alpha$ implies that $\mathcal{O}(\E)$ is simple, this tracial state must be faithful. It follows that $\mathcal{O}(\E)$ is stably finite.
\end{proof}

\section{Rokhlin dimension and classification when \texorpdfstring{$X$}{X} is finite dimensional} \label{sec.RD}

The Rokhlin dimension for a $\mathrm{C}^*$-correspondence was defined by Brown, Tikuisis and Zelenberg \cite{MR3845113}. It generalises the original definition of Rokhlin dimension due to Hirshberg, Winter, and Zacharias \cite{HirWinZac:RokDim}, which in turn is a generalisation of the Rokhlin property \cite{Kishi:RP}, a $\mathrm{C}^*$-algebraic version of the Rokhlin lemma in ergodic theory. The Rokhlin dimension is useful for giving estimates on the nuclear dimension of a $\mathrm{C}^*$-algebra. Nuclear dimension is a refinement of the completely positive approximation property, and when it is finite a $\mathrm{C}^*$-algebra is particularly well-behaved. We won't need the precise definition here, as we only need to know when the nuclear dimension is finite, which will follow from finiteness of the Rokhlin dimension. The interested reader can consult \cite{WinterZac:dimnuc} for the definition and properties of the nuclear dimension.

\begin{definition}[{\cite[Definition 5]{MR3845113}}]\label{Defn31}
	Let $A$ be a separable unital $\mathrm{C}^*$-algebra and let $\mathcal{E}$ be a countably generated $\mathrm{C}^*$-correspondence over $A$. We say that $\mathcal{E}$ has \emph{Rokhlin dimension at most $d$} if, for any $\epsilon>0$, any $p \in \mathbb{Z}_{>0}$, every finite subset $F \subset A$ and every finite subset $\mathcal{G} \subset \mathcal{E}$, there exist positive contractions 
	\[ \{f^{(l)}_k \}_{l = 0, \dots, d; k \in \mathbb{Z}/p\mathbb{Z}} \subset A\]
	satisfying 
	\begin{enumerate}
		\item $\| f_k^{(l)} f_{k'}^{(l)} \| < \epsilon $ when $k\neq k'$,
		\item $\| \sum_{k,l} f^{(l)}_k -1 \| < \epsilon$,
		\item $\| \xi f^{(l)}_k - f^{(l)}_{k+1} \xi \| < \epsilon$ for every $k, l$, every $\xi \in \mathcal{G}$,
		\item $\| [ f_k^{(l)}, a ] \| < \epsilon$ for every $k, l$ and $a \in F$.
	\end{enumerate}
\end{definition}

\begin{theorem} \label{thm.FinRokDim}
	Let $X$ be an infinite compact metric space with $\dim(X) < \infty$,  $\mathscr{V} = [V, p, X]$ a vector bundle, and $\alpha : X \to X$ an aperiodic homeomorphism. Then $\E = \Gamma(\mathscr{V}, \alpha)$ has finite Rokhlin dimension.
	\end{theorem}

\begin{proof} Let $G$ be a finite subset of $\E$, $\epsilon>0$, and $p \in \mathbb{Z}_{>0}$ be given. Without loss of generality, we may assume $G$ consists of norm one elements. Since $\alpha : X \to X$  is aperiodic and $X$ is finite dimensional,  \cite[Corollary 2.6]{MR3342101} shows that $\alpha$ has finite Rokhlin dimension (with single towers) in the sense of \cite[Definition 2.3]{HirWinZac:RokDim}. Let $d$ denote the Rokhlin dimension of the homeomorphism $\alpha$. Then there are positive functions $f_i^{(l)} \in C(X)$, $i \in \mathbb{Z}/p\mathbb{Z}$, $0\leq l\leq d$, such that
	
	\begin{enumerate}[label=(\roman*)]
\item for any $l$, $0\leq l\leq d$, $\|f^{(l)}_i f^{(l)}_j\| < \epsilon$ whenever $i \neq j$,
\item $\| \sum_{l=0}^d \sum_{i \in \mathbb{Z}/p\mathbb{Z}} f^{(l)}_i - 1\| < \epsilon$,
\item $\|f^{(l)}_i \circ \alpha^{-1} - f^{(l)}_{i+1} \| < \epsilon$ for every $i \in \mathbb{Z}/p\mathbb{Z}$ and every $l$, $0\leq l\leq d$.
\end{enumerate}
	
It is straightforward to see that $f_i^{(l)}$ will satisfy (1), (2) and (4) of \Cref{Defn31}. For (3) we have
\begin{align*}
\| \xi f^{(l)}_i - f^{(l)}_{i+1} \xi \| &= \|f^{(l)}_i \circ \alpha^{-1} \xi - f^{(l)}_{i+1} \xi \| \leq \|f^{(l)}_i \circ \alpha^{-1} - f^{(l)}_{i+1} \| \|\xi\| < \epsilon
\end{align*}
for every $\xi\in G$, which completes the proof.
\end{proof}

\begin{theorem}\label{fin_nuc_dim}
	Let $X$ be an infinite compact metric space with $\dim(X) < \infty$,   $\mathscr{V}=[V,p,X]$ a  vector bundle, and $\alpha : X \to X$ an aperiodic homeomorphism. Let $\E = \Gamma(\mathscr{V}, \alpha)$. Then  $\mathcal{O}(\E)$ has finite nuclear dimension. If $\alpha$ is minimal, the  nuclear dimension is at most one.
\end{theorem}

\begin{proof}
    	The $C(X)$-correspondence $\E$ is finitely generated projective and has finite Rokhlin dimension. Since $\dim(X) < \infty$, that $\mathcal{O}(\E)$ has finite nuclear dimension follows immediately from Corollary 4.16 and Example 4.2 in \cite{MR3845113}. Thus when $\alpha$ is minimal, $\mathcal{O}(\E)$ is $\mathcal{Z}$-stable by \cite{Win:Z-stabNucDim}, and that $\mathcal{O}(\E)$ has nuclear dimension at most one now follows from \cite[Theorem B]{CETWW}.
\end{proof}

	Let $\mathcal{C}$ denote the class of $\mathrm{C}^*$-algebras of the form $\mathcal{O}(\Gamma(\mathscr{V}, \alpha))$ for $X$ an infinite compact metric space with $\dim(X) < \infty$,  $\mathscr{V} = [ V, p, X]$ a vector bundle, and $\alpha : X \to X$ a minimal homeomorphism. Now that we know that any $A \in \mathcal{C}$ has finite nuclear dimension, we are able to apply the machinery of the Elliott classification programme. The classification theorem, stated here, is the culmination of many years of work. 

\begin{theorem}[see for example \cite{CETWW, BBSTWW:2Col, EllGonLinNiu:ClaFinDecRan, GongLinNiue:ZClass, GongLinNiue:ZClass2, TWW}]
 \label{ClassThm} Let $A$ and $B$ be separable, unital, simple, infinite dimensional \mbox{$\mathrm{C}^*$-algebras} with finite nuclear dimension and which satisfy the UCT. Suppose there is an isomorphism 
\[ \psi : \Ell(A) \to \Ell(B).\]
Then there is a $^*$-isomorphism 
\[ \Psi : A \to B,\]
which is unique up to approximate unitary equivalence and satisfies $\Ell(\Psi) = \psi$.
\end{theorem}

By \cite[Proposition 8.8]{Katsura2004}, since $C(X)$ is commutative, $\mathcal{O}(\Gamma(\mathscr{V}, \alpha))$ satisfies the UCT. This give us the following classification theorem for $\mathcal{C}$.

\begin{theorem}\label{class}
 Suppose that $A, B \in \mathcal{C}$ and 
	\[ \psi : \Ell(A) \to \Ell(B) \]
	is an isomorphism. Then there exists a $^*$-isomorphism 
	\[ \Psi : A \to B,\]
which is unique up to approximate unitary equivalence and satisfies $\Ell(\Psi) = \psi$.
\end{theorem}

\begin{corollary} \label{dichotomy}
Let $A = \mathcal{O}(\Gamma(\mathscr{V}, \alpha)) \in \mathcal{C}$. 
\begin{enumerate}
    \item If $\mathscr{V}$ is a line bundle, $A$ has stable rank one.
    \item If $\mathscr{V}$ has (not necessarily constant) rank greater than one, $A$ is purely infinite. 
\end{enumerate}
\end{corollary} 

\begin{proof}
When $\mathscr{V}$ is a line bundle, $A$ is stably finite by \Cref{cor:stabfin}. Since $A$ has finite nuclear dimension, it is $\mathcal{Z}$-stable by \cite[Corollary 7.3]{Win:Z-stabNucDim}. This in turn implies that $A$ has stable rank one \cite[Theorem 6.7]{Ror:Z-absorbing}, showing (i).

If $\mathscr{V}$ is not a line bundle, it has no tracial states by \Cref{prop:traces}. Since $A$ is simple and has finite nuclear dimension, $A$ is purely infinite by \cite[Theorem 5.4]{WinterZac:dimnuc}, showing (ii).
\end{proof}

\section{Orbit-breaking subalgebras of \texorpdfstring{$\mathcal{O}(\Gamma(\mathscr{V}, \alpha))$}{O(Gamma(V, alpha)} } \label{sec.large}

In this section, we restrict to the case where the Hilbert $\mathrm{C}^*$-correspondence comes from a Hilbert $C(X)$-bimodule. This means that $\mathscr{V} = [V, p, X]$ is a line bundle (Proposition \ref{prop.CharOfFull}). 

Let $X$ be an infinite compact metric space, $\alpha : X \to X$ a minimal homeomorphism and let $\E := \Gamma(\mathscr{V}, \alpha)$. For any open subset $U \subset X$ we denote by $\E_{X\setminus U} := C_0(U)\E$ the Hilbert $C(X)$-bimodule given by restricting the left action to $C_0(U)$. The Cuntz--Pimsner algebra $\mathcal{O}(\E_{X\setminus U})$ is a $\mathrm{C}^*$-subalgebra of $\mathcal{O}(\E)$. Let $\mathcal{E} =\Gamma(\mathscr V,\alpha)$ where $\mathscr{V}$ is a trivial line bundle. As shown in Example~\ref{CPexamples} (3), the Cuntz--Pimsner algebra is just the usual crossed product by $\alpha$,  that is, $C(X) \rtimes_{\alpha} \mathbb{Z}$. When $U = X \setminus Y$ where $Y$ is a closed non-empty subset, the subalgebra $\mathcal{O}(\E_{Y})$ is then the $\mathrm{C}^*$-subalgebra $\mathrm{C}^*(C(X), C_0(X \setminus Y) u) \subset C(X)\rtimes_{\alpha} \mathbb{Z}$, where $u$ denotes the unitary implementing $\alpha$, which is called an \emph{orbit-breaking} subalgebra. When $Y$ meets every $\alpha$-orbit at most once, that is, $\alpha^n(Y) \cap Y = \emptyset$ for every $n \in \mathbb{Z} \setminus \{0\}$, orbit-breaking subalgebras have been very useful objects in the study of crossed products by minimal homeomorphisms. Originally introduced by Putnam in his study of crossed products arising from Cantor minimal systems \cite{Putnam:MinHomCantor}, they were subsequently used by various authors for more general dynamical systems, see for example \cite{LinPhi:MinHom,TomsWinter:minhom, Str:XxSn, StrWin:Z-stab_min_dyn, QLin:Ay, EllNiu:MeanDimZero}, as they are often more tractable while at the same time sharing many properties of the crossed product in which they are contained. For example, they are also simple (see \cite[Proposition 2.1]{DPS:OrbitBreaking}) and the inclusion of  $\mathcal{O}(\E_Y) \into C(X) \rtimes_{\alpha} \mathbb{Z}$ induces an affine homeomorphism of tracial state spaces. If $Y = \{y\}$, then there is moreover an isomorphism of $K_0$-groups \cite[Theorem 2.4 and Example 2.6]{Put:K-theoryGroupoids}. 

Orbit-breaking algebras also give rise to interesting $\mathrm{C}^*$-algebras in their own right, particularly from the point of view of the Elliott classification programme.  For example, one can realise any simple unital AF algebra (see \cite{Putnam:MinHomCantor, GioPutSkau:orbit}),  the Jiang--Su algebra \cite{DPS:JiangSu}, and simple nuclear $\mathrm{C}^*$-algebras with a wide range of $K$-theory \cite{DPS:OrbitBreaking} as orbit-breaking algebras.  

The definition of an orbit-breaking subalgebra easily generalises to the current setting:

\begin{definition}
Let $X$ be an infinite compact metric space, $\mathscr{V} = [V,p,X]$ a vector bundle and $\alpha : X \to X$ a minimal homeomorphism. Let $\E := \Gamma(\mathscr{V}, \alpha)$.  Let $Y \subset X$ be a non-empty closed subset. The \emph{orbit-breaking subalgebra} of $\mathcal{O}(\E)$ at $Y$ is $\mathcal{O}(C_0(X \setminus Y)\E)$, that is, the Cuntz--Pimsner algebra of the \cstar-correspondence over $C(X)$ given by $C_0(X \setminus Y) \E$. When $\mathscr{V}$ is a line bundle, we call $\E_Y := C_0(X \setminus Y) \E$ an \emph{orbit-breaking bimodule}.
\end{definition}

\subsection{Orbit-breaking bimodules}

We can't apply many of the results of the last section to orbit-breaking bimodules because they are not full. Let $\E_Y$ be an orbit-breaking bimodule. It is easy to see that $\langle \E_Y, \E_Y \rangle_{\E} = C_0(X \setminus \alpha^{-1}(Y))$ while  $_\E\langle \E_Y, \E_Y \rangle = C_0(X \setminus Y)$ and so $\E_Y$ is neither left nor right full. We begin this subsection with the following observation.

\begin{lemma} \label{lem.SectionsOfEy}
For any element $\xi \in \E$, we have that $\xi \in \E_Y$ if and only if $\xi$ vanishes on $\alpha^{-1}(Y)$. 
\end{lemma}

\begin{proof}  If $\xi = f \eta \in \E_Y$ for some $\eta \in \E$ and $f \in C_0(X \setminus Y)$, then $\xi(x) = \eta f \circ \alpha(x) = 0$ whenever $x \in \alpha^{-1}(Y)$. Since $\E_Y = \overline{\mathrm{span}}\{f \xi \mid f \in C_0(X \setminus Y), \xi \in \E\}$, it follows that if $\xi \in \E_Y$ is arbitrary, then $\xi(x)= 0$ for every $x \in \alpha^{-1}(Y).$

Conversely, let $\{h_j :  U_j \times \mathbb{C} \to \mathscr{V}|_{U_j}\}_{j= 1, \dots, n}$ be an atlas for $\mathscr{V}$ and $\gamma_1, \dots, \gamma_n$ a partition of unity subordinate to $U_1, \dots, U_n$. Let $\xi_j = h_j(x, \gamma^{1/2}_j(x))$, $1 \leq j \leq n$ be a corresponding set of generators for $\E$.  Suppose $\xi \in \E$ satisfies $\xi(x) = 0$ for every $x \in \alpha^{-1}(Y)$. Then 
\[ \langle \xi_j, \xi \rangle_\E (x) = \sum_{k=1}^n \gamma_k(x) \overline{h_k^{-1}(\xi_j(x))}h_k^{-1}(\xi(x)) = 0,
\]for every $x \in \alpha^{-1}(Y)$.

Let $f_j = \langle \xi_j , \xi \rangle_\E \circ \alpha^{-1}$. Then $f_j(x) = 0$ for every $x \in Y$, and  
\[ \xi = \sum_{j=1}^n \xi_j \langle \xi_j, \xi \rangle_\E = \sum_{j=1}^n \langle \xi_j, \xi \rangle_\E \circ \alpha^{-1} \xi_j = \sum_{j=1}^n f_j \xi_j .\] Thus $\xi\in \E_Y$. 
\end{proof}

The above lemma allows us to show that if $Y$ is not clopen, $\E_Y$ is not algebraically finitely generated, nor even algebraically countably generated. Fix an atlas $\{h_U : U \times \mathbb{C} \to \mathscr{V}|_{U}\}_{U \in \mathcal{U}}$ for $\mathscr{V}$.  Let $W \subset X$ be an open subset such that $W \cap \alpha^{-1}(Y) \neq \emptyset$, $\mathscr{V}|_W$ is trivial, and $\overline{W} \subsetneq V$ for some $V \in \mathcal{U}$. Suppose that $(\xi_n)_{n \in \mathbb{Z}_{\geq 0}}$ generate $\E_Y$. Then there exist $(f_n)_{n \in \mathbb{Z}_{\geq 0}} \subset C_0(V)$ such that $\xi_n(x) = h_V(x, f_n(x))$ for every $x \in W$ and every $n \in \mathbb{Z}_{\geq 0}$.  We may assume that $\|f_n\| \leq 1$ for every $n \in \mathbb{Z}_{\geq 0}$.

Let $f(x) = \sum_{n \in \mathbb{Z}_{
\geq 0}} 2^{-n}|f_n(x)|^{1/2}$ for every $x \in W$ and $f(x) = 0$ for every $x \in X \setminus V$, and let $\xi(x) :=h_V(x, f(x))$. Note that $\xi \neq 0$. By assumption $(\xi_n)_{n\in \mathbb{Z}_{\geq 0}}$ generate $\E_Y$, so there are $m \in \mathbb{Z}_{> 0}$ and $g_1, \dots, g_m \in C(X)$ such that $\xi(x) = \sum_{j=1}^m h_V(x, f_{n_j}(x)) g_j(x) = \sum_{j=1}^m h_V(x, f_{n_j} g_j(x))$ for some $n_1, \dots, n_m \in \mathbb{Z}_{\geq0}$. Thus $f = \sum_{j=1}^m f_{n_j} g_j$. Let $K = \max_{j = 1, \dots, m} \| g_j \|$. Since $\xi_{n_j} \in \E_Y$, we have that $f_{n_j}(y) = 0$ for any $y \in \alpha^{-1}(Y) \cap W$ and $j=1,\dots, m$. Let $y_0 \in \alpha^{-1}(Y) \cap W$, and let $U$ be a neighbourhood of $y_0$ contained in $W$ such that $|f_{n_j}(x)|^{1/2} < 2^{-n_j-1} K^{-1}$ for every $x \in U$, $1 \leq j \leq m$. Then  
\begin{align*}
\sum_{n \in \mathbb{Z}_{\geq 0}} 2^{-n}|f_n(x)|^{1/2} &= |f(x)| \\
&\leq \sum_{j=1}^m |f_{n_j}(x)||g_j(x)|\\
&<  \sum_{j=1}^m 2^{-n_j-1}  K^{-1}|f_{n_j}(x)|^{1/2}|g_j(x)| \\
&< \sum_{j=1}^m 2^{-n_j-1}|f_{n_j}(x)|^{1/2} \\
&< |f(x)|/2,
\end{align*} for every $x \in U \setminus \alpha^{-1}(Y)$, which is impossible. Thus $\xi \neq \sum_{j=1}^m \xi_{n_j} g_j$ for any $m \in \mathbb{Z}_{> 0}$, $n_1, \dots, n_m\in \mathbb{Z}_{\geq 0}$ and $g_1, \dots, g_j \in C(X)$, so $\E_Y$ is not countably generated. 

 Although $\E_Y$ is not algebraically countably generated, we can still determine when an element $a \in \mathcal{O}(\E)$ is in the orbit-breaking subalgebra $\mathcal{O}(\E_Y)$. First, we need a lemma. 

Recall that when $\E = \Gamma(\mathscr{V}, \alpha)$ with $\mathscr{V}$ a line bundle, $\bigoplus_{n \in \mathbb{Z}} E_n$ is dense in $\mathcal{O}(\E)$ where $E_n \cong \E^{\otimes n}$ and $E_{-n} = E_n^*$ for every $n > 0$ (see \Cref{rem.DAofLB}). In what follows  $\Phi : \mathcal{O}(\E) \to \mathcal{O}(\E)^{\sigma} = E_0 \cong C(X)$ denotes the conditional expectation defined in \Cref{conditionalexpectation}. 

\begin{lemma}\label{lem:cesaro}
Let $X$ be an infinite compact metric space, $\mathscr{V} = [V,p,X]$ a vector bundle and $\alpha : X \to X$ a minimal homeomorphism. Let $\E := \Gamma(\mathscr{V}, \alpha)$. For any $a \in \mathcal{O}(\E)$ and $N > 0$, let $\xi_{j,k} $, $1 \leq k \leq d_j$ be generators of $E_j$, $-N \leq j \leq N$, satisfying $\sum_{k=1}^{d_j} \xi_{j,k}^* \xi_{j,k} = 1$. Define $\Psi_N : \mathcal{O}(\E) \to \bigoplus_{n \in \mathbb{Z}} E_n$ by
\begin{equation*} \Psi_N(a) := \sum_{j=-N}^N \left(1 - \frac{|j|}{N+1} \right)
\left(\sum_{k=1}^{d_j} \Phi(a \xi_{j,k}^*) \xi_{j,k}\right).
\end{equation*}
Then  $\lim_{N \to \infty} \|\Psi_N(a) - a\| \to 0$.
\end{lemma}

\begin{proof}
First we show that for each $N >0$, $\Psi_N$ is contractive. Let $b \in \bigoplus_{n \in \mathbb{Z}} E_n$, so that  $b = \sum_{n=-K}^K  \eta_n$ for some $K\in\mathbb{Z}_{>0}$ and $\eta_n \in E_n$, $-K \leq n \leq K$. We have

\begin{align*}
    \Phi(b \xi_{j,k}^*) \xi_{j,k}  &=  \int_0^1 \sigma_{e^{2 
    \pi i t}}( b \xi_{j,k}^*) dt \, \xi_{j,k} \\
    &= \int_0^1 \left( \sum_{n=-K}^K  e^{2 
    \pi i (n-j) t} \eta_{n} \xi_{j,k}^* \right) dt \, \xi_{j,k} \\
    &=  \int_0^1 e^{-2 \pi i j t} \left( \sum_{n=-K}^K  e^{2 
    \pi i n t} \eta_{n} \right) dt \, \xi_{j,k}^*\xi_{j,k}\\
    &= \int_0^1 e^{-2 \pi i j t} \sigma_{e^{2 
    \pi i t}}( b )  dt \, \xi_{j,k}^*\xi_{j,k}, \\
\end{align*}
so that 
\begin{align*}
    \sum_{k=1}^{d_j}\Phi(b \xi_{j,k}^*) \xi_{j,k}
    &= \int_0^1 e^{-2 \pi i j t} \sigma_{e^{2 
    \pi i t}}( b )  dt.
\end{align*}
Thus $\sum_{k=1}^{d_j} \Phi(a \xi_{j,k}^*) \xi_{j,k} =  \int_0^1 e^{-2 \pi i j t} \sigma_{e^{2 \pi i t}}( a )  dt$ 
for every $a \in \mathcal{O}(\E)$. It follows that
\[
    \Psi_N(a) = \int_0^1 \sum_{j=-N}^N \left(1 - \frac{|j|}{N+1} \right)  e^{-2 \pi i j t} \sigma_{e^{2 
    \pi i t}}( a )  dt.
\]
That $\lim_{N \to \infty} \|\Psi_N(a) - a\| \to 0$ is similar to the case for crossed products (see for example \cite[Theorem VIII.2.2]{Dav:C*-ex}). The details are omitted.
\end{proof}

\begin{proposition} \label{prop:CondExPict}
Let $X$ be an infinite compact metric space, $\alpha : X\to X$ a  homeomorphism and $\mathscr{V} = [V, p, X]$ a line bundle. Let $Y \subset X$ be a non-empty closed subset, and set $\E = \Gamma(\mathscr{V}, \alpha)$ and $\E_Y := C_0(X \setminus Y) \E.$ For $n \in \mathbb{Z}$, set
\[\ Y_n = \left \{ \begin{array}{cc} \bigcup_{j=0}^{n-1} \alpha^j(Y) &  n >0 \\ \emptyset & n = 0 \\ \bigcup_{j=1}^{-n} \alpha^{-j}(Y)& n<0. \end{array} \right.
\]
Let $a \in \mathcal{O}(\E)$. Then $a \in \mathcal{O}(\E _Y)$ if and only if for every $m\in\mathbb{Z}_{>0}$ and $ \eta_1, \dots, \eta_m \in \E$, we have 
\begin{equation}\label{condition1}
\Phi(a \eta_1^* \cdots \eta_m^*) \in C_0(X \setminus Y_m),
\end{equation}
and
 \begin{equation}\label{condition3}
 \Phi(a \eta_1 \cdots \eta_m)  \in C_0(X \setminus Y_{-m}).
 \end{equation}
\end{proposition}

\begin{proof}  
We first establish the proposition when restricted to the dense subalgebra $\bigoplus_{l \in \mathbb{Z}} E_l$. Let $a \in \bigoplus_{l \in \mathbb{Z}} E_l$ and assume that $a$ satisfies \eqref{condition1} and \eqref{condition3}. Then $a = \sum_{l=-N}^N  \zeta_l$ for some  $\zeta_l \in E_l$,  $-N \leq l \leq N$. If $\eta_1, \dots, \eta_m \in \E$ are non-zero then $\zeta_l \eta_1^* \cdots \eta_m^*$ has degree zero if and only $l=m > 0$. It follows that 
\begin{equation}\label{computatenexpectation}
    \Phi(a \eta_1^*\cdots \eta_m^*) =  \zeta_m \eta_1^* \cdots \eta_m^*. 
    \end{equation}

Let $\{h_j :U_j \times \mathbb{C} \to \mathscr{V}|_{U_j}\}_{j=1,\dots, k}$ be an atlas for $\mathscr{V}$, and let $\{\gamma_j\}_{j=1,\dots,k}$ be a partition of unity subordinate to the open cover $\{U_j\}_{j=1,\dots, k}$. Define $\xi_j(x) = h_j(x, \gamma_j^{1/2}(x))$ for $1 \leq j \leq k$. Note that 
\[
 \{\xi_{j_1, \dots, j_l}:= \xi_{j_1} \otimes \xi_{j_2} \otimes \cdots \otimes \xi_{j_l} \mid 1\leq j_1, \dots, j_l\leq k\}
\] 
algebraically generates $\E^{\otimes l} \cong E_l$. In $\E^{\otimes m}$ we have 
\[ 
\zeta_m = \sum_{1\leq j_1, \dots, j_m\leq k} \xi_{j_1, \dots, j_m} \langle \xi_{j_1, \dots, j_m}, \zeta_m \rangle_{\E^{\otimes m}}.
\]
For every $1\leq j_1, \dots, j_m\leq k$, let $f_{j_1, \dots, j_m}=\zeta_m\xi_{j_m}^*\cdots\xi_{j_1}^*$.  By assumption and equation \eqref{computatenexpectation}, $f_{j_1, \dots, j_m}$  vanishes on $Y_m$. Fix $j_1, \dots, j_m  \in \{1, \dots, k\}$ and let $f := f_{j_1, \dots, j_m}$. Since we may write $f$ as the linear combination of positive elements, we may assume that $f$ is positive. Then, using \Cref{lem.SwitchSides} at the first step, we have
\begin{align*}
    \xi_{j_1, \dots, j_m} \langle \xi_{j_1, \dots, j_m}, \zeta_m \rangle_{\E^{\otimes m}} &=  \xi_{j_1, \dots, j_m}  \,_{\E^{\otimes m}}\langle \zeta_m, \xi_{j_1, \dots, j_m} \rangle \circ \alpha^m
    \\
    &=\xi_{j_1} \otimes \cdots \otimes \xi_{j_m} f \circ \alpha^m\\
    &=  (f^{1/m} \, \xi_{j_1}) \otimes \cdots \otimes (f^{1/m} \circ \alpha^{m-1} \xi_{j_m}).
\end{align*}
Each $f^{1/m}\circ \alpha^{j}$, $0 \leq j \leq m-1$, vanishes on $Y$, so $\xi_{j_1, \dots, j_m} \langle \xi_{j_1, \dots, j_m}, \zeta_m \rangle_{\E^{\otimes m}} \in \mathcal{O}(\E_Y)$. Since $\zeta_m$ is the sum of elements in $\mathcal{O}(\E_Y)$, we have $\zeta_m \in \mathcal{O}(\E_Y)$. Thus $\zeta_n \in \mathcal{O}(\E_Y)$ for every $n > 0$. 

Similarly, $\zeta_n \in \mathcal{O}(\E_Y)$ for every $n < 0$. Moreover, $\Phi(a) = \zeta_0 \in C(X) \subset \mathcal{O}(\E_Y)$. It follows that $a$ is the sum of elements of $\mathcal{O}(\E_Y)$. We conclude that $a \in \mathcal{O}(\E_Y)$.

Conversely, if $a \in \bigoplus_{l \in \mathbb{Z}} E_l \cap \mathcal{O}(\E_Y)$, then there exists $N \in \mathbb{Z}_{\geq 0}$ and $\zeta_l  \in E_l \cap \mathcal{O}(\E_Y)$ such that $a = \sum_{l = -N}^N \zeta_l$. If $l = 0$ then $\zeta_l \in C(X)$. Suppose $l > 0$. Since $\zeta_l \in \mathcal{O}(\E_Y)$, there are $K_l \in \mathbb{Z}_{>0}$ and $f_{1,k}, \dots, f_{l,k}  \in C_0(X\setminus Y)$ and $\xi_{1,k}, \dots,  \xi_{l,k} \in \E$, $1 \leq k \leq K_l$ such that  
\[ \zeta_l = \sum_{k=1}^{K_l} (f_{1,k} \xi_{1,k})  (f_{2,k} \xi_{2,k})\cdots (f_{l,k} \xi_{l,k}).
\]
Thus if $\eta_1, \dots, \eta_m \in \E$ we have that 

\begin{align*}
\Phi(a \eta_1^* \cdots \eta_m^*) &= 
   \zeta_m\eta_1^* \cdots \eta_m^*\\
   &= \sum_{k=1}^{K_m} (f_{1,k} \xi_{1,k})  (f_{2,k} \xi_{2,k}) (f_{3,k}\xi_{3,k}) \cdots (f_{m,k} \xi_{m,k}) \eta_1^* \cdots \eta_m^* \\
   &= \sum_{k=1}^{K_m}  f_{1,k}\, f_{2,k}  \circ \alpha^{-1} \, \xi_{1,k} \xi_{2,k} (f_{3,k}\xi_{3,k}) \cdots (f_{m,k}\xi_{m,k}) \eta_1^* \cdots \eta_m^*\\
   &= \sum_{k=1}^{K_m}  f_{1,k} \, f_{2,k}  \circ \alpha^{-1} \cdots f_{m,k} \circ \alpha^{1-m} \, \xi_{1,k} \xi_{2,k} \xi_{3,k}\cdots \xi_{m,k}\eta_1^* \cdots \eta_m^*.\\
\end{align*}

Since  $f_{1,k} f_{2,k} \, \circ \alpha^{-1} \cdots f_{m,k} \circ \alpha^{1-m} \in C_0(X \setminus Y_m)$  for every $1 \leq k \leq K_m$, while $\xi_{1,k} \xi_{2,k} \xi_{3,k}\cdots \xi_{m,k}\eta_1^* \cdots \eta_m^* \in C(X)$ for every $1 \leq k \leq K_m$, it follows that
\[ \Phi(a \eta_1^* \cdots \eta_m^*)  \in C_0(X \setminus Y_m).\]

Similar calculations show that $\Phi(a \xi_1 \dots \xi_m) \in C_0(X \setminus Y_{-m})$.

Now consider any $a \in \mathcal{O}(\E)$.  Suppose first that $a$ satisfies conditions \eqref{condition1} and \eqref{condition3}.  Let $a_N := \Psi_N(a)$  for every $N >0$, where $\Psi_N$ is defined in \Cref{lem:cesaro}. Then $a_N \in \bigoplus_{l \in \mathbb{Z}} E_l$; moreover, using the definition of $\Psi_N$ and the fact that $\Phi$ is a conditional expectation, it is easy to see that conditions \eqref{condition1} and \eqref{condition3} hold with $a_N$ instead of $a$ for all $N > 0$. By the previous case, this implies that $a_N \in \mathcal{O}(\E_Y)$ for all $N > 0$.  By \Cref{lem:cesaro}, $a_N$ converges to $a$ as $N \to \infty$, hence $a \in \mathcal{O}(\E_Y)$.

Conversely, suppose $a \in \mathcal{O}(\E_Y)$.  Since $\E_Y$ is a submodule of $\E$, it follows from the construction of $\mathcal{O}(\E_Y)$ that $\mathcal{O}(\E_Y) \cap \bigoplus_{l \in \mathbb{Z}} E_l$ is dense in $\mathcal{O}(\E_Y)$, so there exists a sequence $(a_N)_{N = 1}^\infty \subset \mathcal{O}(\E_Y) \cap \bigoplus_{l \in \mathbb{Z}} E_l$ such that $a_N \to a$.  By the previous case,  for any fixed $\eta_1, \dots, \eta_m \in \E$ we have $\Phi(a_N \eta_1^* \cdots \eta_m^*) \in C_0(X \setminus Y_m)$ 
and $\Phi(a_N \eta_1 \cdots \eta_m)  \in C_0(X \setminus Y_{-m})$ for all $N > 0$.  But then $a_N \to a$ implies that $\Phi(a \eta_1^* \cdots \eta_m^*) \in C_0(X \setminus Y_m)$ 
and $\Phi(a \eta_1 \cdots \eta_m)  \in C_0(X \setminus Y_{-m})$ as well, concluding the proof.
\end{proof}

\subsection{Centrally large subalgebras}

In an attempt to generalise the relationship between $C(X)\rtimes_\alpha \mathbb{Z}$ and the orbit-breaking subalgebra $\mathrm{C}^*
(C(X), C_0(X \setminus Y)u)$, Phillips introduced the abstract notion of a so-called \emph{large subalgebra} \cite{PhLarge}. Before presenting the definition, we require the notion of \emph{Cuntz subequivalence} for positive elements in a $\mathrm{C}^*$-algebra.  

Given a $\mathrm{C}^*$-algebra $A$ and two positive elements $a, b \in (\mathcal{K} \otimes A)_+$, we say $a$ is \emph{Cuntz subequivalent} to $b$, written $a \precsim b$, if there is a sequence $(r_n)_{n \in \mathbb{N}} \subset \mathcal{K} \otimes A$ such that $\| r_n b r_n^* - a \| \to 0$ as $n \to \infty$.  We say that $a$ is \emph{Cuntz equivalent} to $b$ if $a \precsim b$ and $b \precsim a$. We will routinely use the following facts about the Cuntz relation in a $\mathrm{C}^*$-algebra $A$. For (1) and (3) see \cite{Ror:uhfII}, while for (2) and (4) see Lemma 2.5(ii) and the discussion after Definition 2.3 of \cite{KirRor:pi}. See  Lemma 2.8(ii), Lemma 2.8(iii), and Lemma 2.9 in \cite{KirRor:pi} for (5), (6), and (7), respectively. 

\begin{lemma} \label{lem.CuntzResults}  Let $A$ be a $\mathrm{C}^*$-algebra.
\begin{enumerate}
\item If $A = C(X)$ for $X$ a compact Hausdorff space, then $f \precsim g$ if and only if the open support of $f$ is contained in the open support of $g$. 
\item For any $a, b \in A_+$, if $a \leq b$ then $a \precsim b$.
\item For any $\epsilon > 0$ and any $a, b \in A_+$  satisfying $\|a - b \| < \epsilon$, we have $(a - \epsilon)_+ \precsim b$.
\item For any $a \in A$, $a^*a \sim aa^*$.
\item For any $a,b \in A_+$, $a+b\precsim a\oplus b$.
\item For orthogonal elements $a,b \in A_+$, $a+b\sim a\oplus b$.
\item If $a_1,a_2, b_1, b_2 \in A_+$ satisfy $a_1\precsim a_2$ and $b_1\precsim b_2$, then $a_1\oplus b_1\precsim a_2\oplus b_2$.
\end{enumerate}
\end{lemma}

The definition of a large subalgebra is due to Phillips \cite{PhLarge}.  We will also require the refined definition of a \emph{centrally large subalgebra}, due to Archey and Phillips \cite{ArchPhil:SR1}.

\begin{definition}[{\cite[Definition 4.1]{PhLarge}, \cite[Definition 2.1]{ArchPhil:SR1}}] \label{defn.largesubalgebra}
Let $A$ be an infinite dimensional simple unital $\mathrm{C}^*$-algebra. A unital $\mathrm{C}^*$-subalgebra $B \subset A$ is $\emph{large}$ if, for every $m \in \mathbb{Z}_{>0}$, $a_1, \dots, a_m \in A$, $ \epsilon > 0$, $x \in A_+$ with $\| x \| = 1$ and every $y \in B_+ \setminus \{ 0 \}$, there are $c_1, c_2, \dots, c_m \in A$ and $g \in B$ such that 
\begin{enumerate}
\item \label{item.pvebound} $0 \leq g \leq 1$;
\item \label{item.closetoas} $\| c_j - a_j\| < \epsilon$,
\item \label{item.cutdown} $(1 - g) c_j \in B,$
\item \label{item.gleqCub} $g \precsim_B y$ and $g \precsim_A x$,
\item \label{defn.largesubalgebra.largeg} $\| (1 - g )x(1-g) \| > 1 - \epsilon$.
\end{enumerate}
If, in addition, $g$ can be chosen so that
\begin{enumerate}[resume]
    \item $\| g a_j - a_j g \| < \epsilon,$
    \end{enumerate} 
then we say that $B$ is \emph{centrally large}. We say that $B$ is \emph{stably large} in $A$ if $M_n(B)$ is large in $M_n(A)$ for every $n \in \mathbb{Z}_{>0}$.
\end{definition}

\begin{lemma} \label{thm.LocalUnits} Let $X$ be a compact metric space, $\alpha : X \to X$ a homeomorphism, $\mathscr{V} = [V, p, X]$ a line bundle over $X$ and $\E = \Gamma(\mathscr{V}, \alpha)$. Let  $U \subset X$ be an open subset such that $\mathscr{V}|_U$ is trivial, and $ F \subset C(X)$ a finite subset with $\supp(f) \subset U$ for every $f \in F$. Then there exists $\xi$ such that
\begin{equation*}
\xi \xi^* (f \circ \alpha^{-1}) = f \circ \alpha^{-1} \text{ and } \xi^* \xi f = f,
\end{equation*}
for every $f \in F$.
\end{lemma}

\begin{proof}
Since $\mathscr{V}|_U$ is trivial, there exists a chart $h : U \times \mathbb{C} \to \mathscr{V}|_U$. Let $W = \cup_{f \in F} \supp(f) \subset U$. Since $W$ is closed, there is a function  $\gamma \in C(X)$ satisfying $\gamma|_W = 1$ and $\gamma|_{X \setminus U} = 0$.  Set $\xi(x) := h(x, \gamma^{1/2}(x))$. Then $\xi \xi^* (f \circ \alpha^{-1}) = (\gamma\circ \alpha^{-1})(f \circ \alpha^{-1}) = f \circ \alpha^{-1} $ and $\xi^*\xi f = \gamma f = f$, for every $f \in F$.
\end{proof}

\begin{lemma} \label{lem.CuCompareFnTranslates} Let $X$ be a compact metric space, $\alpha : X \to X$ a homeomorphism, $\mathscr{V} = [V, p, X]$ a line bundle, and $\E = \Gamma(\mathscr{V}, \alpha)$.  Suppose that $Y \subset X$ is a non-empty closed subset. Let $\E_Y := C_0(X \setminus Y) \E$. Let $N \in \mathbb{Z}_{>0}$ and let $U \subset X$ be a non-empty open subset.
\begin{enumerate}
    \item Suppose that $\mathscr{V}|_{\alpha^n(U)}$ is trivial for every $0 \leq n \leq N-1$ 
    .  If $f \geq 0$ with $\supp f \subset U$ satisfies $\restrict{f}{\cup_{n=1}^{N} \alpha^{-n}(Y)} = 0$ then $f \sim_{\cpalgOfEY} f \circ \alpha^{-N}$.
    \item Suppose that $\mathscr{V}|_{\alpha^{-n}(U)}$ is trivial for every $1 \leq n \leq N$. 
    If $f \geq 0$ with $\supp f \subset U$ satisfies $\restrict{f}{\cup_{n=0}^{N-1} \alpha^n(Y)} = 0$  then $f \sim_{\cpalgOfEY} f \circ \alpha^{N}$.
\end{enumerate} 
\end{lemma}

\begin{proof}
We show (1). The proof that (2) holds is similar and left to the reader. By \Cref{thm.LocalUnits}, there are $\xi_1, \dots, \xi_N \in \E$ such that $\xi_n^* \xi_n (f \circ \alpha^{-n+1})= f \circ \alpha^{-n+1}$  and $\xi_n \xi_n^* (f \circ \alpha^{-n}) = f \circ \alpha^{-n}$, for any $1\leq n \leq N$.
Let $a =f^{1/2} \xi_1^* \dots \xi_N^*$. Since $f$ vanishes on $\bigcup_{n=1}^{N} \alpha^{-n}(Y)$ by \Cref{prop:CondExPict}, we have $a \in \cpalgOfEY$. Moreover
\begin{align*}
    aa^* &= f^{1/2} \xi_1^*\cdots \xi_N^*\xi_N \cdots \xi_1 f^{1/2} \\
    &=\xi_1^*\cdots \xi_{N-1}^* \xi_{N-1} \cdots \xi_1 (\xi_N^*\xi_N) \circ \alpha^{N-1} f \\
    &= \xi_1^*\cdots \xi_{N-1}^* \xi_{N-1} \cdots \xi_1 (\xi_N^*\xi_N (f \circ \alpha^{-N+1})) \circ \alpha^{N-1}\\
     &= \xi_1^*\cdots \xi_{N-1}^* \xi_{N-1} \cdots \xi_1 f\\
     &= \dots = \xi_1^*\xi_1 f = f.\\
     \end{align*}
while 
\begin{align*}
    a^*a &=\xi_N \cdots \xi_1 f \xi_1^*\cdots \xi_N^* \\
    &= \xi_N\cdots \xi_2\xi_1\xi_1^* \xi_2^*\cdots \xi_N^* f\circ \alpha^{-N} \\
    &=\xi_N\cdots \xi_2\xi_2^*\cdots \xi_N^* (\xi_1\xi_1^*)\circ\alpha^{-N+1} f\circ \alpha^{-N} \\
     &=\xi_N\cdots \xi_2\xi_2^*\cdots \xi_N^* ((\xi_1\xi_1^*)f\circ\alpha^{-1})\circ \alpha^{-N+1} \\
     &=\xi_N\cdots \xi_2\xi_2^*\cdots \xi_N^* f\circ\alpha^{-N} \\
     &= \dots = \xi_N\xi_N^* f\circ\alpha^{-N}  = f\circ\alpha^{-N}.
\end{align*}
Thus $f \sim_{\mathcal{O}(\E_Y) } f\circ \alpha^{-N}$.
\end{proof}

We say that a non-empty closed subset $Y \subset X$ \emph{meets every $\alpha$-orbit at most once} if $\alpha^n(Y) \cap Y = \emptyset$ for every $n \in \mathbb{Z} \setminus \{0\}$.

\begin{lemma} \label{thm.CuntzConstruction}
Let $X$ be an infinite compact metric space, $\alpha : X\to X$ a minimal homeomorphism and $\mathscr{V} = [V, p, X]$ a line bundle.  Let $Y \subset X$ be a non-empty closed subset meeting each $\alpha$-orbit at most once and such that for every $N \in \Z_{\geq 0}$ there exists an open set $W_N \supset Y$ for which $\mathscr{V}|_{\alpha^n(W_N)}$ is trivial whenever $-N \leq n \leq N$.  Let $\E := \Gamma(\mathscr{V}, \alpha)$ and $\E_Y := C_0(X \setminus Y) \E$. 

Let $U \subset X$ be a non-empty open subset and $m \in \mathbb{Z}$. Then there exist positive functions $f, g \in C(X)$ such that
\begin{align*}
& \supp(g) \subset U, \quad  f \precsim_{\mathcal{O}(\E_Y)} g,  \quad f|_{\alpha^m(Y)} = 1,  \text{ and } \quad 0 \leq f \leq 1.
\end{align*}
\end{lemma}
\begin{proof}
By \cite[Lemma 7.4]{PhLarge}, there exist an $l \in \Z_{>0}$, compact subsets $Y_1, \dots, Y_l \subset X$ and $n_1 < n_2 < \dots < n_l \in \Z_{>0}$ such that
\[ Y \subset \bigcup_{j=1}^l Y_j, \quad \bigcup_{j=1}^l \alpha^{n_j}(Y_j) \subset U, \quad  \text{and} \quad \alpha^{n_j}(Y_j) \cap \alpha^{n_k}(Y_k) = \emptyset, \quad 1 \leq j \neq k \leq l.\]
Assume $m = 0$. Let $W \supset Y$ be an open set such that $\mathscr{V}|_{\setWn{n}}$ is trivial for $0 \leq n \leq n_l-1$.  Choose open subsets $V_1, \dots, V_l \subset U$ such that $\alpha^{n_j}(Y_j) \subset V_j$ and $V_j \cap V_k = \emptyset$ for $1 \leq j \neq k \leq l$.   Define
\[ W_j := \alpha^{-n_j}(V_j) \cap \left( X  \setminus \bigcup_{n=1}^{n_l} \alpha^{-n}(Y) \right) \cap W. \]
Observe that $W_1, \dots, W_l$ are open and their union covers $Y$.  Thus there exist functions $0 \leq f_1, \dots, f_l \leq 1$ with $\supp(f_j) \subset W_j$, and such that $f := \sum_{j=1}^l f_j $ satisfies $f|_Y = 1$, $0 \leq f \leq 1$ and $\supp(f) \subset W \setminus \bigcup_{n=1}^{n_l} \alpha^{-n}(Y)$.  By \Cref{lem.CuCompareFnTranslates} (1), for each $j$, we have 
\[ f_j \sim_{\mathcal{O}(\E_Y)} f_j \circ \alpha^{-n_j}.\]
Define $g := \sum_{j = 1}^l f_j \circ \alpha^{-n_j}$.  Note that $\supp(f_j \circ \alpha^{-n_j}) \subset V_j$, and so $\supp(g) \subseteq U$ and the functions $f_j \circ \alpha^{-n_j}$ are pairwise orthogonal.  Hence
\[ f = \sum_{j=1}^l f_j \precsim_{\mathcal{O}(\E_Y)}  \bigoplus_{j=1}^l f_j \sim_{\mathcal{O}(\E_Y)}  \bigoplus_{j=1}^l  f_j \circ \alpha^{-n_j} \sim_{C(X)} g.\]
Thus the lemma holds for $m=0$.

Now suppose that $m > 0$. Let $K := \max\{m,n_l\}$, and choose $W \supset Y$ to be such that $\mathscr{V}|_{\alpha^n(W)}$ is trivial for every $0 \leq n \leq K-1$  and  $\overline{W} \cap \bigcup_{n=1}^K \alpha^{-n}(Y) = \emptyset$. As in the case $m=0$, we construct $f_0$ and $g$ such that $0 \leq f_0 \leq 1$, $f_0|_{Y} = 1$, $\supp(f_0) \subset W$,  $\supp(g) \subset U$ and $f_0 \precsim_{\mathcal{O}(\E_Y)} g$. Note that $\supp(f_0) \cap \bigcup_{n=1}^{n_l} \alpha^{-n}(Y) = \emptyset$. Let $f := f_0 \circ \alpha^{-m}$ and note that $\supp(f) \subseteq \setWn{m}$, and $\restrict{f}{\setYn{m}} = 1$. Moreover, by \Cref{lem.CuCompareFnTranslates} (1), $f_0 \sim_{\mathcal{O}(\E_Y)} f_0 \circ \alpha^{-m}$.  Hence
\[ f = f_0 \circ \alpha^{-m} \sim_{\cpalgOfEY} f_0 \precsim_{\cpalgOfEY} g,\]
which shows the lemma holds for $m > 0$ .

That the proof now holds for $m=-1$ is similar to the case $m=0$ applied to $\alpha^{-1}$ and using \Cref{lem.CuCompareFnTranslates} (2). The case $m < -1$ follows from a similar argument as for $m>0$. The details are omitted. 
\end{proof}
\begin{corollary} \label{thm.Cuntzlattice} 
Let $X$ be a compact metric space, $\alpha : X \to X$ a minimal homeomorphism, $\mathscr{V} = [V, p, X]$ a line bundle, and $\E = \Gamma(\mathscr{V}, \alpha)$. Let $Y \subset X$ be a non-empty closed subset meeting each $\alpha$-orbit at most once and such that for every $N \in \Z_{\geq 0}$ there exists an open set $W_N \supset Y$ for which $\mathscr{V}|_{\alpha^n(W_N)}$ is trivial whenever $-N \leq n \leq N$.   Let $a, b \in C(X)_+ \setminus \{0\}$.  Then there exists $f \in C(X)_+ \setminus \{0\}$ such that $f \precsim_\cpalgOfEY a$ and $f \precsim_\cpalgOfEY b$.
\end{corollary}
\begin{proof} Let $U = \{x \in X : a(x) \not= 0\} = a^{-1} ((0, \infty))$.  Using \Cref{thm.CuntzConstruction} there are $a_1, g \in C(X)_+$ such that $0 \leq a_1 \leq 1$, $\restrict{a_1}{Y} = 1$, $a_1 \precsim_\cpalgOfEY g$ and $\supp(g) \subset U$. Since $U \subset \supp(a)$, we have $g \precsim_{C(X)} a$.  Repeating this with $b$, there are $b_1, h \in C(X)_+$ such that $0 \leq b_1 \leq 1$, $\restrict{b_1}{Y} = 1$, $b_1 \precsim_\cpalgOfEY h \precsim_{C(X)} b$.

Since $\restrict{a_1}{Y} = 1$ and $\restrict{b_1}{Y} = 1$ we have $a_1 b_1 \not= 0$, so $f := a_1 b_1 \in C(X)_+$ is non-zero.  Since $f \leq a_1$ it follows that  $f \precsim_{C(X)} a_1 \precsim_\cpalgOfEY g \precsim_{C(X)} a$.  Similarly, $f \precsim_\cpalgOfEY b$.
\end{proof}
\begin{lemma}\label{thm.fromfntocx}
Let $X$ be an infinite compact metric space, $\alpha : X\to X$ a minimal homeomorphism and $\mathscr{V} = [V, p, X]$ a line bundle. For any $a \in \bigoplus_{n \in \Z} E_n$ and any $\epsilon >0$, there exists $f \in C(X)$ such that 
$0 \leq f \leq 1$, and 
$\|fa^* a f \| \geq \| \Phi (a^* a) \| - \epsilon$.
\end{lemma}
\begin{proof}
If $\|\Phi(a^* a)\| \leq \epsilon$ we can take $f=0$, so assume $\|\Phi (a^*a)\| > \epsilon$.
In that case, there exists an $N \in \Z_{\geq 0}$ such that $a^*a \in \bigoplus_{n = -N}^N E_n$.  Let $b_0 = \Phi(a^*a) \in C(X)$. Since $a^*a > 0$ and $\Phi$ is faithful, $b_0 > 0$.

Define $U := \{ x \in X \mid b_0 (x) > \|\Phi (a^*a)\| - \epsilon \}$, which is a non-empty open subset of $X$.  Choose $x_0 \in U$.  Since $\alpha$ is minimal, there exists a non-empty neighbourhood $V$ of $x_0$ such that $\setVn{i} \cap \setVn{j} = \emptyset$ for all $i\not=j$, $-N\leq i, j \leq N$. Taking the intersection with $U$ if necessary, we may assume $V \subset U$.  Choose $f \in C(X)$ such that $0 \leq f \leq 1$ with $\supp(f) \subset V$ and $f(x_0) = 1$. Since $\alpha^{n}(V) \cap V = \emptyset$ for any $n$ such that $0 < |n| \leq N$,  we have $f \xi f = f(f \circ \alpha^{-n}) \xi = 0$ for any $\xi \in E_n$.  Since $a^* a \in \bigoplus_{n = -N}^N E_n$ and $\Phi(a^*a) = b_0$, it follows that $f(a^*a)f = f b_0 f$. Finally, using $f(x_0)=1$ at the second step and the fact that $x_0 \in U$ at the third step, we have
\begin{equation*}
\|f b_0 f \| \geq f(x_0) b_0 (x_0) f(x_0) =b_0 (x_0) > \| \Phi (a^*a) \| - \epsilon, 
\end{equation*}
as required.
\end{proof}

The next lemma can be applied when $B = \mathcal{O}(\E)$ or $B = \mathcal{O}(\E_Y)$.

\begin{lemma} \label{thm.movecomparisonqntocx}
Let $X$ be an infinite compact metric space, $\alpha : X\to X$ a minimal homeomorphism and $\mathscr{V} = [V, p, X]$ a line bundle.  Set $\E := \Gamma(\mathscr{V}, \alpha)$.  Let $B \subset \cpalgOfE$ be a subalgebra for which $C(X) \subset B$ and $B \cap \oplus_{n \in \Z} E_n$ is dense in $B$.  Then, for any non-zero $a \in B_+$, there exists a non-zero $b \in C(X)_+$ such that $b \precsim_B a$.
\end{lemma}
\begin{proof}
Having established \Cref{thm.fromfntocx}, the proof is identical to that of \cite[Lemma 7.9]{PhLarge}. The details are omitted.
\end{proof}

\begin{proposition}\label{Almost_large_sub}
Let $X$ be an infinite compact metric space, $\alpha : X \to X$ a minimal homeomorphism and $\mathscr{V} = [V, p, X]$ a line bundle. Let $Y \subset X$ be a non-empty closed subset meeting each $\alpha$-orbit at most once and such that for every $N \in \Z_{\geq 0}$ there exists an open set $W_N$ containing $Y$ for which $\mathscr{V}|_{\alpha^n(W_N)}$ is trivial whenever $-N \leq n \leq N$.  Let $\E = \Gamma(\mathscr{V}, \alpha)$ and $\E_Y := C_0(X \setminus Y) \E.$ 

For every $m, M \in \mathbb{Z}_{>0}$, $c_1, \dots, c_m \in \bigoplus_{n \in \mathbb{Z}} E_n$, non-zero $x \in \mathcal{O}(\E)_+$, and non-zero $y \in \mathcal{O}(\E_Y)_+$, there is $g \in C(X)_+$ such that
\begin{enumerate}
\item \label{pvebound} $0 \leq g \leq 1$
\item  \label{cutdown} for $j = 1, \dots, m $ we have $(1 - g) c_j \in \mathcal{O}(\E_Y)$,
\item \label{gleqCub} $g \precsim_\cpalgOfEY y$ and $g \precsim_\cpalgOfE x$,
\item \label{item.ExtraCond} $g \sim_{\mathcal{O}(\E)} g \circ \alpha^k$ for every $-M \leq k \leq M$.
\end{enumerate}
\end{proposition}

\begin{proof}
Apply \Cref{thm.movecomparisonqntocx} twice (once for $x  \in \cpalgOfE$ and once for $y \in \cpalgOfEY$), followed by \Cref{thm.Cuntzlattice}, to find a function $f \in C(X)_+ \setminus \{0\}$ such that $f \precsim_\cpalgOfE x$ and $f \precsim_\cpalgOfEY y$.

Choose $N \in \Z_{>0}$ such that $c_j \in \bigoplus_{n = -N}^N E_n$ for all $ 1\leq j \leq m$. Let 
\[ K := N+M.\]

Let $U := \{ x \in X \mid f(x) \neq 0 \}$ denote the open support of $f$ and let $U_l \subset U$, for $l= -N, \dots, N$ be non-empty open sets satisfying $U_l \cap U_k = \emptyset$ for $k\neq l$.  

Choose an open set $V\subset X$ such that  $Y  \subset V$, $\alpha^{-K}(V), \dots , \alpha^K(V)$ are pairwise disjoint and $\mathscr{V}|_{\alpha^l(V)}$ is trivial for $-K \leq l \leq K$. 

For each $-N \leq l \leq N$, by  \Cref{thm.CuntzConstruction} there are functions $f_l, g'_l \in C(X)_+$ such that $\supp(f_l) \subset U_l$, $g'_l|_{\setYn{l}} = 1$, and $g'_l \precsim_{\mathcal{O}(\E_Y)} f_l$. Since $Y \subset V$ there are $g_l \leq g'_l$ with $\supp(g_l) \subset \alpha^l(V)$. Note that since $\alpha^{-K}(V), \dots , \alpha^K(V)$ are pairwise disjoint, this implies $g_l g_k = 0$ when $l \neq k$. Let $g = \sum_{l = -N}^N g_l$. Note that $0 \leq g \leq 1$, so $g$ satisfies (\ref{pvebound}). Since $\supp(f_l) \subset U_l$, and the $U_l$ are pairwise disjoint, we have
\begin{equation*}
g \sim_{C(X)} \bigoplus_{l=-N}^N g_l \leq  \bigoplus_{l=-N}^N g'_l \precsim_\cpalgOfEY \bigoplus_{l = -N}^N f_l \sim_{C(X)} \sum_{l = -N}^N f_l \precsim_{C(X)} f.
\end{equation*}
Thus since $f \precsim_\cpalgOfE x$ and $f \precsim_\cpalgOfEY y$, we see that $g$ satisfies (\ref{gleqCub}). 

Since $\mathscr{V}|_{\alpha^{k}(V)}$ is trivial for every $-K \leq k \leq K$, and $\supp(g_l \circ \alpha^j) \subset \alpha^{l-j}(V)$, $-M \leq j \leq M$, $-N \leq l \leq N$, by \Cref{thm.LocalUnits} there are $\xi_{l-j} \in \E$, $-K \leq l-j \leq K$ such that $\xi_{l-j}^*\xi_{l-j} \, g_l \circ \alpha^j = g_l \circ \alpha^j$ and $\xi_{l-j}\xi_{l-j}^* \, g_l \circ \alpha^{j-1} = g_l \circ \alpha^{j-1}$. Thus 
\begin{align*}
g_l\circ\alpha^j &=\xi_{l-j}^*\xi_{l-j} g_l\circ\alpha^j=(g_l\circ\alpha^j)^{1/2}\xi_{l-j}^*\xi_{l-j}(g_l\circ\alpha^j)^{1/2}\\
&\sim 
\xi_{l-j}(g_l\circ\alpha^j)\xi_{l-j}^*\\
&=(\xi_{l-j}\xi_{l-j}^*)(g_l\circ\alpha^j)\circ \alpha^{-1}\\
&=g_l\circ\alpha^{j-1}.\end{align*}
Therefore $g_l\sim g_l\circ\alpha^j$ for $-N\leq l\leq N$ and $-M\leq j\leq M$. Fix $j \in \{-M, \ldots, M\}$.  Since the sets $\alpha^{-N-j}(V), \cdots, \alpha^{N-j}(V)$ are pairwise disjoint, the functions $g_{-N} \circ \alpha^j, \dots, g_{N} \circ \alpha^j$ are pairwise orthogonal. Thus 
\begin{align*}
g = \sum_{l=-N}^N g_l &\sim_{C(X)} \bigoplus_{l=-N}^N   g_l \sim_{\mathcal{O}(\E)} \bigoplus_{l=-N}^N    g_l \circ \alpha^j  \sim_{\mathcal{O}(\E)} \sum_{l=-N}^N g_l \circ \alpha^j = g \circ \alpha^j.\\
\end{align*}
This shows (\ref{item.ExtraCond}).

Finally, since 
\[
g|_{\bigcup_{l=-N}^{N}\alpha^l(Y)} = \sum_{l=-N}^N \, g_l|_{\alpha^l(Y)} = \chi_{\cup_{l=-N}^{N}\alpha^l(Y)},
\]
we have $1 - g \in C_0(X \setminus \cup_{k = -N}^{N} \setYn{k})$.  Thus by \Cref{prop:CondExPict}, for any $n$, $-N\leq n \leq N$ and $\xi \in E_n$ we must have $(1 - g) \xi \in \cpalgOfEY$. By linearity, since $c_j \in \bigoplus_{n = -N}^N E_n$, we have $(1 - g)c_j \in \cpalgOfEY$ for all $j$, $1\leq j \leq m$, which is to say, (\ref{cutdown}) is satisfied.
\end{proof}

\begin{theorem}\label{large_sub}Let $X$ be an infinite compact metric space, $\alpha : X \to X$ a minimal homeomorphism and $\E= \Gamma(\mathscr{V}, \alpha)$, for a line bundle $\mathscr{V} = [V, p, X]$. Let $Y\subset X$ be a non-empty closed subset meeting each $\alpha$-orbit at most once and such that for every $N \in \Z_{\geq 0}$ there exists an open set $W_N \supset Y$ for which  $\mathscr{V}|_{\alpha^n(W_N)}$ is trivial whenever $-N \leq n \leq N$. Then 
$\mathcal{O}(\E_Y)$ is a large subalgebra of  $\mathcal{O}(\E)$. In fact, it is stably large.
\end{theorem}

\begin{proof} Let $m \in \mathbb{Z}_{>0}$, $a_1, \dots, a_m \in \mathcal{O}(\E)$, $\epsilon > 0$, $x \in \mathcal{O}(\E)_+$ with $\| x \| =1$, and non-zero $y \in \mathcal{O}(\E_Y)_+$. 

Choose  $c_1 , \dots c_m \in \bigoplus_{n \in \Z} E_n$ such that $\|a_j - c_j \| < \epsilon$ for $ j =1, \dots, m$. Then the $c_j$, for $1\leq j\leq m$, satisfy (\ref{item.closetoas}) of \Cref{defn.largesubalgebra}.

Apply \Cref{Almost_large_sub}  to $m \in \mathbb{Z}_{>0}$, $c_1, \dots, c_m \in \bigoplus_{n \in \mathbb{Z}}E_n$, $\epsilon>0$, $x \in \mathcal{O}(\E)_+$ and $y \in \mathcal{O}(\E_Y)_+$. This gives us $g \in C(X)_+$ satisfying (\ref{item.pvebound}), (\ref{item.cutdown}) and (\ref{item.gleqCub}) of \Cref{defn.largesubalgebra}.

Since $\mathcal{O}(\E)$ is finite when $\alpha$ is minimal and $\mathscr{V}$ is a line bundle (\Cref{cor:stabfin}), by  \cite[Proposition 4.5]{PhLarge}, condition (\ref{defn.largesubalgebra.largeg}) of \Cref{defn.largesubalgebra} can always be satisfied. Thus $\mathcal{O}(\E_Y)$ is a large subalgebra of $\mathcal{O}(\E)$.

For the final statement, by~\Cref{cor:stabfin} $\mathcal{O}(\E)$ is stably finite so any large subalgebra in $\mathcal{O}(\E)$ is stably large \cite[Corollary 5.8]{PhLarge}.
\end{proof}

\begin{corollary} \label{thm.largesubalgproperties}
Let $X$ be an infinite compact metric space, $\alpha : X \to X$ a minimal homeomorphism and $\E= \Gamma(\mathscr{V}, \alpha)$, for a line bundle $\mathscr{V} = [V, p, X]$. Let $Y\subset X$ be a non-empty closed subset meeting each $\alpha$-orbit at most once and such that for every $N \in \Z_{\geq 0}$ there exists an open set $W_N \supset Y$ for which  $\mathscr{V}|_{\alpha^n(W_N)}$ is trivial whenever $-N \leq n \leq N$. Then  $\mathcal{O}(\E_Y)$ is simple and there is an affine homeomorphism $T(\mathcal{O}(\E_Y)) \cong T(\mathcal{O}(\E))$.
\end{corollary}

\begin{proof}
This follows from~\Cref{19-10-18-3} together with the properties of large subalgebras of infinite dimensional simple unital \cstar-algebras.  Proposition 5.2 of \cite{PhLarge} gives us the simplicity, and Theorem 6.2 of \cite{PhLarge} gives the result about the traces. 
\end{proof}

To determine when $\mathcal{O}(\E_Y)$ itself falls within the scope of \Cref{ClassThm}, we need to determine that it has finite nuclear dimension. To show this, as well as to establish classification when $X$ is not necessarily finite dimensional (see \cite{FoJeSt:rsh}), we require that $\mathcal{O}(\E_Y)$ is centrally large in $\mathcal{O}(\E)$.

\begin{theorem} \label{thm.iscentrallylarge}
Let $X$ be an infinite compact metric space, $\alpha : X \to X$ a minimal homeomorphism and $\mathscr{V} = [V, p, X]$ a line bundle over $X$.  Let $Y \subset X$ be a non-empty closed subset meeting each $\alpha$-orbit at most once and such that for every $N \in \Z_{\geq 0}$ there exists an open set $W_N \supset Y$ for which $\mathscr{V}|_{\alpha^n(W_N)}$ is trivial whenever $-N \leq n \leq N$. 

Let $\E = \Gamma(\mathscr{V}, \alpha)$ and $\E_Y := C_0(X \setminus Y) \E$. Then $\mathcal{O}(\E_Y)$ is a centrally large subalgebra of $\mathcal{O}(\E)$.
\end{theorem}

\begin{proof} 
Let $m \in \Z_{> 0}$, $a_1, \ldots, a_m \in \cpalgOfE$, $x \in \cpalgOfE_+$ such that $\|x\| = 1$, non-zero $y \in \cpalgOfEY_+$, and $\epsilon > 0$ be given. Without loss of generality, assume that $\epsilon < 1$. As in the proof of \Cref{large_sub},  since $\cpalgOfE$ is finite (\Cref{cor:stabfin}), by \cite[Proposition 2.2]{ArchPhil:SR1} we can drop requirement (\ref{defn.largesubalgebra.largeg}) on $g$ from the definition of centrally large subalgebra, that is, it is not necessary to ensure explicitly that $\|(1 - g)x(1 - g)\| > 1 - \epsilon$.

Choose $c_j \in \bigoplus_{n = -N}^N E_n$ such that $\|a_j - c_j\| < \epsilon/3$ for $1\leq j\leq m$.  For each $j$ write $c_j = \sum_{k = -N}^N \xi_{j, k}$ for some $\xi_{j, k} \in E_k$. Since $a_1, \dots, a_m$ are arbitrary, we may assume that $N > 0$.  Define 
\begin{equation*}
K = \max \{  \|\xi_{j, k}\| \mid  -N \leq k \leq N, 1 \leq j \leq m\}.
\end{equation*}
Let 
\begin{equation*} 
\epsilon_0=\frac{\epsilon}{3N(N+1)K},
\end{equation*} 
and fix $M \in \Z_{>0}$ such that 
\[
\frac{1}{M} < \frac{\epsilon_0}{12}.  
\]
Since both $\cpalgOfE$ and $\cpalgOfEY$ are simple (\Cref{19-10-18-3}, \Cref{thm.largesubalgproperties}), by  \cite[Lemma 2.4]{PhLarge} there are $y_0 \in \cpalgOfEY_+ \setminus \{0\}$ such that $\bigoplus_{k=1}^{2M} y_0 \precsim_\cpalgOfEY y$, and $x_0 \in \cpalgOfE_+ \setminus \{0\}$ such that $\bigoplus_{k=1}^{2M-1} x_0 \precsim_\cpalgOfE x$.

Apply \Cref{Almost_large_sub} to $M-1$, $c_1, \ldots, c_m \in \bigoplus_{n = -N}^N E_n$, $y_0 \in \cpalgOfEY_+ \setminus \{0\}$, and $x_0 \in \cpalgOfE_+ \setminus \{0\}$ to find $g_0 \in C(X)_+$ such that
\begin{enumerate}[label=(\roman*)]
\item $0 \leq g_0 \leq 1$,
\item $(1 - g_0)c_j \in \cpalgOfEY$, for every $j=1, \cdots, m$,
\item $g_0 \precsim_\cpalgOfEY y_0$ and $g_0 \precsim_\cpalgOfE x_0$,
\item $g_0 \sim_{\cpalgOfE} g_0 \circ \alpha^k$ whenever $-M+1\leq k\leq M-1$.
\end{enumerate}

We will use $g_0$ to construct $g \in C(X)_+$ which satisfies (1)-(4) and (6) of Definition 6.7.

For each $k$, $0 \leq k \leq M-2$, construct $g_{k+1} \in C(X)_+$ from $g_{k}$ as follows. Choose $g_{k+1} \in C(X)_+$ such that $\supp(g_{k+1}) 
\subset \supp(g_k \circ \alpha^{-1} + g_k + g_k \circ \alpha)$, $0 \leq g_{k+1} \leq 1$ and $g_{k+1}(x) = 1$ for every $x \in X$ satisfying $(g_k \circ \alpha^{-1} + g_k + g_k \circ \alpha)(x) \geq \epsilon_0/12$. 
Then, for every $1\leq k\leq M-1$,
\[
\supp(g_k) \subset \bigcup_{l = -k}^{k} \supp(g_0 \circ \alpha^l), \text{ } 0 \leq g_k \leq 1 ,
\]
and for every $0 \leq l < k\leq M-1$ and $-1\leq r\leq 1$, 
 \begin{equation} \label{gkdef} \|g_{k} (g_l \circ \alpha^r) - g_l \circ \alpha^r  \| < \epsilon_0/6.
 \end{equation}
Let
\[
g:= 1- \left (1 - \frac{1}{M}\sum_{k=1}^{M-1} g_k \right )(1-g_0).
\]
Observe that $0 \leq g \leq 1$. We verify that $g$ and $c_1, \dots, c_m$ satisfy (1)--(4) and (6) of \Cref{defn.largesubalgebra}.

Let 
\[ h := \frac{1}{M} \sum_{k=1}^{M-1} g_k.\]
Then $g = g_0 + h - hg_0$, so
\begin{align*}
\|g - h\| &= \left\|g_0 - \frac{1}{M} \sum_{k = 1}^{M - 1} g_k g_0 \right\| \\
&= \left\|\frac{1}{M} g_0 + \frac{1}{M}\sum_{k = 1}^{M - 1} (g_0 - g_kg_0)\right\|\\
&\leq \frac{1}{M} + \frac{M-1}{M} \frac{\epsilon_0}{6} \\
&< \frac{\epsilon_0}{12} + \frac{\epsilon_0}{6}\\
&< \frac{\epsilon_0}{3}, 
\end{align*}
where in the third line we used \eqref{gkdef} and the fact that $0 \leq g_0 \leq 1$. It follows that
\begin{align*}
\|g - g \circ \alpha\| &\leq \|g - h\| + \|h - h \circ \alpha\| + \|h \circ \alpha - g \circ \alpha\|\\
&< \frac{2\epsilon_0}{3} + \| h - h \circ \alpha\|.  
\end{align*}
Note that $h - h \circ \alpha = \frac{1}{M} \sum_{k = 1}^{M - 1} (g_k - g_k \circ \alpha)$.  Since $0 \leq g_k \leq 1$ for all $0 \leq k \leq M-1$, we have
\begin{equation*}
\frac{1}{M} \sum_{k = 1}^{M - 1} (g_k g_{k - 1} \circ \alpha - g_k \circ \alpha) \leq \frac{1}{M} \sum_{k = 1}^{M - 1} (g_k - g_k \circ \alpha) \leq \frac{1}{M} \sum_{k = 1}^{M - 1} (g_k - g_{k - 1} g_k \circ \alpha),
\end{equation*}
so 
\[
\|h - h \circ \alpha\| \leq \max\left\{\left\|\frac{1}{M}\sum_{k = 1}^{M - 1} (g_k g_{k - 1} \circ \alpha - g_k \circ \alpha)\right\|, \left\|\frac{1}{M}\sum_{k = 1}^{M - 1} (g_k - g_{k - 1} g_k \circ \alpha) \right\|\right\}.
\]
We have
\begin{align*}
\lefteqn{\left\|\sum_{k = 1}^{M - 1} (g_k g_{k - 1} \circ \alpha - g_k \circ \alpha)\right\|}\\
&=  \left\|\sum_{k = 1}^{M - 1} g_k g_{k - 1} \circ \alpha - \sum_{k = 2}^{M} g_{k - 1} \circ \alpha \right\|\\
&\leq \|g_1 g_0 \circ \alpha\| + \|g_{M - 1} \circ \alpha\| + \sum_{k = 2}^{M - 1} \|g_k g_{k - 1} \circ \alpha - g_{k - 1} \circ \alpha \| \\
&\leq 1 + 1 + (M - 2) \frac{\epsilon_0}{6}
\end{align*}
by \eqref{gkdef}, whence 
\[
\left\|\frac{1}{M}\sum_{k = 1}^{M - 1} (g_k g_{k - 1} \circ \alpha - g_k \circ \alpha)\right\| < \frac{2}{M} + \frac{\epsilon_0}{6}.
\]
Moreover,
\begin{align*}
\lefteqn{\left\|\sum_{k = 1}^{M - 1} (g_k - g_{k - 1} g_k \circ \alpha)\right\|}
\\&= \left\|\sum_{k = 1}^{M - 1} (g_k \circ \alpha^{-1} - g_{k - 1} \circ \alpha^{-1} g_k)\right\| \\
&= \left\|\sum_{k = 2}^{M} g_{k - 1} \circ \alpha^{-1} - \sum_{k = 1}^{M - 1}g_{k - 1} \circ \alpha^{-1} g_k\right\| \\
&\leq \|g_{M - 1} \circ \alpha^{-1}\| + \|g_0 \circ \alpha^{-1} g_1\| + \sum_{k = 2}^{M - 1} \|g_k g_{k - 1} \circ \alpha^{-1} - g_{k - 1} \circ \alpha^{-1}\| \\
&\leq 1 + 1 + (M - 2)\frac{\epsilon_0}{6}
\end{align*}
again by \eqref{gkdef}, whence 
\[
\left\|\frac{1}{M}\sum_{k = 1}^{M - 1} (g_k - g_{k - 1} g_k \circ \alpha) \right\| < \frac{2}{M} + \frac{\epsilon_0}{6}.
\]
It follows that
\begin{equation*}
\|h - h \circ \alpha\| < \frac{2}{M} + \frac{\epsilon_0}{6} < \frac{\epsilon_0}{6} + \frac{\epsilon_0}{6} = \frac{\epsilon_0}{3},
\end{equation*}
and thus
\[ 
\| g - g \circ \alpha \| < \epsilon_0.
\]
From this we get
\begin{align*}
    \| g c_j - c_j g\| &\leq \sum_{k=-N}^N \|g \xi_{j,k} - \xi_{j,k} g\| \\
    &= \sum_{k=-N}^N \| \xi_{j,k} g \circ \alpha^k-  \xi_{j,k} g  \| \\
    &\leq  \sum_{k=-N}^N \|  g \circ \alpha^k- g  \| K \\
&= 2\sum_{k=1}^N\|g\circ\alpha^k-g\|K\\
    &=2K\sum_{k=1}^N \|  g \circ \alpha^k- g \circ \alpha^{k-1} + g \circ \alpha^{k-1} - \dots + g \circ \alpha -  g  \| \\
    &\leq2K\sum_{k=1}^N k\|g\circ\alpha-g\| \\
    &\leq 2K \epsilon_0 \sum_{k=1}^N k\\
    &\leq 2K\epsilon_0  \,\frac{N(N+1)}{2}\\
    &=N(N+1)K \epsilon_0\\
    &= \epsilon/3.
\end{align*}
It follows that $\|g a_j - a_j g\| < \epsilon$ for all $1 \leq j \leq m$. Thus (6) holds.

Next, observe that \begin{equation*}
(1 - g)c_j = \left(1 - \frac{1}{M}\sum_{k = 1}^{M - 1} g_k\right)(1 - g_0) c_j.
\end{equation*}
Thus that $(1 - g)c_j \in \cpalgOfEY$ follows from the fact that $g_k \in C(X) \subset \cpalgOfEY$ and that $(1 - g_0)c_j \in \cpalgOfEY$ by the choice of $g_0$.

Since each $g_k$ has support contained in $\bigcup_{l=-k}^k \supp(g_0 \circ \alpha^l)$, we see that $\supp(g) \subset \bigcup_{k = -M + 1}^{M - 1} \supp(g_0 \circ \alpha^k)$.  Furthermore, $g_0$ satisfies $g_0 \sim_\cpalgOfE g_0 \circ \alpha^k$ for $-M+1\leq k\leq M-1$. Thus 
\[
g \precsim_{C(X)} \sum_{k = - M + 1}^{M - 1} g_0 \circ \alpha^{k} \precsim_{C(X)} \bigoplus_{k = - M + 1}^{M - 1} g_0 \circ \alpha^{k} \sim_\cpalgOfE \bigoplus_{k = - M + 1}^{M - 1} g_0.
\]

By construction $g_0 \precsim_\cpalgOfE x_0$.  It immediately follows that
\begin{equation*}
g \precsim_\cpalgOfE \bigoplus_{k=-M + 1}^{M - 1} x_0 \precsim_\cpalgOfE x.
\end{equation*}
Finally, we show $g \precsim_\cpalgOfEY y$. First, from \Cref{large_sub}, $\cpalgOfEY$ is stably large in $\cpalgOfE$.  This allows us to appeal to \cite[Lemma 6.5]{PhLarge} with $a = g$, $b = y$ elements of $\cpalgOfEY_+$ and $c = \bigoplus_{k = - M + 1}^{M - 1} g_0$  and $x = g_0$  elements of $(\mathcal{K} \otimes \mathcal O({\E)})_+$.
Recall that
\begin{equation*}
g \precsim_\cpalgOfE \bigoplus_{k = - M + 1}^{M - 1} g_0.
\end{equation*}
By construction $g_0 \precsim_{\cpalgOfEY} y_0$ and hence $g_0 \precsim_{\cpalgOfE} y_0$, so we have that
\begin{equation*}
g_0 \oplus \bigoplus_{k = - M + 1}^{M - 1} g_0 = \bigoplus_{k=1}^{2M} g_0 \precsim_{\cpalgOfE} \bigoplus_{k=1}^{2M} y_0 \precsim_{\cpalgOfE} y.
\end{equation*}
It follows from \cite[Lemma 6.5]{PhLarge} that $g \precsim_{\cpalgOfEY} y$, showing that (4) holds. Thus $\mathcal{O}(\E_Y)$ is centrally large in $\mathcal{O}(\E)$.
\end{proof}

\begin{theorem} \label{cor.OBNucdim} Let $X$ be an infinite compact metric space with $\dim(X) < \infty$, $\alpha : X \to X$ a minimal homeomorphism, $\mathscr{V} = [V, p, X]$ a line bundle. Suppose that $Y\subset X$ is a non-empty closed subset meeting each $\alpha$-orbit at most once and such that for every $N \in \Z_{\geq 0}$ there exists an open set $W_N \supset Y$ for which  $\mathscr{V}|_{\alpha^n(W_N)}$ is trivial whenever $-N \leq n \leq N$. Then $\mathcal{O}(\E_Y)$ has nuclear dimension at most one.
\end{theorem}

\begin{proof}
Since $\mathcal{O}(\E_Y)$ and $\mathcal{O}(\E)$ are nuclear by \cite[Corollary 7.4]{Katsura2004} and $\mathcal{O}(\E_Y)$ is centrally large in $\mathcal{O}(\E)$, we have that $\mathcal{O}(\E_Y)$ is $\mathcal{Z}$-stable if $\mathcal{O}(\E)$ is $\mathcal{Z}$-stable \cite[Corollary 3.5]{ArBkPh-Z}. By \Cref{fin_nuc_dim}, $\mathcal{O}(\E)$ has finite nuclear dimension, and hence is $\mathcal{Z}$-stable by \cite{Win:Z-stabNucDim}. Thus $\mathcal{O}(\E_Y)$ is also $\mathcal{Z}$-stable and hence has nuclear dimension at most one by \cite[Theorem B]{CETWW}.
\end{proof}

We can now extend the classification result of \Cref{class}. Let $\mathcal{C}_{\mathrm{ob}}$ denote the class of $\mathrm{C}^*$-algebras of the form  $\mathcal{O}(C_0(X \setminus Y)\Gamma(\mathscr{V}, \alpha))$ where $X$ is a finite dimensional infinite compact metric space, $\mathscr{V} = [V, p, X]$ is a line bundle, $\alpha : X \to X$ a minimal homeomorphism and $Y \subset X$ is a non-empty closed subset meeting every $\alpha$-orbit at most once such that for every $N \in \mathbb{Z}_{>0}$ there exists an open set $W_N \supset Y$ such that $\mathscr{V}|_{\alpha^n(W_N)}$ is trivial for every $-N \leq n \leq N$. Again by \cite[Proposition 8.8]{Katsura2004}, since $C(X)$ is commutative, $\mathcal{O}(C_0(X \setminus Y)\Gamma(\mathscr{V}, \alpha))$ satisfies the UCT. Thus \Cref{ClassThm} gives us the following.

\begin{theorem}\label{class2}
 Suppose that $A, B \in \mathcal{C} \cup \mathcal{C}_{\mathrm{ob}}$ and 
	\[ \psi : \Ell(A) \to \Ell(B) \]
	is an isomorphism. Then there exists a $^*$-isomorphism 
	\[ \Psi : A \to B,\]
which is unique up to approximate unitary equivalence and satisfies $\Ell(\Psi) = \psi$.
\end{theorem}

\section{Applications} \label{sec.examples}

In this section we provide some interesting examples of Cuntz--Pimsner algebras that fall within the classification theorems of the previous sections. We begin by noting that both classes $\mathcal{C}$ and $\mathcal{C}_{ob}$ are quite rich. Let $X$ be an infinite compact metric space with $\dim(X) < \infty$. If $\mathcal{T} = [X \times \mathbb{C}, p, X]$ is the trivial line bundle, then as observed in \Cref{CPexamples} (3), any minimal homeomorphism $\alpha : X \to X$ gives us $\mathcal{O}(\Gamma(\mathcal{T}, \alpha)) \cong C(X) \rtimes_{\alpha} \mathbb{Z}$, so $\mathcal{C}$ contains all crossed products by minimal homeomorphisms on infinite compact metric spaces with $\dim(X) < \infty$. Similarly, $\mathcal{C}_{\mathrm{ob}}$ contains any orbit-breaking subalgebra of the form $\mathrm{C}^*(C(X), C_0(X \setminus Y)u) \subset C(X) \rtimes_\alpha \mathbb{Z}$, where $u$ is the unitary implementing $\alpha$, for any non-empty closed subset $Y \subset X$ satisfying $\alpha^n(Y) \cap Y = \emptyset$ for every $n \in \mathbb{Z} \setminus \{0\}$.

\subsection{The general outlook}
 Recall that the \emph{Picard group} $\mathrm{Pic}(A)$ of a $\mathrm{C}^*$-algebra $A$ consists of the isomorphism classes of Hilbert $\mathrm{C}^*$-bimodules over $A$ which are both left and right full, where the group multiplication is given by the internal tensor product \cite{BrownGreenRieffel1977}. A Hilbert $A$-bimodule $\E$ is \emph{symmetric} if the left and right multiplication by $A$ is the same, that is, $a \xi = \xi a$ for every $\xi \in \E$ and every $a \in A$. The subgroup ${\rm CPic}(A)$ consisting of isomorphism classes of left and right full symmetric Hilbert $\mathrm{C}^*$-bimodules is called the \emph{classical Picard group} of $A$. For an automorphism $\alpha : A \to A$, let $A_{\alpha}$ denote the Hilbert $A$-bimodule which as a right Hilbert $A$-module is given by $A$ and with left Hilbert $A$-module structure given by $a \cdot \xi = \alpha^{-1}(a)\xi $  and $_{A_\alpha} \langle \xi, \eta \rangle = \alpha(\xi \eta^*)$ . As shown in \cite[Proposition 3.1]{BrownGreenRieffel1977}, there is an anti-homomorphism  $\psi\colon\mathrm{Aut}(A)\to \mathrm{ Pic}(A)$ given by $\psi(\alpha)=A_{\alpha^{-1}}$ such that the sequence 
\[
 1\to \mathrm{Gin}(A)\to \mathrm{Aut}(A)\to \mathrm{Pic}(A)
 \]
 is exact, where $\mathrm{Gin}(A)$ is the group of generalised inner automorphisms which consists of all automorphisms of the form $\mathrm{Ad}(u)$ for a unitary $u$ in the multiplier algebra of $A$. Thus $\phi\colon\mathrm{Aut}(A)\to \mathrm{Pic}(A)$ given by $\phi(\alpha)=A_\alpha$ is a group homomorphism  with kernel $\mathrm{Gin}(A)$. If $A=C(X)$ is a commutative $\mathrm{C}^*$-algebra, then the group $\mathrm{Gin}(A)$ is trivial and $\mathrm{Aut}(C(X)) \cong \mathrm{Homeo}(X)$, hence each $\alpha\in \mathrm{Homeo}(X)$ can be identified with $C(X)_\alpha$ in $\mathrm{Pic}(C(X))$, where by abuse of notation we have used $\alpha$ to also denote the automorphism $f \mapsto f \circ \alpha$. Moreover, the Picard group $\mathrm{Pic}(C(X))$ is known to be isomorphic to the semidirect product group $\mathrm{CPic}(C(X))\rtimes  \mathrm{Homeo}(X)$ where the action of $ \mathrm{Homeo}(X)$ is given by $\alpha\cdot M = C(X)_\alpha \otimes M\otimes C(X)_{\alpha^{-1}}$. (This observation was made in \cite[Section 3]{BrownGreenRieffel1977}, see also \cite[Theorem 1.12 and Lemma 2.1]{AbadieExel1997}. Note that there the left and right actions are reversed.) Therefore every right and left full Hilbert $\mathrm{C}^*$-bimodule $M$ over $C(X)$ is isomorphic to $M^s\otimes C(X)_\alpha$  for some $\alpha\in\mathrm{Aut}(C(X))$ where $M^s$ denotes the symmetrisation of $M$ (see \cite[Definition 1.5 and proof of Theorem 1.12]{AbadieExel1997}). As a right and left Hilbert $C(X)$-module, $M^s$ is of course just the module of sections of some line bundle $\mathscr{V}$ over $X$ (in our notation $M_s \cong \Gamma(\mathscr{V},\text{id}_X)$), so we have that $\Gamma(\mathscr{V}, \alpha) \cong M$, as we expect from \Cref{prop.CharOfFull}. In particular, the group of Hilbert $C(X)$-bimodules is strictly larger than $\mathrm{Homeo}(X)$ whenever $X$ admits a non-trivial line bundle. Thus we have many examples of Cuntz--Pimsner algebras which fall within our classification results and are not a priori given by crossed products of minimal homeomorphisms. (Note that, thanks to the classification machinery, one can look at the Elliott invariant of a given Cuntz--Pimsner algebra to determine if it is isomorphic to \emph{some} crossed product by a minimal homeomorphism, even if this is far from obvious otherwise.) Some examples follow.

1.  The only two-dimensional manifolds admitting minimal homeomorphisms are finite unions of the two-torus $\mathbb{T}^2$ and the Klein bottle. All of these spaces admit non-trivial line bundles.  By \cite[Theorem 8.6]{Katsura2004} (also \cite[Theorem 4.9]{Pimsner1997}), for a Cuntz--Pimsner algebra $\mathcal{O}(\Gamma(\mathscr{V}, \alpha))$ there is a six-term exact sequence
\[
\xymatrix{ K^0(X) \ar[rr]^{\id_* - [\Gamma(\mathscr{V}, \alpha)]_*} & & K^0(X) \ar[rr] & & K_0(\mathcal{O}(\Gamma(\mathscr{V}, \alpha))\ar[d] \\
K_1(\mathcal{O}(\mathscr{V}, \alpha)) \ar[u] & & K^1(X) \ar[ll] & & K^1(X). \ar[ll]_{\id_* - [\Gamma(\mathscr{V}, \alpha)]_*}\\ }
\]
Thus if $X$ is either a finite union of 2-tori or the Klein bottle, $\alpha : X \to X$ a minimal homeomorphism and $\mathscr{V}$ a non-trivial line bundle, in general the maps $\id - [\Gamma(\mathscr{V}, \alpha)]$ and $\id - \alpha$ will not coincide and hence the $K$-theory of $\mathcal{O}(\Gamma(\mathscr{V}, \alpha))$ and $C(X)\rtimes_{\alpha} \mathbb{Z}$ will differ, implying that for $\E = \Gamma(\mathscr{V}, \alpha)$, the $\mathrm{C}^*$-algebras  $\mathcal{O}(\E)$ and $C(X)\rtimes_\alpha \mathbb{Z}$ are not isomorphic.  If $x \in X$ is any point, then the orbit-breaking algebra $\mathcal{O}(\E_{\{x\}})$ is easily seen to satisfy the requirements of \Cref{thm.iscentrallylarge}, and hence is also covered by the classification theorem of the previous section.  In the case of  $X=\mathbb{T}^2$, certain Cuntz--Pimsner algebras $\mathcal{O}_{C(\mathbb{T}^2)}(\Gamma(\mathscr{V}, \alpha))$ correspond to quantum Heisenberg manifolds. These are discussed in greater detail in Example~\ref{quantumH}, below. 

2. In \cite{DPS:MinDynK}, many examples of minimal homeomorphisms on infinite compact metric spaces of finite dimensional metric spaces are constructed. In \cite{DPS:OrbitBreaking}, various orbit-breaking constructions are also given, leading to a wide range of $K$-theory for the corresponding $\mathrm{C}^*$-algebras. Consequently, any space admitting non-trivial vector bundles provides examples of Cuntz--Pimsner algebras that are not a priori given by crossed products. In that paper, a main motivation is to determine the range of the Elliott invariant for minimal principal \'etale groupoid $\mathrm{C}^*$-algebras that arise from minimal dynamical systems. It remains an open question whether the Elliott invariant can be exhausted by such constructions. However, related work by X. Li \cite{Li:ClassifiableCartan} shows that \emph{twisted} principal \'etale groupoid constructions (based on inductive limit, rather than dynamical, constructions) do indeed exhaust the invariant. It is unclear whether twists are needed. Nevertheless, allowing for the more general Cuntz--Pimsner algebras of $C(X)$-bimodules and their orbit-breaking subalgebras certainly makes the possibility of exhausting the invariant appear more likely. We note that many of the constructions in \cite{DPS:OrbitBreaking} require infinite dimensional spaces. The results of this paper can be extended to the case where the minimal dynamical system $(X, \alpha)$ has mean dimension zero \cite{FoJeSt:rsh}. This does not require that $\dim(X) < \infty$.

\subsection{Quantum Heisenberg manifolds} \label{quantumH} We discuss the particular example of quantum Heisenberg manifolds in a bit more detail. Quantum Heisenberg manifolds were constructed by Rieffel in the late 1980's as deformation-quantisations of classical Heisenberg manifolds (see \cite{Rieffel1989}). As shown by Abadie, Eilers and Exel in \cite[Example 3.3]{AEE:Cross}, a quantum Heisenberg manifold is an example of a crossed product by a Hilbert bimodule over $\mathbb{C}(\mathbb{T}^2)$ which cannot be realised as the crossed product of $\mathbb{C}(\mathbb{T}^2)$ by a homeomorphism nor by a partial homeomorphism. For any  pair of real numbers $\mu, \nu$, let $\lambda$ be the action of $\mathbb{Z}$ on $\mathbb{R}\times\mathbb{T}$ given by 
\[\lambda_p(x,y)=(x-2p\mu, y-2p\nu),\]
where $p\in\mathbb{Z}$, $x\in\mathbb{R}$, $y\in\mathbb{T}$. In what follows, we will implicitly use the usual identification of $\mathbb{R}$ with $\mathbb{T}$ given by $y\rightarrow e^{2\pi iy}$ when writing the second coordinate of any map defined on $\mathbb{R}\times\mathbb{T}$. Observe that $\lambda$ is minimal whenever $1,\mu,\nu$ are rationally independent. 

Consider the crossed product $B=C_b(\mathbb{R}\times\mathbb{T})\rtimes_{\lambda}\mathbb{Z}$, and for a positive integer $c$ define an action $\rho$ of $\mathbb{Z}$ on $B$ by 
\[\rho_m(\Phi)(p,x,y)=e^{-2\pi icmp(y-p\nu)} \Phi(p,x+m,y),\]
for $\Phi\in C_c(\mathbb{Z}\times\mathbb{R}\times\mathbb{T})$, $x\in\mathbb{R}$, $y\in\mathbb{T},$ and $m,p\in\mathbb{Z}$. 
  The \emph{quantum Heisenberg manifold} $D_{\mu,\nu}^{c}$ is defined as the $\mathrm{C}^*$-subalgebra of the multiplier algebra $\mathcal{M}(C_0(\mathbb{R}\times\mathbb{T})\rtimes_{\lambda}\mathbb{Z})$ generated by the functions $\Phi\in  C_c(\mathbb{Z},C_b(\mathbb{R}\times\mathbb{T}))$ such that $\rho_m(\Phi)=\Phi$, that is,
 \[D_{\mu.\nu}^{c} = \overline{\{\Phi\in C_c(\mathbb{Z}, C_b(\mathbb{R} \times \mathbb{T})) \mid \Phi(p, x+m, y)=e^{-2\pi icmp(y-p\nu)} \Phi(p,x,y)\}}.\]
There exists an action $\gamma$ of $\mathbb{T}$ on $D_{\mu,\nu}^{c}$, defined by
\[\gamma_z(\Phi)(p,x,y)=z^p\Phi(p,x,y),\]
for all $\Phi\in D_{\mu,\nu}^{c}$. As in \cite[Example 3.3]{AEE:Cross}, for $n\in\mathbb{Z}$ the $n$-th spectral subspace $D_n$ of $D_{\mu,\nu}^{c}$ is given by the functions $\delta_n\in C_c(\mathbb Z, C_b(\mathbb R\times \mathbb T))$ (see \cite{AbadieExel1997}), that is,
\[D_n=\{f\delta_n|\, f\in C_b(\mathbb{R}\times\mathbb{T}), f(x+1,y)=e^{-2\pi icn(y-n\nu)}f(x,y)\}.\]

A Hilbert $C(\mathbb{T}^2)$-bimodule structure on $D_1$ is given by the following operations: for $g\in C(\mathbb{T}^2)$, $f\delta_1,h\delta_1\in D_1$,  
\[\begin{array}{l}
              (f\delta_1\cdot g)(x,y)=f(x,y)g(x,y)\delta_1\\
              (g\cdot f\delta_1)(x,y)=f(x,y)g(x-2\mu,y-2\nu)\delta_1\\
              \langle f\delta_1,h\delta_1\rangle_\E(x,y)=\overline{f(x,y)}h(x,y) \\
              _\E\langle f\delta_1,h\delta_1\rangle(x,y)=(f\overline{h})\circ\lambda^{-1}(x,y),
\end{array}\]
for all $x\in\mathbb{R}$, and $y\in\mathbb{T}$.
Moreover, $D_1D_1^\ast=D_1^\ast D_1=C(\mathbb{T}^2)$, which implies that $D_1$ is a left and right full $C(\mathbb{T}^2)$-bimodule, and hence is of the form $\Gamma(\mathscr{V}, \alpha)$ for a line bundle $\mathscr{V}  = [V, p, \mathbb{T}^2]$ and homeomorphism $\alpha :  \mathbb{T}^2 \to \mathbb{T}^2$.

Since the action $\gamma$ is saturated (see \cite{AEE:Cross}) and the fixed point algebra $D_0$ is isomorphic to $C(\mathbb{T}^2)$,  we have  \begin{equation}\label{quantumheisenbergmanifold}
\mathcal{O}(D_1)\cong C(\mathbb{T}^2)\rtimes_{D_1}\mathbb{Z}\cong D_{\mu,\nu}^{c}.
\end{equation}
It follows from~\Cref{fin_nuc_dim} and~\Cref{dichotomy} that, when $1,\mu,\nu$ are rationally independent, $D_{\mu,\nu}^{c}$ is a simple, separable, unital, $\mathrm{C}^\ast$-algebra with finite nuclear dimension and stable rank one. In particular, such quantum Heisenberg manifolds are completely classified by the Elliott invariant.

Let us say a bit more about the module $\E$ corresponding to a given quantum Heisenberg manifold. If $X=\mathbb T^2$, then then by \cite[Lemma 2.1]{AbadieExel1997} we have $\mathrm{Pic}(C(\mathbb T^2))\cong \mathbb Z\rtimes_\delta \mathrm{Aut}(C(\mathbb T^2))$, where $\delta_\beta(c)=\mathrm{det}\beta_*\cdot c$, for $\beta\in\mathrm{Aut}(C(\mathbb{T}^2))$ and $c\in\mathbb{Z}$. Here $\beta_*$ denotes the standard automorphism of $K_0(C(\mathbb{T}^2))\cong\mathbb{Z}^2$, viewed as an element of $GL_2(\mathbb{Z})$, and we see that if $\E = \Gamma(\mathscr{V}, \alpha)$ is a left and right full Hilbert $\mathrm{C}^*$-bimodule over $C(\mathbb T^2)$, then 
\[ 
\E \cong M^c\otimes C(\mathbb T^2)_\alpha\ \ \text{ for some }c\in \mathbb Z\ \text{and }\alpha\in \mathrm{Homeo}(\mathbb T^2),
\]
where $M^c=\{f\in C_b(\mathbb R\times \mathbb T) \mid  
f(x+1,y)=e^{-2\pi i cy}f(x,y)\}$ is the symmetric Hilbert $\mathrm{C}^*$-bimodule over $C(\mathbb T^2)$ for pointwise action and inner products given by $\langle f,g\rangle_L=f\bar{g}$ and $\langle f,g\rangle_R= \bar{f}g$. In particular, $c$ allows us to recover $\mathscr{V}$.

Also, for each $c\in \mathbb Z$, $M^c\otimes C(\mathbb T^2)_\alpha\cong M^c\otimes C(\mathbb T^2)_\beta$ if and only if $\alpha=\beta$ (see the proof of \cite[Theorem 1.12 ]{AbadieExel1997}). In particular, $D_1=M^c\otimes C(\mathbb T^2)_{\alpha_{\mu,\nu}}$ where $\alpha_{\mu,\nu}$ is the homeomorphism of $\mathbb T^2$ given by $\alpha_{\mu,\nu}(x,y)=(x+2\mu,y+2\nu)$, and by \eqref{quantumheisenbergmanifold}
\[
 D_{\mu,\nu}^c \cong C(\mathbb T^2)\rtimes_{M^c\otimes C(\mathbb T^2)_{\alpha_{\mu,\nu}}} \mathbb Z\cong \mathcal O(M^c\otimes C(\mathbb T^2)_{\alpha_{\mu,\nu}}).
 \]

Since $K_0(D_{\mu,\nu}^c)=\mathbb Z^3\oplus \mathbb Z_c$ \cite[Theorem 3.4]{Abadie1995}, we know that if $c\neq c'$, then $D_{\mu,\nu}^c\not\cong D_{\mu',\nu'}^{c'}$. Since $c$ determines the line bundle, this means that $\mathcal{O}(\Gamma(\mathscr{V}, \alpha)) \cong D_{\mu,\nu}^c \not\cong D_{\mu',\nu'}^{c'} \cong \mathcal{O}(\Gamma(\mathscr{W}, \beta)) $ if $\mathscr{V} \not\cong \mathscr{W}$.

On the other hand, Theorem 2.2 of \cite{AbadieExel1997} says that if $(\mu,\nu)$ and 
$(\mu',\nu')$ belong to the same orbit under the usual action of $GL_2(\mathbb Z)$ on $\mathbb T^2$, then $D_{\mu,\nu}^c\cong D_{\mu',\nu'}^c$. So the Cuntz--Pimsner algebras of non-isomorphic right and left full Hilbert $C(\mathbb{T}^2)$-bimodules corresponding to non-trivial line bundles can be isomorphic. This is not so surprising as already in the trivial line bundle case, which is to say, the usual crossed product by a minimal homeomorphism, there exist examples of different homeomorphisms giving rise to isomorphic $\mathrm{C}^*$-algebras, see for example \cite{Str:XxSn}.

Is every Cuntz--Pimsner algebra of a right and left full, minimal, non-periodic Hilbert $C(\mathbb{T}^2)$-bimodule isomorphic to a simple quantum Heisenberg manifold? The homeomorphisms $\alpha_{\mu,\nu}$ are always homotopic to the identity via $\alpha_t(x,y):=(x+2t\mu, y+2t\nu)$. However, it is known that the mapping class group of $\mathbb T^2$, the group of homotopy classes of orientation preserving homeomorphisms on $\mathbb T^2$, is isomorphic to the group $SL_2(\mathbb Z)$ \cite[Theorem 2.5]{FarbMargalit:Primer}. Thus every homeomorphism $\alpha_{\mu,\nu}$ corresponds, under this isomorphism, to the identity of $SL_2(\mathbb Z)$ (see \cite[Theorem 2.5]{FarbMargalit:Primer}). The existence of a minimal skew product diffeomorphism of $\mathbb T^2$ of the form $\theta(x,y)=(x+\mu, cx-y+f(x))$ which is \emph{not} homotopic to the identity (in fact, not even orientation-preserving) was proved in \cite{dosSUrz2009}.
It follows that if $\theta$ is not homotopic to the identity then the $K$-theory of $\mathcal{O}(\Gamma(\mathscr{V}, \theta))$ will not be the same as that of any quantum Heisenberg manifold of the form $\mathcal{O}(\Gamma(\mathscr{V}, \beta))$, and so $ \mathcal{O}(\Gamma(\mathscr{V}, \theta))$ cannot be isomorphic to 
$\mathcal{O}(\Gamma(\mathscr{V}, \beta))$. However this does not rule out an isomorphism with a quantum Heisenberg manifold with respect to a line bundle not isomorphic to $\mathscr{V}$.


\begin{thebibliography}{10}

\bibitem{Abadie1995}
B.~Abadie.
\newblock Generalized fixed-point algebras of certain actions on crossed
  products.
\newblock {\em Pacific J. Math.}, 171(1):1--21, 1995.

\bibitem{AEE:Cross}
B.~Abadie, S.~Eilers, and R.~Exel.
\newblock Morita equivalence for crossed products by {H}ilbert
  \mbox{{$\mathrm{C}^*$}-bimodules}.
\newblock {\em Trans. Amer. Math. Soc.}, 350(8):3043--3054, 1998.

\bibitem{AbadieExel1997}
B.~Abadie and R.~Exel.
\newblock Hilbert {$\mathrm{C}^*$}-bimodules over commutative
  {$\mathrm{C}^*$}-algebras and an isomorphism condition for quantum
  {H}eisenberg manifolds.
\newblock {\em Rev. Math. Phys.}, 9(4):411--423, 1997.

\bibitem{ArBkPh-Z}
D.~E. Archey, J.~Buck, and N.~C. Phillips.
\newblock Centrally large subalgebras and tracial {$\mathcal{Z}$}-absorption.
\newblock {\em Int. Math. Res. Not. IMRN}, 2018(6):1857--1877, 2018.

\bibitem{ArchPhil:SR1}
D.~E. Archey and N.~C. Phillips.
\newblock Permanence of stable rank one for centrally large subalgebras and
  crossed products by minimal homeomorphisms.
\newblock {\em J. Operator Theory}, 83(2):353--389, 2020.

\bibitem{BBSTWW:2Col}
J.~Bosa, N.~P. Brown, Y.~Sato, A.~Tikuisis, S.~White, and W.~Winter.
\newblock Covering dimension of {$\rm \mathrm{C}^*$}-algebras and 2-coloured
  classification.
\newblock {\em Mem. Amer. Math. Soc.}, 257(1233):vii+97, 2019.

\bibitem{BrownGreenRieffel1977}
L.~G. Brown, P.~Green, and M.~A. Rieffel.
\newblock Stable isomorphism and strong {M}orita equivalence of
  {$C\sp*$}-algebras.
\newblock {\em Pacific J. Math.}, 71(2):349--363, 1977.

\bibitem{BrownMingoShen}
L.~G. Brown, J.~A. Mingo, and N.-T. Shen.
\newblock Quasi-multipliers and embeddings of {H}ilbert
  {$\mathrm{C}^\ast$}-bimodules.
\newblock {\em Canad. J. Math.}, 46(6):1150--1174, 1994.

\bibitem{BrownOzawaBook}
N.~P. Brown and N.~Ozawa.
\newblock {\em {$C^*$}-algebras and finite-dimensional approximations},
  volume~88 of {\em Graduate Studies in Mathematics}.
\newblock American Mathematical Society, Providence, RI, 2008.

\bibitem{MR3845113}
N.~P. Brown, A.~Tikuisis, and A.~M. Zelenberg.
\newblock Rokhlin dimension for {$\mathrm{C}^\ast$}-correspondences.
\newblock {\em Houston J. Math.}, 44(2):613--643, 2018.

\bibitem{CETWW}
J.~Castillejos, S.~Evington, A.~Tikuisis, S.~White, and W.~Winter.
\newblock Nuclear dimension of simple {$\rm \mathrm{C}^*$}-algebras.
\newblock {\em Invent. Math.}, 224(1):245--290, 2021.

\bibitem{Dav:C*-ex}
K.~R. Davidson.
\newblock {\em {$\mathrm{C}^*$}-algebras by Example}.
\newblock Fields Institute Monographs. Amer. Math. Soc., Providence, R.I.,
  1996.

\bibitem{DPS:OrbitBreaking}
R.~J. Deeley, I.~F. Putnam, and K.~R. Strung.
\newblock Classifiable {$\mathrm{C}^*$}-algebras from minimal
  {$\mathbb{Z}$}-actions and their orbit-breaking subalgebras.
\newblock arXiv preprint mathOA/2012.10947. To appear in Math. Ann.

\bibitem{DPS:MinDynK}
R.~J. Deeley, I.~F. Putnam, and K.~R. Strung.
\newblock Minimal homeomorphisms and topological {$K$}-theory.
\newblock arXiv preprint mathDS/2012.10950. To appear in Groups, Geom. Dyn.

\bibitem{DPS:JiangSu}
R.~J. Deeley, I.~F. Putnam, and K.~R. Strung.
\newblock Constructing minimal homeomorphisms on point-like spaces and a
  dynamical presentation of the {J}iang-{S}u algebra.
\newblock {\em J. Reine Angew. Math.}, 742:241--261, 2018.

\bibitem{dosSUrz2009}
N.~M. dos Santos and R.~Urz\'{u}a-Luz.
\newblock Minimal homeomorphisms on low-dimensional tori.
\newblock {\em Ergodic Theory Dynam. Systems}, 29(5):1515--1528, 2009.

\bibitem{Ell:AF}
G.~A. Elliott.
\newblock {On the classification of inductive limits of sequences of semisimple
  finite-dimensional algebras}.
\newblock {\em J. Algebra}, 38:29--44, 1976.

\bibitem{EllGonLinNiu:ClaFinDecRan}
G.~A. Elliott, G.~Gong, H.~Lin, and Z.~Niu.
\newblock On the classification of simple amenable
  \mbox{{$\mathrm{C}^*$}-algebras} with finite decomposition rank {II}.
\newblock {arXiv preprint math.OA/1507.03437v2}, 2015.

\bibitem{EllNiu:MeanDimZero}
G.~A. Elliott and Z.~Niu.
\newblock The {$\mathrm{C}^*$}-algebra of a minimal homeomorphism of zero mean
  dimension.
\newblock {\em Duke Math. J.}, 166(18):3569--3594, 2017.

\bibitem{FarbMargalit:Primer}
B.~Farb and D.~Margalit.
\newblock {\em A primer on mapping class groups}, volume~49 of {\em Princeton
  Mathematical Series}.
\newblock Princeton University Press, Princeton, NJ, 2012.

\bibitem{FoJeSt:rsh}
M.~Forough, J.~A. Jeong, and K.~R. Strung.
\newblock Crossed products associated to homeomorphisms of mean dimension zero
  twisted by line bundles.
\newblock In preparation.

\bibitem{AdFoSt:Cartan}
M.~Forough and K.~R. Strung.
\newblock Cartan preserving isomorphisms between crossed products by {H}ilbert
  {$C(X)$}-bimodules.
\newblock In preparation.

\bibitem{FraLar:Frames}
M.~Frank and D.~R. Larson.
\newblock Frames in {H}ilbert {$\mathrm{C}^\ast$}-modules and
  {$\mathrm{C}^\ast$}-algebras.
\newblock {\em J. Operator Theory}, 48(2):273--314, 2002.

\bibitem{GioKerPhi:CRM}
T.~Giordano, D.~Kerr, N.~C. Phillips, and A.~Toms.
\newblock {\em Crossed products of {$\mathrm{C}^*$}-algebras, topological
  dynamics, and classification}.
\newblock Advanced Courses in Mathematics. CRM Barcelona.
  Birkh\"{a}user/Springer, Cham, 2018.

\bibitem{GioPutSkau:orbit}
T.~Giordano, I.~F. Putnam, and C.~F. Skau.
\newblock {Topological orbit equivalence and $\mathrm{C}^*$-crossed products}.
\newblock {\em {J. Reine Angew. Math.}}, {469}:{51--111}, {1995}.

\bibitem{GongLinNiue:ZClass2}
G.~Gong, H.~Lin, and Z.~Niu.
\newblock A classification of finite simple amenable {$\mathcal Z$}-stable
  {${\mathrm C}^\ast$}-algebras, {II}: {${\mathrm C}^\ast$}-algebras with
  rational generalized tracial rank one.
\newblock {\em C. R. Math. Acad. Sci. Soc. R. Can.}, 42(4):451--539, 2020.

\bibitem{GongLinNiue:ZClass}
G.~Gong, H.~Lin, and Z.~Niu.
\newblock A classification of finite simple amenable {$\mathcal Z$}-stable
  {$\mathrm{C}^\ast$}-algebras, {I}: {$\mathrm{C}^\ast$}-algebras with
  generalized tracial rank one.
\newblock {\em C. R. Math. Acad. Sci. Soc. R. Can.}, 42(3):63--450, 2020.

\bibitem{HirWinZac:RokDim}
I.~Hirshberg, W.~Winter, and J.~Zacharias.
\newblock Rokhlin dimension and {$\mathrm{C}^*$}-dynamics.
\newblock {\em Comm. Math. Phys.}, 335(2):637--670, 2015.

\bibitem{Hus:fibre}
D.~Husemoller.
\newblock {\em Fibre bundles}, volume~20 of {\em Graduate Texts in
  Mathematics}.
\newblock Springer-Verlag, New York, third edition, 1994.

\bibitem{JS1999}
X.~Jiang and H.~Su.
\newblock On a simple unital projectionless {$C^*$}-algebra.
\newblock {\em Amer. J. Math.}, 121(2):359--413, 1999.

\bibitem{Katsura2003}
T.~Katsura.
\newblock A construction of {$\mathrm{C}^*$}-algebras from
  {$\mathrm{C}^*$}-correspondences.
\newblock In {\em Advances in quantum dynamics ({S}outh {H}adley, {MA}, 2002)},
  volume 335 of {\em Contemp. Math.}, pages 173--182. Amer. Math. Soc.,
  Providence, RI, 2003.

\bibitem{Katsura2004}
T.~Katsura.
\newblock On {$\mathrm{C}^*$}-algebras associated with
  {$\mathrm{C}^*$}-correspondences.
\newblock {\em J. Funct. Anal.}, 217(2):366--401, 2004.

\bibitem{KirRor:pi}
E.~Kirchberg and M.~R{\o}rdam.
\newblock Non-simple purely infinite {$\mathrm{C}^\ast$}-algebras.
\newblock {\em Amer. J. Math.}, 122(3):637--666, 2000.

\bibitem{KirWinter:dr}
E.~Kirchberg and W.~Winter.
\newblock Covering dimension and quasidiagonality.
\newblock {\em Internat. J. Math.}, 15(1):63--85, 2004.

\bibitem{Kishi:RP}
A.~Kishimoto.
\newblock The {R}ohlin property for automorphisms of {UHF} algebras.
\newblock {\em J. Reine Angew. Math.}, 465:183--196, 1995.

\bibitem{Lan:modules}
E.~C. Lance.
\newblock {\em Hilbert $\mathrm{C}^*$-modules---A toolkit for operator
  algebraists}.
\newblock Number 210 in London Math. Soc.---Lecture notes Series. Cambridge
  University Press, 1995.

\bibitem{Li:ClassifiableCartan}
X.~Li.
\newblock Every classifiable simple {$\rm \mathrm{C}^*$}-algebra has a {C}artan
  subalgebra.
\newblock {\em Invent. Math.}, 219(2):653--699, 2020.

\bibitem{LinPhi:MinHom}
H.~Lin and N.~C. Phillips.
\newblock {Crossed products by minimal homeomorphisms}.
\newblock {\em J. Reine Angew. Math.}, 641:95--122, 2010.

\bibitem{QLin:Ay}
Q.~Lin.
\newblock Analytic structure of the transformation group
  {$\mathrm{C}^*$}-algebra associated with minimal dynamical systems.
\newblock preprint.

\bibitem{Miscenko1979}
A.~S. Mi\v{s}\v{c}enko.
\newblock Banach algebras, pseudodifferential operators and their applications
  to {$K$}-theory.
\newblock {\em Uspekhi Mat. Nauk}, 34(6(210)):67--79, 1979.

\bibitem{PhLarge}
N.~C. Phillips.
\newblock Large subalgebras.
\newblock {arXiv preprint math.OA/1408.5546v1 }, 2014.

\bibitem{Pimsner1997}
M.~V. Pimsner.
\newblock A class of {$\mathrm{C}^*$}-algebras generalizing both
  {C}untz-{K}rieger algebras and crossed products by {${\bf Z}$}.
\newblock In {\em Free probability theory ({W}aterloo, {ON}, 1995)}, volume~12
  of {\em Fields Inst. Commun.}, pages 189--212. Amer. Math. Soc., Providence,
  RI, 1997.

\bibitem{Putnam:MinHomCantor}
I.~F. Putnam.
\newblock {The $\mathrm{C}^*$-algebras associated with minimal homeomorphisms
  of the Cantor set}.
\newblock {\em Pacific J. Math.}, 136(2):329--353, 1989.

\bibitem{Put:K-theoryGroupoids}
I.~F. Putnam.
\newblock {On the {$K$}-theory of {$\mathrm{C}^*$}-algebras of principal
  groupoids}.
\newblock {\em {Rocky Mountain J. Math.}}, {28}({4}):{1483--1518}, {1998}.

\bibitem{RaeWil:morita}
I.~Raeburn and D.~P. Williams.
\newblock {\em Morita equivalence and continuous-trace
  {$\mathrm{C}^*$}-algebras}, volume~60 of {\em Mathematical Surveys and
  Monographs}.
\newblock American Mathematical Society, Providence, RI, 1998.

\bibitem{Rieffel1989}
M.~A. Rieffel.
\newblock Deformation quantization of {H}eisenberg manifolds.
\newblock {\em Comm. Math. Phys.}, 122(4):531--562, 1989.

\bibitem{Rordam2003}
M.~R\o~rdam.
\newblock A simple {$C^*$}-algebra with a finite and an infinite projection.
\newblock {\em Acta Math.}, 191(1):109--142, 2003.

\bibitem{Ror:uhfII}
M.~R{\o}rdam.
\newblock On the structure of simple {$\mathrm{C}^*$}-algebras tensored with a
  {UHF}-algebra. {II}.
\newblock {\em J. Funct. Anal.}, 107(2):255--269, 1992.

\bibitem{Ror:Z-absorbing}
M.~R{\o}rdam.
\newblock The stable and the real rank of {$\mathcal{Z}$}-absorbing
  {$\mathrm{C}^*$}-algebras.
\newblock {\em Internat. J. Math.}, 15(10):1065--1084, 2004.

\bibitem{Schweizer2001}
J.~Schweizer.
\newblock Dilations of {$\mathrm{C}^*$}-correspondences and the simplicity of
  {C}untz-{P}imsner algebras.
\newblock {\em J. Funct. Anal.}, 180(2):404--425, 2001.

\bibitem{Str:XxSn}
K.~R. Strung.
\newblock {On the classification of $\mathrm{C}^*$-algebras of minimal product
  systems of the Cantor set and an odd dimensional sphere}.
\newblock {\em J. Funct. Anal.}, 268(3):671--689, 2015.

\bibitem{Str:book}
K.~R. Strung.
\newblock {\em An introduction to {${\rm C}^*$}-algebras and the classification
  program}.
\newblock Advanced Courses in Mathematics. CRM Barcelona.
  Birkh\"{a}user/Springer, Cham, [2021] \copyright 2021.
\newblock Edited and with a foreword by Francesc Perera.

\bibitem{StrWin:Z-stab_min_dyn}
K.~R. Strung and W.~Winter.
\newblock {Minimal dynamics and $\mathcal{Z}$-stable classification}.
\newblock {\em Internat. J. Math.}, 22(1):1--23, 2011.

\bibitem{Swa:bundles}
R.~G. Swan.
\newblock {Vector bundles and projective modules}.
\newblock {\em Trans. Amer. Math. Soc.}, 105:264--277, 1962.

\bibitem{MR3342101}
G.~Szab\'{o}.
\newblock The {R}okhlin dimension of topological {$\mathbb{Z}^m$}-actions.
\newblock {\em Proc. Lond. Math. Soc. (3)}, 110(3):673--694, 2015.

\bibitem{TWW}
A.~Tikuisis, S.~White, and W.~Winter.
\newblock Quasidiagonality of nuclear {$\mathrm{C}^\ast$}-algebras.
\newblock {\em Ann. of Math. (2)}, 185(1):229--284, 2017.

\bibitem{Toms2005}
A.~Toms.
\newblock On the independence of {$K$}-theory and stable rank for simple
  {$C^*$}-algebras.
\newblock {\em J. Reine Angew. Math.}, 578:185--199, 2005.

\bibitem{TomsWinter:minhom}
A.~S. Toms and W.~Winter.
\newblock {Minimal {D}ynamics and {K}-{T}heoretic {R}igidity: {E}lliott's
  {C}onjecture}.
\newblock {\em Geom. Funct. Anal.}, 23(1):467--481, 2013.

\bibitem{Vasselli2003}
E.~Vasselli.
\newblock Continuous fields of {$\mathrm{C}^*$}-algebras arising from
  extensions of tensor {$\mathrm{C}^*$}-categories.
\newblock {\em J. Funct. Anal.}, 199(1):122--152, 2003.

\bibitem{Villadsen1998}
J.~Villadsen.
\newblock Simple {$C^*$}-algebras with perforation.
\newblock {\em J. Funct. Anal.}, 154(1):110--116, 1998.

\bibitem{Villadsen1999}
J.~Villadsen.
\newblock On the stable rank of simple {$C^\ast$}-algebras.
\newblock {\em J. Amer. Math. Soc.}, 12(4):1091--1102, 1999.

\bibitem{Weg:k-theory}
N.~E. Wegge-Olsen.
\newblock {\em {$K$-Theory and $\mathrm{C}^*$-algebras}}.
\newblock Oxford University Press, New York, 1993.

\bibitem{Win:cpr}
W.~Winter.
\newblock Covering dimension for nuclear {$\mathrm{C}^*$}-algebras.
\newblock {\em J. Funct. Anal.}, 199(2):535--556, 2003.

\bibitem{Win:cpr2}
W.~Winter.
\newblock Covering dimension for nuclear {$\mathrm{C}^*$}-algebras. {II}.
\newblock {\em Trans. Amer. Math. Soc.}, 361(8):4143--4167, 2009.

\bibitem{Win:Z-stabNucDim}
W.~Winter.
\newblock Nuclear dimension and {$\mathcal{Z}$}-stability of pure {$\rm
  \mathrm{C}^*$}-algebras.
\newblock {\em Invent. Math.}, 187(2):259--342, 2012.

\bibitem{Win:ICM}
W.~Winter.
\newblock Structure of nuclear {$\rm C^*$}-algebras: from quasidiagonality to
  classification and back again.
\newblock In {\em Proceedings of the {I}nternational {C}ongress of
  {M}athematicians---{R}io de {J}aneiro 2018. {V}ol. {III}. {I}nvited
  lectures}, pages 1801--1823. World Sci. Publ., Hackensack, NJ, 2018.

\bibitem{WinterZac:dimnuc}
W.~Winter and J.~Zacharias.
\newblock The nuclear dimension of {$C^\ast$}-algebras.
\newblock {\em Adv. Math.}, 224(2):461--498, 2010.

\end{thebibliography}
\end{document}